\documentclass[12pt]{amsart}
\usepackage{amsmath}
\usepackage{amssymb}
\usepackage{graphicx}
\usepackage{pinlabel}
\usepackage[all]{xy}

\usepackage{amsmath,amssymb,amsfonts,amsthm,graphicx}

\usepackage{pinlabel}
\usepackage[usenames,dvipsnames]{color}

\usepackage[usenames,dvipsnames]{color}

\newtheorem{dummy}{dummy}[section]

\newtheorem{theorem}[dummy]{Theorem}

\newtheorem{conjecture}[dummy]{Conjecture}
\newtheorem{corollary}[dummy]{Corollary}
\newtheorem{proposition}[dummy]{Proposition}

\theoremstyle{definition}
\newtheorem{definition}[dummy]{Definition}

\newtheorem{example}[dummy]{Example}
\newtheorem{remark}[dummy]{Remark}

\newtheorem{question}[dummy]{Question}
\newtheorem{construction}[dummy]{Construction}
\newtheorem{observation}[dummy]{Observation}

\newtheorem{axiom}[dummy]{Axiom}

\numberwithin{equation}{section}

\newcommand{\R}{\mathbb {R}}

\newcommand{\Z}{\mathbb {Z}}

\newcommand{\e}{\epsilon}
\newcommand{\dd}{\partial}

\newcommand{\alg}{\mathcal{A}}
\newcommand{\calg}{\mathcal{A}_{\mathit{cell}}}

\newcommand{\B}{\mathcal{B}}
\newcommand{\calC}{\mathcal{C}}

\newcommand{\wt}{\widetilde}

\newcommand{\symp}{\mathit{Symp}}
\newcommand{\aug}{\mathit{Aug}}
\newcommand{\dga}{\mathfrak{DGA}}

\newcommand{\MCF}{\mathit{MCF}}
\newcommand{\res}{\mathit{res}}

\newcommand{\comp}{\mathit{cell}}
\newcommand{\leg}{\mathfrak{Leg}}

\def\dr#1{\textcolor{Black}{#1}}

\def\wt#1{\widetilde{#1}}

\begin{document}

\title[Augmentations and immersed Lagrangian fillings]{Augmentations and immersed Lagrangian fillings}

%\thanks{DR is partially supported by grant 429536 from the Simons Foundation. YP is partially supported by the NSF Grant (DMS-1510305).}

\author{Yu Pan}
\address{Tianjin University}
\email{ypan@tju.edu.cn}

\author{Dan Rutherford}
\address{Ball State University}
\email{rutherford@bsu.edu}

\begin{abstract}
For a Legendrian link $\Lambda \subset J^1M$ with $M = \R$ or $S^1$, immersed exact Lagrangian fillings $L \subset \mbox{Symp}(J^1M) \cong T^*(\R_{>0} \times  M)$ of $\Lambda$ can be lifted to conical Legendrian fillings $\Sigma \subset J^1(\R_{>0} \times M)$ of $\Lambda$. When $\Sigma$ is embedded, using the version of functoriality for Legendrian contact homology (LCH) from \cite{PanRu1}, for each augmentation $\alpha: \mathcal{A}(\Sigma) \rightarrow \Z/2$ of the LCH algebra of $\Sigma$, there is an induced augmentation $\epsilon_{(\Sigma,\alpha)}: \mathcal{A}(\Lambda) \rightarrow \Z/2$.   With $\Sigma$ fixed, the set of homotopy classes of all such induced augmentations, $I_\Sigma \subset \mathit{Aug}(\Lambda)/{\sim}$, is a Legendrian isotopy invariant of $\Sigma$.  
We establish methods to compute $I_\Sigma$ based on the correspondence between Morse complex families and augmentations.  This includes developing a functoriality for the cellular DGA from \cite{RuSu1} with respect to Legendrian cobordisms, and proving its equivalence to the functoriality for LCH.  For arbitrary $n \geq 1$, we give examples of Legendrian torus knots with $2n$ distinct conical Legendrian fillings distinguished by their induced augmentation sets.  
 We prove that when $\rho \neq 1$ and $\Lambda \subset J^1\R$ {\it every} $\rho$-graded augmentation of $\Lambda$ can be induced in this manner by an immersed Lagrangian filling.  Alternatively, this is viewed as a computation of cobordism classes for an appropriate notion of $\rho$-graded augmented Legendrian cobordism. 

\end{abstract}

\maketitle

{\small \tableofcontents}

\section{Introduction}

A fundamental holomorphic curve invariant of a Legendrian submanifold, $\Lambda$, is the Legendrian contact homology (LCH) dg-algebra (DGA), denoted $\mathcal{A}(\Lambda)$.  As part of the symplectic field theory package, the LCH algebra is functorial for an appropriate class of cobordisms.  In this article we consider $1$-dimensional Legendrian links in the $1$-jet spaces, $J^1M$ with $M = \R$ or $S^1$, and exact Lagrangian cobordisms in the symplectization, $\mathit{Symp}(J^1M) = \R \times J^1M$; throughout, our coefficient field is $\Z/2$.  For $\Lambda_-, \Lambda_+ \subset J^1M$, such a cobordism, $L: \Lambda_- \rightarrow \Lambda_+$, cylindrical over $\Lambda_-$ and $\Lambda_+$ at the negative and positive ends of $\mathit{Symp}(J^1M)$, equipped with a $\Z/\rho$-valued Maslov potential induces a $\Z/\rho$-graded DGA map $f_{L} :\alg(\Lambda_+) \rightarrow \alg(\Lambda_-)$, cf. \cite{E2, EHK}.  In particular, when $L$ is an exact Lagrangian filling, i.e. a cobordism $L: \emptyset \rightarrow \Lambda$, the induced map
\[
\epsilon_L: \mathcal{A}(\Lambda) \rightarrow \Z/2
\]
is a {\bf $\rho$-graded augmentation} which by definition is a unital ring homomorphism that satisfies $\epsilon_L \circ \partial =0$ and preserves a $\Z/\rho$-grading on $\mathcal{A}(\Lambda)$.

A natural question is:

\begin{question}
Which augmentations come from exact Lagrangian fillings?
\end{question} 

While orientable exact Lagrangian fillings have been constructed for several classes of Legendrian knots
\cite{EHK, HS, STWZ, Tagami}, there are many augmentations that cannot be induced by any orientable filling as obstructions to such fillings arise from the Thurston-Bennequin number of $\Lambda$ and from the linearized homology of the augmentation;  
 see \cite{Chan, E2, DR}.  The main result of this article shows that if one extends the setting to allow immersed cobordisms with double points then the algebra more closely matches the geometry.  Indeed we prove that when $\rho$ is even, {\it every $\rho$-graded augmentation can be induced by an orientable immersed exact Lagrangian filling.}

An extension of the functoriality for LCH to immersed Lagrangian cobordisms is implemented in \cite{PanRu1} by working with a class of Legendrian cobordisms as follows.  Applying a symplectomorphism $\mathit{Symp}(J^1M) \cong T^*(\R_{>0} \times M)$ an exact immersed Lagrangian cobordism, $L \subset \mathit{Symp}(J^1M)$, can be lifted to a Legendrian $\Sigma \subset J^1(\R_{>0} \times M)$ with the cylindrical ends of $L$ translating to conical ends for $\Sigma$ and double points of $L$ becoming Reeb chords of $\Sigma$.  See Section \ref{sec:2-4}.  When $\Sigma$ is embedded and equipped with a $\Z/\rho$-valued Maslov potential, such a {\bf conical Legendrian cobordism}, $\Sigma: \Lambda_- \rightarrow \Lambda_+$, induces a diagram of $\rho$-graded DGA maps
\begin{equation} \label{eq:intro1}
\mathcal{A}(\Lambda_+) \stackrel{f_\Sigma}{\rightarrow} \mathcal{A}(\Sigma) \stackrel{i_\Sigma}{\hookleftarrow} \mathcal{A}(\Lambda_-)
\end{equation}
where $\mathcal{A}(\Sigma)$ is generated by $\mathcal{A}(\Lambda_-)$ and the Reeb chords of $\Sigma$.  Diagrams of the above form are referred to in
 \cite{PanRu1} as {\bf immersed DGA maps}.  There, a notion of homotopy for immersed DGA maps is introduced, and the homotopy type of the immersed map (\ref{eq:intro1}) is shown to be an invariant of the conical Legendrian isotopy type of $\Sigma$.  When $\Sigma: \emptyset \rightarrow \Lambda$ is a conical Legendrian filling equipped with a choice of $\rho$-graded augmentation, $\alpha:\mathcal{A}(\Sigma) \rightarrow \Z/2$, we can then define an {\bf induced augmentation} $\epsilon_{(\Sigma, \alpha)}: \mathcal{A}(\Lambda) \rightarrow \Z/2$ as the composition  $\epsilon_{(\Sigma, \alpha)} = \alpha \circ f_{\Sigma}$.  This generalizes the construction of induced augmentations from embedded Lagrangian fillings.  

We can now state our first main result, where in the following we write $\epsilon \simeq \epsilon'$ to indicate that two augmentations are DGA homotopic.

\begin{theorem} \label{thm:main}
Let $\Lambda \subset J^1\R$ have the $\Z/\rho$-valued Maslov potential $\mu$ where $\rho \geq 0$, and  let $\epsilon: \mathcal{A}(\Lambda) \rightarrow \Z/2$ be any $\rho$-graded augmentation.  
\begin{enumerate}
\item If $\rho \neq 1$, there exists a conical Legendrian filling $\Sigma$  of $\Lambda$ with $\Z/\rho$-valued Maslov potential extending $\mu$ together with a $\rho$-graded augmentation $\alpha:\alg(\Sigma) \rightarrow \Z/2$ such that $\epsilon \simeq \epsilon_{(\Sigma, \alpha)}$.  Moreover, if $\rho$ is even, then $\Sigma$ is orientable.
\item If $\rho =1$, then there exists a conical Legendrian cobordism $\Sigma: U \rightarrow \Lambda$ where $U$ is the standard Legendrian unknot with $\mathit{tb}(U) = -1$ together with a $1$-graded augmentation $\alpha:\alg(\Sigma) \rightarrow \Z/2$ such that $\epsilon \simeq \alpha \circ f_\Sigma$.
\end{enumerate}
\end{theorem}

The algebra of the standard Legendrian unknot, $U$, is generated by a single Reeb chord $b$ of degree $1$.  In the case where $\rho =1$,  if the restriction of $\alpha$ to $\mathcal{A}(U) \subset \mathcal{A}(\Sigma)$ sends $b$ to $0$, then by concatenating with the standard filling of $U$ we see that $\epsilon$ can be induced by a pair $(\Sigma', \alpha')$ where $\Sigma'$ is a conical Legendrian filling.  However, in the case that $\alpha$ restricts to the augmentation of $\mathcal{A}(U)$ that maps $b$ to $1$, we are not sure whether $\epsilon$ can be induced by a Legendrian filling.  In fact, we conjecture that this is not  possible.

\begin{conjecture} \label{co:1}
There is no conical Legendrian filling $\Sigma$ of the Legendrian unknot $U$ with $1$-graded augmentation $\alpha:\mathcal{A}(\Sigma) \rightarrow \Z/2$ such that the induced augmentation $\epsilon_{(\Sigma, \alpha)}$ maps the unique Reeb chord $b$ to $1$.
\end{conjecture}

For the case of embedded Lagrangians, induced augmentations provide an effective means of distinguishing Lagrangian fillings of a given Legendrian knot.  The results of \cite{PanRu1} show that these induced augmentations are actually invariants of the associated conical Legendrian fillings.  (Note that conical Legendrian isotopy appears to be a much less restrictive notion of equivalence than isotopy of the corresponding exact Lagrangians, since during the course of a Legendrian isotopy any number of double points can be added and removed from the Lagrangian projections.)   

A key new feature in the immersed case is that a single Legendrian filling, $\Sigma$, can induce more than one augmentation of $\Lambda$ due to the dependence of the construction on the choice of augmentation $\alpha$ of $\mathcal{A}(\Sigma)$.  As a result, to obtain an invariant of $\Sigma$ we should consider the {\it set} of (DGA homotopy classes of) induced augmentations $I_{\Sigma} \subset \mathit{Aug}_\rho(\Lambda)/{\sim}$.  After developing methods to compute this invariant {\bf induced augmentation set}, we demonstrate that $I_\Sigma$ can be effective for distinguishing conical Legendrian fillings in the following theorem.

\begin{theorem} \label{thm:1-4}
For each $n \geq 1$, there exists $2n$ distinct conical Legendrian fillings, $\Sigma_1, \ldots, \Sigma_{2n}$ of the max-$\mathit{tb}$ Legendrian torus knot $T(2, 2n+1)$ such that 
\begin{itemize}
\item[(i)] the $\Sigma_i$ are all orientable with genus $n-1$ and have $\Z$-valued Maslov potentials,
\item[(ii)] each $\Sigma_i$ has a single Reeb chord of degree $0$,
\item[(iii)] and the induced augmentation sets satisfy $I_{\Sigma_i} \neq I_{\Sigma_j}$ when $i \neq j$.
\end{itemize}
\end{theorem}

Note that the Lagrangian projections $L_1, \ldots, L_n$ are immersed Lagrangian fillings of $T(2, 2n+1)$ with a single double point.   Moreover, the dg-algebras $\mathcal{A}(\Sigma_i)$ are all isomorphic to one another, and thus do not distinguish the $\Sigma_i$ on their own.

\subsection{Cobordism classes of augmented Legendrians}  
In the context of relative symplectic field theory (SFT) type invariants \cite{EGH}, it is natural to consider Lagrangian cobordisms between Legendrian submanifolds in the symplectization of a contact manifold (or more generally in symplectic cobordisms with concave and convex  ends modeled on the negative and positive ends of symplectizations).  
In the following discussion, we refer to exact Lagrangian cobordisms in $\mathit{Symp}(J^1M)$  with cylindrical ends as {\bf SFT-cobordisms}.  The relation defined by embedded SFT-cobordisms, or even SFT-concordances, is not symmetric \cite{Chan2, Pan1}, and hence does not define an equivalence relation on Legendrian links in $J^1M$.  In fact, it is a major open question \cite{CornNS}
 in the field whether or not SFT-cobordisms define a partial order on the set of Legendrian isotopy classes in $J^1\R$.  

The lack of a readily available symmetry is visible after lifting an  
SFT-cobordism to a conical Legendrian cobordism $\Sigma \subset J^1(\R_{>0} \times M) = T^*(\R_{>0} \times M) \times \R_z$ in the difference in behavior at the two ends of $\Sigma$:  as $\Sigma$ approaches $0$ (resp. $+\infty$) along the $\R_{>0}$ factor the cotangent coordinates and the $z$-coordinates of $\Sigma$ appear as those of $\Lambda_-$ (resp. $\Lambda_+$) but shrinking (resp. expanding).  However, when one allows for SFT-cobordisms to be immersed it becomes possible to reverse their direction as an expanding end can be modified to a shrinking end (and vice-versa) at the expense of creating some additional Reeb chords.  Thus, the relation of conical Legendrian cobordism is equivalent to another standard notion of Legendrian cobordism in $1$-jet spaces introduced by Arnold, cf. \cite{Arnold1, Arnold2, ArnoldWave}.   In Arnold's definition, a Legendrian cobordism between two Legendrians $\Lambda_0, \Lambda_1 \subset J^1M$ is a compact Legendrian,
$\Sigma \subset J^1([0,1]\times M)$, whose restriction
to $J^1(\{i\} \times M)$ is $\Lambda_i$.  Seminal results in this theory of Legendrian (and also Lagrangian) cobordisms were achieved by Audin, Eliashberg, and Vassiliev in the 1980's, including homotopy theoretic characterizations of various Lagrange and Legendre cobordism groups.  See \cite{E84, Audin87, Vass}. 

For an alternate perspective on our main theorem, we can incorporate augmentations into an Arnold-type cobordism theory.  Define a {\bf $\rho$-graded augmented Legendrian} to be a Legendrian submanifold $\Lambda \subset J^1M$ equipped with a $\Z/\rho$-valued Maslov potential and a $\rho$-graded augmentation $\epsilon: \mathcal{A}(\Lambda) \rightarrow \Z/2$.  For compact Legendrian cobordisms $\Sigma \subset J^1([0,1]\times M)$ satisfying a suitable Morse minimum boundary condition (see Section \ref{sec:minimum}), the LCH dg-algebra of $\Sigma$ is defined as in \cite{EK}, and contains $\mathcal{A}(\Lambda_0)$ and $\mathcal{A}(\Lambda_1)$ as sub-dg-algebras.  We then declare two $\rho$-graded augmentation Legendrians, $(\Lambda_0, \epsilon_0)$ and $(\Lambda_1, \epsilon_1)$, to be cobordant if there exists a pair $(\Sigma, \alpha)$ consisting of a Legendrian cobordism equipped with a $\rho$-graded augmentation of $\mathcal{A}(\Sigma)$ whose restriction to $\mathcal{A}(\Lambda_i)$ is homotopic to $\epsilon_i$  for $i =0,1$.  As a variant on Theorem \ref{thm:main} we obtain the following.

\begin{theorem} \label{thm:main2} Let $\rho \geq 0$ be a non-negative integer, and let $(\Lambda, \epsilon)$ be a $\rho$-graded augmented Legendrian in $J^1\R$.

\begin{enumerate}
\item If $\rho \neq 1$, then $(\Lambda, \epsilon)$ is cobordant to $\emptyset$.

\item If $\rho =1$, then $(\Lambda, \epsilon)$ is either cobordant to $\emptyset$ or $(U, \epsilon_1)$ where $U$ is the standard Legendrian unknot and $\epsilon_1: \mathcal{A}(U) \rightarrow \Z/2$ satisfies $\epsilon_1(b) = 1$ on the unique Reeb chord of $U$.
\end{enumerate}
\end{theorem}
As in Conjecture \ref{co:1}, we expect that $(U, \epsilon_1)$ is not null-cobordant, so that there should be exactly one cobordism class when $\rho \neq 1$ and exactly two cobordism classes when $\rho = 1$.

In a further article \cite{PanRu3}, we give a complete classification of cobordism classes of augmented Legendrians in $J^1S^1$.  Interestingly, in $J^1S^1$ it is often the case that cobordant Legendrians may become non-cobordant once they are equipped with augmentations.  
	We note that for long Legendrian knots in $\R^3$ a concordance group of similar spirit, incorporating quadratic at infinity generating families as additional equipment rather than augmentations, is considered recently in the work of Limouzineau, \cite{Limo}.

\subsection{Methods and outline}  
Our approach is based on working with certain algebraic/combinatorial structures equivalent to augmentations called Morse complex families (MCFs) that were introduced by Pushkar (unpublished) and studied by Henry in \cite{Henry}.  
MCFs can be viewed as combinatorial approximations to generating families of functions, as they consist of formal Morse complexes and handleslide data assigned to a Legendrian submanifold.  For $1$-dimensional Legendrians, the work of Henry \cite{Henry,HenryRu2} establishes a bijection between equivalence classes of MCFs and homotopy classes of augmentations.  More recently, a correspondence between MCFs and augmentations for $2$-dimensional Legendrians was obtained in \cite{RuSu3} using  
the cellular DGA which is a cellular model for LCH developed in the work of the second author and Sullivan \cite{RuSu1,RuSu2,RuSu25}. 
The strategy of the present article is to extend these methods to the case of Legendrian cobordisms, and then apply Morse complex families to compute induced augmentation sets.  In more detail, to prove Theorem \ref{thm:main} we accomplish the following tasks (1)-(3).  Note that (1) and (2) may be of some independent interest.

\medskip

\begin{enumerate}
\item[(1)]  We develop a functoriality of the cellular DGA for Legendrian cobordisms, and we extend the equivalence with LCH from \cite{RuSu1,RuSu2,RuSu25} to an isomorphism between the immersed LCH functor from \cite{PanRu1} and its cellular analog; \dr{see  Proposition \ref{prop:Fisomorph}.}
\end{enumerate}

\medskip

Although it seems crucial for this construction to be carried out in the broader setting of Legendrian cobordisms and immersed DGA maps, this allows the cellular DGA to be applied just as well for working with embedded SFT-cobordisms.

\medskip

\begin{enumerate}
\item[(2)]  We extend the correspondence between MCFs and augmentations to the case of $2$-dimensional Legendrian cobordisms.  In particular, we give a method based on MCFs for computing the induced augmentation sets of Legendrian cobordisms and fillings; \dr{see Proposition \ref{prop:MCFcomp}.}  
\end{enumerate} 

\medskip

In Section \ref{sec:example}, we illustrate this method with examples and use it to prove Theorem \ref{thm:1-4}.  Making use of (1) and (2),  Theorem \ref{thm:main} is then reduced to the following construction.  

\medskip

\begin{enumerate}
\item[(3)] Given a $1$-dimensional Legendrian knot, $\Lambda \subset J^1\R$, with a  $\rho$-graded MCF $\mathcal{C}_\Lambda$ with $\rho \neq 1$ (resp. $\rho=1$), we produce a Legendrian filling  $\Sigma$ of $\Lambda$ (resp. a Legendrian cobordism $\Sigma: U\rightarrow \Lambda$) with a $2$-dimensional $\rho$-graded MCF $\mathcal{C}_{\Sigma}$ extending $\mathcal{C}_\Lambda$.  
\end{enumerate}

\medskip

In Section \ref{sec:aug} we provide such a construction via an induction  on the complexity of the front projection of $\Lambda$.  \dr{The resulting cobordism is built up out of elementary building blocks that include those used in constructing decomposable, embedded Lagrangian cobordisms (Legendrian isotopies, pinch and unknot moves) as well as clasp moves that produce double points in the Lagrangian projection.  
%At the inductive step, the next portion of the cobordism is selected by examining the form of the Legendrian and MCF at the current negative boundary 
Care is taken to ensure that at each step the MCF can be extended over the cobordism, and these considerations are simplified by}  
%In addition, to the moves (Legendrian isotopies, pinch moves, and used for constructing decomposable, embedded Legendrian
%This construction is simplified by 
making use of a standard form for MCFs (the ``$SR$-form'') introduced by Henry.   

\begin{remark}
\dr{For embedded, decomposable SFT cobordisms, explicit computations of induced DGA maps and in particular induced augmentations are possible via the results of \cite{EHK}, and this approach has been commonly used in the literature.  A possible alternate route to proving Theorem \ref{thm:main} could be to extend the arsenal from \cite{EHK} to include computations of the immersed DGA map associated to the Clasp Move and attempt a similar inductive construction in this context.  
% associated to those elementary cobordisms from Section \ref{sec:Construct} that cannot be realized as embedded SFT cobordisms.  These include the Clasp Move and single directions of the Pinch and Unknot move; the Clasp move is the most crucial in the current inductive that establishes Theorem \ref{thm:main}.  
Such a computation would be useful more generally for handling immersed Lagrangian cobordisms within the framework of \cite{EHK}.  We leave this as an interesting direction for possible future research, while noting that the approach taken in the current article and \cite{PanRu1} is  independent of \cite{EHK} in both the foundations (based on  \cite{EESLCH, Ekh}) and computational methods (based on \cite{RuSu1,RuSu3}).}   
\end{remark}

The remainder of the paper is organized as follows.  In Section \ref{sec:background} we collect background from \cite{PanRu1} on immersed DGA maps, conical Legendrian cobordisms, and the immersed LCH functor, $F$.  
In addition, we discuss Morse minimum cobordisms and show that Theorem \ref{thm:main2} follows from Theorem \ref{thm:main}.  In Section \ref{sec:3}, we establish some basic properties of induced augmentation sets associated to immersed DGA maps including a composition formula and Legendrian isotopy invariance.
In Section \ref{sec:4}, we review the cellular DGA and define a cellular LCH functor, $F_\comp$.  We state an isomorphism between the cellular LCH functor and the immersed LCH functor.  The proof of the isomorphism is an extension of \cite{RuSu2,RuSu25} \dr{that is sketched here with a more detailed argument appearing in the preprint version of this article \cite{PanRu2Arxiv}.}
  In Section \ref{sec:MCF}, %we consider augmentations of the cellular DGA, and 
\dr{after reviewing the definition of Morse complex families, 
we extend the correspondence between MCFs and augmentations of the cellular DGA to the cobordism case and establish an MCF characterization of the induced augmentation set.}  
In Section \ref{sec:example}, we make use of $A$-form MCFs to compute the induced augmentation sets for several examples of Legendrian fillings, and we prove Theorem \ref{thm:1-4}.  In particular, we give explicit examples of augmentations that can be induced by Legendrian fillings but cannot be induced by any embedded fillings.  Finally, the article concludes with Section \ref{sec:aug} that, after recalling Henry's $SR$-form MCFs, provides the proof of Theorem \ref{thm:main}.  \dr{The reader that is content to blackbox the characterization of the induced augmentation set in terms of MCFs can more quickly arrive at the proof of Theorem \ref{thm:main} by omitting Sections \ref{sec:4}, \ref{sec:Aform}, and \ref{sec:example}.}

\begin{remark}  In the present article we restrict considerations to $\Z/2$ valued augmentations.  This is mainly because the isomorphism between the cellular DGA and LCH DGA proven in \cite{RuSu1, RuSu2, RuSu25} is only established with $\Z/2$ coefficients.  We expect that a version of this isomorphism should extend to a more general coefficient ring, eg. to the setting of the LCH DGA with fully non-commutative $\Z[\pi_1(\Lambda)]$ coefficients which is appropriate for considering augmentations to general rings.  
%Although a formulation for the cellular DGA with such coefficients can be made, see eg. \cite{RuSu4} for a definition in all dimensions with $\Z$ coefficients, the current difficulty in extending the isomorphism with LCH to this setting is to be able to explicitly evaluate signs for the gradient flow trees used in computing the LCH differential;   see \cite{Karl} for progress in this direction.  
Assuming this point, the proof of Theorem \ref{thm:main} should be extendable to apply to augmentations with values in an arbitrary {\it field}; see also the discussion on coefficients in \cite[Section 3.3]{PanRu3}.  The proof  %of Theorem \ref{thm:main} 
breaks down in the case of augmentations to a general ring, as the ability to take multiplicative inverses of handleslide coefficients is crucial in extending the arguments in Section \ref{sec:aug}.  
%In fact, Theorem \ref{thm:main} may not be true at such a level of generality, eg. because of the existence of higher dimensional representations mapping the generator of $H_1(\Lambda)$ to elements that are not multiples of the identity matrix, cf. \cite{NgRu}.
%In our upcoming work \cite{PanRu3} that covers the wider class of Legendrian links in $J^1S^1$, we carry out the part of the arguments involving MCFs over an arbitrary field, $\mathbb{F}$.  Moreover, in the case of $\mbox{char} \, \mathbb{F} =2$, the connection with augmentations of the LCH DGA only requires the isomorphism with the cellular DGA be extended to $\Z/2[H_1(\Lambda)]$.  In this case, the extension is understood, and will hopefully appear in upcoming joint work of the second author and M. Sullivan. 
\end{remark}

\subsection{Acknowledgements}  DR is partially supported by grant 429536 from the Simons Foundation. YP is partially supported by the NSF Grant (DMS-1510305).

\section{Immersed DGA maps and cobordisms}  \label{sec:background}

In this section we review the algebraic and geometric setup from \cite{PanRu1}, and discuss two classes of Legendrian cobordisms.   Section \ref{sec:2-1} recalls a class of DGAs relevant for Legendrian contact homology, and records in Proposition \ref{prop:quotient} a method for producing stable tame isomorphisms.  In Section \ref{sec:2-2} we review the concept of immersed DGA maps 
and immersed homotopy then discuss two equivalent characterizations of composition. In addition, we recall the category, $\mathfrak{DGA}^\rho_{\mathit{im}}$, of $\Z/\rho$-graded DGAs with immersed maps constructed in \cite{PanRu1}, and its connection with the ordinary homotopy category of DGAs, $\mathfrak{DGA}^\rho_{\mathit{hom}}$; see Proposition \ref{prop:IFunctor}.  

Next, we turn to the geometric side with a brief review of Legendrians and LCH in the setting of $1$-jet spaces in \ref{sec:2-3}.  Section \ref{sec:2-4} discusses conical Legendrian cobordisms and, after observing their connection with immersed Lagrangian cobordisms, recalls the construction of immersed DGA maps from conical Legendrian cobordisms as in \cite{PanRu1}.  The construction is nicely encoded in a functor, $F$, called the {\it immersed LCH functor} that plays a central role in this article.  Finally, Section \ref{sec:minimum} discusses Morse minimum Legendrian cobordisms, connects them with conical Legendrian cobordisms, and provides a proposition that allows for Morse minimum cobordisms to be used for computations with the immersed LCH functor.  The section concludes by defining the equivalence relation of $\rho$-graded augmented Legendrian cobordism and showing that Theorem \ref{thm:main2} follows from Theorem \ref{thm:main}.

\subsection{Differential graded algebras}  \label{sec:2-1} We work in the algebraic context of \cite[Sections 2 and 3]{PanRu1} that we will now briefly review.  In this article, {\bf differential graded algebras} (abbr. {\bf DGAs}), $(\mathcal{A}, \partial)$, are defined over $\Z/2$ and are graded by $\Z/\rho$ for some fixed $\rho \geq 0$.  The differential, $\partial$, has degree $-1$ (mod $\rho$).  We restrict attention to {\bf based DGAs} where the $\Z/2$-algebra $\mathcal{A} = \Z/2\langle x_1, \ldots, x_n\rangle$ is free associative (non-commutative) with identity element and  is equipped with a choice of (finite) free generating set $\{x_1, \ldots, x_n\}$; generators have degrees $|x_i| \in \Z/\rho$.  Subalgebras (resp. $2$-sided ideals) generated by a subset $Y \subset \mathcal{A}$ are notated as $\Z/2\langle Y \rangle$ (resp. $\mathcal{I}(Y)$).  Such a DGA $(\mathcal{A}, \partial)$ is {\bf triangular} if (with respect to some ordering of the generating set) we have $\partial x_i \in \Z/2\langle x_1, \ldots, x_{i-1} \rangle$ for all $1\leq i \leq n$.  The coproduct of based DGAs, $\mathcal{A} * \mathcal{B}$, is the based DGA whose generating set is the union of the generating sets of $\mathcal{A}$ and $\mathcal{B}$ and whose differential extends the differentials of $\mathcal{A}$ and $\mathcal{B}$.  A {\bf stabilization} of a DGA $\mathcal{A}$ is a DGA of the form $\mathcal{A} * S$ where $S$ has generating set of the form $\{a_1, b_1, \ldots, a_r, b_r\}$ with differential $\partial a_i = b_i$, $1 \leq i \leq r$.  {\bf DGA morphisms} are unital, algebra homomorphisms that preserve the $\Z/\rho$-grading and commute with differentials.  A {\bf stable tame isomorphism} from $\mathcal{A}$ to $\mathcal{B}$ is a DGA isomorphism $\varphi: \mathcal{A}*S \rightarrow \mathcal{B}*S'$ between stabilizations of $\mathcal{A}$ and $\mathcal{B}$ that is tame, i.e. it is a composition of isomorphisms that have a certain form on generators; see \cite[Section 2.2]{PanRu1}.   We say DGAs are {\bf equivalent} if they are stable tame isomorphic.  Two DGA maps $f,g: \mathcal{A} \rightarrow \mathcal{B}$ are {\bf DGA homotopic} if they satisfy
\[
f-g = \partial_\mathcal{B} \circ K + K \circ \partial_\mathcal{A}
\]
for some $(f,g)$-derivation, $K:\mathcal{A} \rightarrow \mathcal{B}$, where an {\bf $(f,g)$-derivation} is a degree $1$ (mod $\rho$) linear map satisfying $K(xy) = K(x) g(y) + (-1)^{|x|}f(x) K(y)$.  When $\varphi: \mathcal{A}*S \rightarrow \mathcal{B}*S'$ is a stable isomorphism there is an {\bf associated DGA homotopy equivalence} $h:\mathcal{A} \rightarrow \mathcal{B}$ given by $h = \pi' \circ h \circ \iota$ where $\iota:\mathcal{A} \rightarrow \mathcal{A}*S$ and $\pi':\mathcal{B} *S' \rightarrow \mathcal{B}$ are inclusion and projection.  

The following proposition is contained in \cite[Propositions 2.3 and 2.5]{PanRu1}.

\begin{proposition}  \label{prop:quotient}
Let $(\mathcal{A},\partial)$ be a DGA that is triangular with respect to the ordered generating set $\{x_1, \ldots, x_n\}$, and suppose that $\partial x_i = x_j + w$ where $w \in \Z/2\langle x_1, \ldots, x_{j-1} \rangle$.  
\begin{enumerate}
\item  Then, $\mathcal{A}/I$ where $I =\mathcal{I}(x_i, \partial x_i)$ is a triangular DGA with respect to the generating set $\{\overline{x_1}, \ldots, \widehat{j}, \ldots, \widehat{i}, \ldots, \overline{x_n} \}$, and there is a stable tame isomorphism 
\[
\varphi: \mathcal{A}/I * S \rightarrow \mathcal{A}
\]
with $S = \Z/2\langle y,z\rangle$.  Moreover,  $\varphi = g*h$ where $h(y) = x_i$, $h(z)=\partial x_i$, and $g:\mathcal{A}/I \rightarrow \mathcal{A}$ is DGA homotopy inverse to the quotient map $p:\mathcal{A} \rightarrow \mathcal{A}/I$ via
\[
g \circ p - \mathit{id}_{\mathcal{A}} = \partial \circ H + H \circ \partial
\]  
where $H:\mathcal{A} \rightarrow \mathcal{A}$ is the $(g\circ p,\mathit{id}_{\mathcal{A}})$-derivation satisfying $H(x_j) = x_i$ and $H(x_l) =0$ for $l\neq j$.
\item  If $\mathcal{B} \subset \mathcal{A}$ is a based sub-DGA generated by $Y \subset \{x_1, \ldots, x_n\}$ and $x_i, x_j \notin Y$, then $\varphi(\overline{x_k}) = x_k$ for all $x_k \in Y$.  
\item For $\varphi^{-1}$, the associated DGA homotopy equivalence $\mathcal{A} \stackrel{\varphi^{-1}}{\rightarrow} \mathcal{A}/I*S \stackrel{\pi}{\rightarrow} \mathcal{A}/I$ has $\pi \circ \varphi^{-1} =p$ where $p:\mathcal{A} \rightarrow \mathcal{A}/I$ is the quotient map.
\end{enumerate}
\end{proposition}

\begin{definition}
A {\bf $\rho$-graded augmentation} to $\Z/2$ of a $\Z/\rho$-graded DGA, $(\mathcal{A}, \partial)$, is a DGA morphism $\epsilon:(\mathcal{A}, \partial) \rightarrow (\Z/2, 0)$ where the grading on $\Z/2$ is concentrated in degree $0$ mod $\rho$.  I.e., $\epsilon$ is a $\Z/2$-algebra homomorphism from $\mathcal{A}$ to $\Z/2$ that satisfies
\[
\epsilon \circ \partial = 0, \quad \epsilon(1) = 1, 
\] 
and $\epsilon(x) \neq 0$ implies that $|x| =0 \in \Z/\rho$.
\end{definition}

\subsection{Immersed DGA maps} \label{sec:2-2}
The
main construction of \cite{PanRu1} extends the functoriality of the Legendrian contact homology DGA to the case where the domain category consists of Legendrians with a class of {\it immersed} Lagrangian cobordisms.  To accomplish this it is natural to also enlarge the class of morphisms in the target category of DGAs, and this is done by introducing {\it immersed} DGA maps with a suitable notion of homotopy.  Here, we recall these notions which will be central in the remainder of the article.

\begin{definition}  Let $(\mathcal{A}_1, \partial_1)$ and $(\mathcal{A}_2, \partial_2)$ be triangular DGAs.  An {\bf immersed DGA map}, $M$, from $\mathcal{A}_1$ to $\alg_2$ is a diagram of DGA maps
\[
M = \left( \mathcal{A}_1 \stackrel{f}{\rightarrow} \mathcal{B} \stackrel{i}{\hookleftarrow} \mathcal{A}_2 \right)
\]
where $\mathcal{B}$ is a triangular DGA and $f$ and $i$ are DGA maps such that $i$ is an inclusion induced by an inclusion of the generating set of $\mathcal{A}_2$ into the generating set of $\mathcal{B}$.  

Two immersed DGA maps $M =\left( \mathcal{A}_1 \stackrel{f}{\rightarrow} \mathcal{B} \stackrel{i}{\hookleftarrow} \mathcal{A}_2 \right)$ and $M' = \left( \mathcal{A}_1 \stackrel{f'}{\rightarrow} \mathcal{B}' \stackrel{i'}{\hookleftarrow} \mathcal{A}_2\right)$ are {\bf immersed homotopic} if there exists a stable tame isomorphism $\varphi: \mathcal{B} *S \rightarrow \mathcal{B} * S'$ such that
\begin{itemize}
\item $\varphi \circ i = i'$ and
\item $\varphi \circ f \simeq f'$  (DGA homotopy).
\end{itemize}
That is, the left (resp. right) half of the diagram
\[
\xymatrix{ & \B *S  \ar[dd]^{\varphi} & \\ \alg_1 \ar[ru]^{f}  \ar[rd]^{f'}  & & \alg_2 \ar[lu]_{i} \ar[ld]_{i'}  \\ & \B'*S' & } 
\]
is commutative up to DGA homotopy (resp.  fully commutative).
\end{definition}

\begin{definition}
\label{def:immersedcomp}
The {\bf composition} $M_2\circ M_1$ of two immersed DGA maps $M_k=\big(\alg_k \stackrel{f_k}{\rightarrow} \mathcal{B}_k \stackrel{i_k}{\hookleftarrow} \alg_{k+1} \big),
$ for ${k=1,2}$
 is given by $M_2\circ M_1 = \big(\alg_1 \stackrel{f}{\rightarrow} \mathcal{B} \stackrel{i}{\hookleftarrow} \alg_3 \big),$ where $\mathcal{B} = \mathcal{B}_1*\mathcal{B}_2/\mathcal{I}( \{i_1(x) - f_2(x) \,|\, x \in \mathcal{A}_2\})$ is the categorical push out of $i_1$ and $f_2$ and $f= p_1\circ f_1$, $i= p_2\circ i_2$ with $p_k:\mathcal{B}_k \rightarrow \mathcal{B}$ the projection maps.  The algebra $\mathcal{B}$ is triangular with respect to the generating set obtained as the union of the generators of $\mathcal{B}_1$ and $\mathcal{B}_2$ with the generators of $\mathcal{A}_2$ removed.   This is summarized by the diagram:
$$
  \xymatrix{ & & \B  &&  \\ & \B_1 \ar@{.>}[ru]^{p_1}  & & \B_2 \ar@{.>}[lu]_{p_2}&  &\\
  \alg_1 \ar[ru]^{f_1}&& \alg_2 \ar[lu]_{i_1} \ar[ru]^{f_2} & &\alg_3 \ar[lu]_{i_2}  \\ } 
 $$
\end{definition}

There is an alternate characterization of composition of immersed maps, up to immersed homotopy, as follows.  Again, let $M_k=\big(\alg_k \stackrel{f_k}{\rightarrow} \mathcal{B}_k \stackrel{i_k}{\hookleftarrow} \alg_{k+1} \big)$, $k=1,2$ be immersed DGA maps.  Let $a_1, \ldots, a_m$ denote the generators of $\mathcal{A}_2$, and let $\widehat{\mathcal{A}}_2$ have free generating set $\widehat{a}_1, \ldots, \widehat{a}_m$ with degree shift $|\widehat{a_i}| = |a_i|+1$.  
\begin{proposition}[Proposition 3.6 of \cite{PanRu1}] \label{prop:CompositionAlt} Suppose
%\footnote{\dr{Have omitted the existence statement.  Is it used later?  Is only the standard form differential on $\mathcal{D}$ used later?  May adjust statement accordingly.}} 
that $(\mathcal{D}, \partial)$ is a triangular DGA such that 
\begin{itemize}
\item $\mathcal{D} = \mathcal{B}_1 * \widehat{\mathcal{A}}_2 * \mathcal{B}_2$,
\item for $k=1,2$, $\partial|_{\mathcal{B}_k} = \partial_{\mathcal{B}_k}$, and 
\item for $1 \leq i \leq m$, 
\begin{equation} \label{eq:prop2-5}
\partial \widehat{a}_i = i_1(a_i) + f_2(a_i) + \gamma_i
\end{equation} where $\gamma_i \in \mathcal{I}(\widehat{a}_1, \ldots, \widehat{a}_m)$.
\end{itemize}
Then, there is an immersed homotopy
\[
M_2 \circ M_1 \simeq  \left( \mathcal{A}_1 \stackrel{j_1 \circ f_1}{\rightarrow} \mathcal{D} \stackrel{j_2 \circ i_2}{\hookleftarrow} \mathcal{A}_3 \right)
\]
where $j_k: \mathcal{B}_k \rightarrow \mathcal{B}_1 * \widehat{\mathcal{A}}_2 * \mathcal{B}_2$ is the inclusion.
\end{proposition}
Note that differentials of the form (\ref{eq:prop2-5}) always exist on $\mathcal{D}$.  For instance, 
one can define
\begin{equation} \label{eq:prop2-5b}
\partial \widehat{a}_i= i_1(a_i) + f_2(a_i) + \Gamma \circ \partial a_i
\end{equation}
where $\Gamma: \mathcal{A}_2 \rightarrow \mathcal{D}$ is the unique $(i_1, f_2)$-derivation satisfying $\Gamma(a_i) = \widehat{a}_i$.
\begin{definition}   
For
$\rho\in \Z_{\geq 0}$ fixed, we define a category $\mathfrak{DGA}^\rho_{\mathit{im}}$ whose objects are triangular DGAs graded by $\Z/\rho$ and whose morphisms are immersed homotopy classes of immersed DGA maps.  Define a related category 
$\mathfrak{DGA}^\rho_{\mathit{hom}}$
 to have the same objects as $\mathfrak{DGA}^\rho_{\mathit{im}}$ but with morphisms given by 
 DGA homotopy classes of ordinary DGA maps.
\end{definition}    
\begin{proposition}[Propositions 3.9 and 3.10 of \cite{PanRu1}]  \label{prop:IFunctor}
For any $\rho \in \Z_{\geq 0}$, $\mathfrak{DGA}^\rho_{\mathit{im}}$ is a category with identity morphisms given by the homotopy classes of the immersed maps 
\[
\left( \mathcal{A} \stackrel{\mathit{id}}{\rightarrow} \mathcal{A} \stackrel{\mathit{id}}{\hookleftarrow} \mathcal{A} \right).
\]
Moreover, there is a functor $I:\mathfrak{DGA}^\rho_{\mathit{hom}} \rightarrow \mathfrak{DGA}^\rho_{\mathit{im}}$ that is the identity on objects and has
\[
I\left( [ \mathcal{A}_1 \stackrel{f}{\rightarrow} \mathcal{A}_2 ] \right) = \left[\left( \mathcal{A}_1 \stackrel{f}{\rightarrow} \mathcal{A}_2 \stackrel{\mathit{id}}{\hookleftarrow} \mathcal{A}_2 \right) \right],
\]
and $I$ is injective on all hom-spaces.
\end{proposition}
We emphasize that the triangularity condition is used crucially in \cite{PanRu1} in establishing the well-definedness of compositions in $\mathfrak{DGA}^\rho_{\mathit{im}}$.  

\subsection{The Legendrian contact homology DGA and exact Lagrangian cobordisms}  \label{sec:2-3}
Recall that the $1$-jet space, $J^1E$, of an $n$-dimensional manifold, $E$, has its standard contact structure $\xi=\ker \alpha$ where $\alpha=dz-\sum y_i dx_i$ in coordinates $(x_1, \ldots, x_n, y_1, \ldots, y_n,z) \in T^*E \times \R = J^1E$ arising from local coordinates $(x_1, \ldots, x_n)$ on $E$.  A {\bf Legendrian submanifold}, $\Lambda$, is an $n$-dimensional submanifold that is tangent to $\xi$ everywhere.  In this article, we will only need to consider Legendrian submanifolds of dimension $n=1$ or $2$, i.e. Legendrian knots and surfaces.  We use the notations $\pi_x:J^1E \rightarrow E$ and $\pi_{xz}:J^1E \rightarrow J^0E = E \times \R$ for the {\bf base} and {\bf front projections}.  Generically, outside
 of a codimension 1 subset $\Lambda_{\mathit{cusp}} \subset \Lambda$ consisting of cusp points (resp. cusp edge and swallowtail points)  the front projection, $\pi_{xz}|_\Lambda$, of a Legendrian knot (resp. surface) is an immersion.  Moreover, any $p \in \Lambda \setminus \Lambda_{\mathit{cusp}}$ has a neighborhood $W \subset \Lambda$ that is the $1$-jet, $j^1f$, of some {\bf local defining function} $f:U \rightarrow \R$ defined in a neighborhood $U \subset E$ of $\pi_x(p)$.  In particular, $\pi_{xz}(W)$  
agrees with the graph of $f$.  Suppose that $E$ has boundary and $N \subset \partial E$ is a boundary component.  If $\Lambda$ intersects $\pi_{x}^{-1}(N)$ transversally (this is equivalent to the map $\pi_x|_\Lambda$ being transverse to $N$), then the {\bf restriction} of $\Lambda$ to $J^1N$ is the Legendrian submanifold $\Lambda|_{N} \subset J^1N$ that is the image of $\Lambda \cap \pi_x^{-1}(N)$ under the restriction map $\pi_x^{-1}(N) \rightarrow J^1N$.  The front projection of $\Lambda|_{N}$ is the intersection of $\pi_{xz}(\Lambda) \subset E \times \R$ with $N \times \R$.

The standard Reeb vector field on $J^1E$ is $\frac{\partial}{\partial z}$, and Reeb chords of $\Lambda$, i.e. trajectories of $\frac{\partial}{\partial z}$ having endpoints on $\Lambda$, correspond to critical points of {\bf local difference functions}, 
$f_{i,j}=f_i-f_j$, where $f_i , f_j$ are local defining functions for $\Lambda$ with $f_i> f_j$.  For $\rho \in \Z_{\geq 0}$, a $\Z/\rho$-valued {\bf Maslov potential} $\mu$ for a $1$ or $2$ dimensional Legendrian $\Lambda$ is a locally constant map $\mu:\Lambda\backslash \Lambda_{cusp} \to \Z/\rho$, 
such that near each cusp point or edge, the value of $\mu$ at the upper sheet of the cusp is one more than the value of $\mu$ at the lower sheet;  such a Maslov potential exists if and only if $\rho$ is a divisor of the Maslov number of $\Lambda$ which is $2 \mathit{rot}(\Lambda)$ when $\Lambda$ is $1$-dimensional, cf. \cite[Section 4]{PanRu1}.    When $\Lambda$ is equipped with a choice of $\Z/\rho$-valued Maslov potential, each Reeb chord, $c$, is assigned a $\Z/\rho$-grading by
$$|c|=\mu(c_u)-\mu(c_l)+\mathit{ind}(f_{i,j})-1 \in \Z/\rho$$
where $c_u$ and $c_l$ are the upper and lower endpoints of $c$ and $ind(f_{i,j})$ is the Morse index of $c$ when viewed as a critical point of the local difference function $f_{i,j}=f_i-f_j$ (where $f_i , f_j$ are defining functions for $\Lambda$ near $c_u$ and $c_l$.)  In the case that $\Lambda$ is connected, the grading of Reeb chords is independent of the choice of Maslov potential, but this is not true in the multi-component case.  

In order to have a well defined grading of Reeb chords (and also of the Legendrian contact homology algebra) we will always work with Legendrians equipped with a choice of Maslov potential.

\begin{definition}
For $\rho \in \Z_{\geq 0}$, a {\bf $\rho$-graded Legendrian}  is a pair $(\Lambda, \mu)$ consisting of a Legendrian submanifold $\Lambda \subset J^1E$ together with a choice of $\Z/\rho$-valued Maslov potential, $\mu$.
\end{definition}

\begin{remark}\label{rem:orient}
When the base space $E$ is oriented a $\Z/2$-valued Maslov potential is equivalent to a choice of orientation for $\Lambda$.  [Indeed, given an orientation for $\Lambda$ one defines a $\Z/2$-valued Maslov potential, $\mu: \Lambda \setminus \Lambda_{\mathit{cusp}}\rightarrow \Z/2$, to be $0$ (resp. $1$) at points where the base projection of $\Lambda$ to $E$ is orientation preserving (resp. reversing).  Moreover, this procedure can be reversed to produce an orientation from a $\Z/2$-valued Maslov potential.]  Thus, $\Z/2$-graded Legendrians are the same as oriented Legendrians, and any $\Z/\rho$-graded Legendrian with $\rho$ even has a well defined orientation (from reducing the Maslov potential mod $2$).
\end{remark}

 For a $\rho$-graded Legendrian submanifold $\Lambda$ in  $J^1E$, symplectic field theory gives a Floer type invariant called the {\bf  Legendrian contact homology DGA} (aka. the Chekanov-Eliashberg algebra) \cite{Che,Eli, EES} that we will denote by  $(\alg(\Lambda), \dd)$.
It is a $\Z/\rho$-graded triangular, based DGA 
over $\Z/2$ generated by Reeb chords of $\Lambda$ with $\Z/\rho$-grading arising from the choice of Maslov potential, $\mu$.  
 The differential $\dd$ is defined by counting holomorphic disks either in $T^*M$, cf. \cite{EES}, or $\mathit{Symp}(J^1M)$, cf. \cite{DR},  with boundary on either the cotangent projection of $\Lambda$ or the Lagrangian cylinder $\R\times \Lambda$.
 According to \cite{Ekh}, one can alternatively compute the differential by counting {\bf gradient flow trees} (abbrv. {\bf GFTs}) which are certain trees whose edges parametrize flowlines of gradients of local difference functions, $-\nabla f_{i,j}$, with $f_{i,j} >0$; see also \cite{RuSu2}.    
The $\Z/\rho$-graded DGA $(\alg(\Lambda), \dd)$ is an invariant of $(\Lambda, \mu)$ up to stable tame isomorphism.

 The LCH DGA is functorial for exact Lagrangian cobordisms, and we will restrict our attention to the case of cobordisms in the symplectization of $J^1M$.  An {\bf exact Lagrangian cobordism} $L$ from $\Lambda_-$ to $\Lambda_+$ is an embedded surface in  $\mathit{Symp}(J^1M) := (\R_t\times J^1M, d(e^t \alpha))$ (as shown in Figure \ref{fig:lagcob})
 that 
 \begin{itemize}
 \item agrees with the cylinder $\R \times \Lambda_+$ (resp. $\R \times \Lambda_-$) when $t$ is very positive (resp. very negative; and 
 \item there is a function $g:L\to \R$ such that $(e^t \alpha)|_{L}=dg$  and $g$ is constant for $t$ near $\pm \infty$. Such a function $g$ is called a {\bf primitive}.
 \end{itemize}
When $\Lambda_-$ is empty, we say that $L$ is an {\bf exact Lagrangian filling} of $L$.
  
 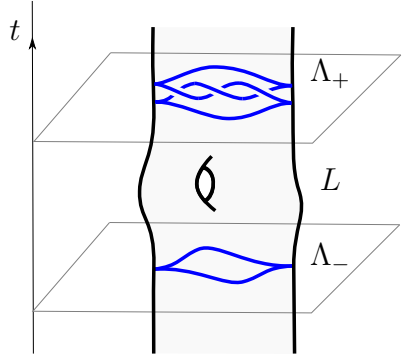
\begin{figure}[!ht]
 \labellist
 \pinlabel $t$ at -10 380
 \pinlabel $\Lambda_-$ at 350 120
 \pinlabel $\Lambda_+$ at  350 330
 \pinlabel $L$ at  350 210
 \endlabellist
 \includegraphics[width=2in]{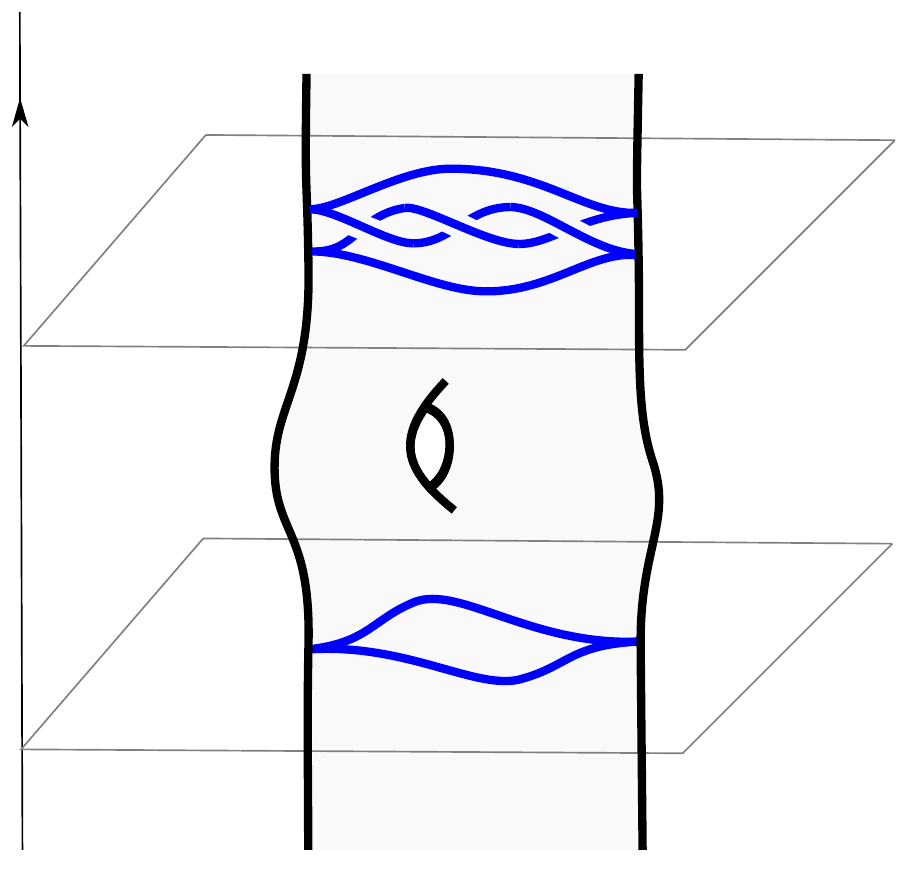}
 \caption{An exact Lagrangian cobordism from $\Lambda_-$ to $\Lambda_+$.}
 \label{fig:lagcob}
 \end{figure} 
 
 According to \cite{EHK, E2, EK}, an exact Lagrangian cobordism $L$ from $\Lambda_-$ to $\Lambda_+$ induces a DGA map $f_{L}$ from $\alg(\Lambda_+)$ to $\alg(\Lambda_-)$.  (For now, we suppress discussion of the grading.)
When two such cobordisms are isotopic through exact Lagrangian cobordisms, their induced DGA  maps are DGA homotopic. 
Moreover, when  two exact Lagrangian cobordisms $L_i$, $i=1,2$ from $\Lambda_{i}$ to $\Lambda_{i+1}$ are concatenated together to form  $L_2 \circ L_1$ (by translating $L_2$ in the positive $t$-direction, truncating the two cobordisms, and gluing along a region of the form $(a,b) \times \Lambda_1$), the induced DGA map $f_{L_2\circ L_1}$ is DGA homotopic to $f_{L_1}\circ f_{L_2}$.  
In summary, we have a functor between a suitably defined category of Legendrian knots with morphisms exact Lagrangian cobordisms and the homotopy category of DGAs.

\subsection{Conical Legendrian cobordisms} \label{sec:2-4}

In \cite{PanRu1}, functoriality for the LCH DGA was generalized to a class of {\it immersed} exact Lagrangian cobordisms by working with conical Legendrian cobordisms. We now review relevant results from \cite{PanRu1}.

\medskip

Let $M$ be a $1$-manifold.  In $J^1(\R_{>0}\times M)$, we denote the $\R_{>0}$ coordinate by $s$.

\begin{definition}  \label{def:j1} Let $\Lambda$ be a Legendrian link in $J^1M$ that is parametrized by $\theta \mapsto (x(\theta), y(\theta), z(\theta))$, let $f:I \rightarrow \R_{>0}$ where $I \subset \R_{>0}$ is an interval, and let $A\in \R$ be a constant. 
Define 
\[
j^1(f(s)\cdot \Lambda+A) \subset J^1(I\times M)
\]
 to be the Legendrian that is parametrized by 
\begin{align*}
(s,\theta) \mapsto &(s, x(\theta), f'(s)z(\theta), f(s) y(\theta), f(s)z(\theta)+A)\\ &= (s, x, u,y,z) \in J^1(I\times M)
\end{align*}
where $(s,x)$ are the coordinates on $I\times M$, $(u,y)$ are the coordinates on the cotangent fibers, and $z$ is the $\R$ coordinate on $J^1(I\times M) = T^*(I\times M) \times \R$. 
\end{definition}
Note that the front projection of $j^1(f(s)\cdot \Lambda+A)$ is obtained from $\pi_{xz}(\Lambda)$ by forming the cylinder in the $s$-direction, multiplying $z$-coordinates by $f(s)$, and then shifting them by $A$.

\begin{definition}
A {\bf conical Legendrian cobordism} $\Sigma$ from $\Lambda_-$ to $\Lambda_+$ is an embedded Legendrian surface in $J^1(\R_{>0}\times M)$ (see Figure \ref{fig:conLeg}) such that
\begin{itemize}
\item
 $\Sigma$ has conical ends, i.e, when $s> s_+$ (resp. $s<s_-$) for some positive number $s_{\pm}$, the Legendrian surface $\Sigma$ is $j^1(s\cdot \Lambda_{\pm}+A_{\pm})$ for some constant $A_{\pm}$; and
 \item the intersection $\Sigma \cap J^1([s_-,s_+] \times M)$ is compact.
\end{itemize}
Two conical
Legendrian cobordisms from $\Lambda_-$ to $\Lambda_+$ are {\bf conical Legendrian isotopic} if they are isotopic through conical Legendrian cobordisms from $\Lambda_-$ to $\Lambda_+$.
 \end{definition}
 
 \begin{figure}[!ht]
 \labellist
 \pinlabel  $\Lambda_-$ at 80 80
 \pinlabel $\Lambda_+$ at 150 90
 \pinlabel $\Sigma$ at 100 20
 \pinlabel $z$ at 8 55
 \pinlabel $x$ at 20 40
 \pinlabel $s$ at  50 0
 \endlabellist
 \includegraphics[width=3in]{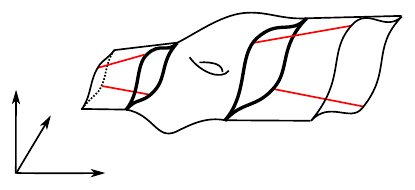}
\caption{A sketch of a conical Legendrian cobordism in the front projection to $J^0(\R_{>0}\times M)$.}
\label{fig:conLeg}
 \end{figure}
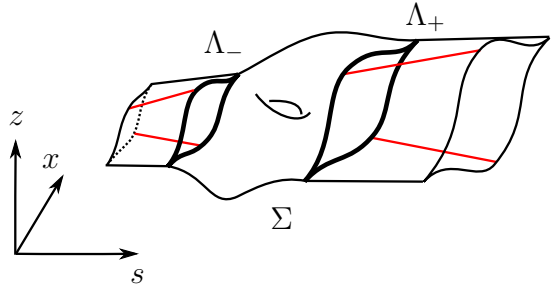
 
 Note that there is a contactomorphism between $J^1(\R_{>0}\times M)$ and $(\symp(J^1M)\times \R_w, dw+e^t \alpha)$ (see \cite{PanRu1} for details).
 Consider the image of $\Sigma$ in $\symp(J^1M)\times \R$ and then take the Lagrangian projection to $\symp(J^1M)$.
 The resulting surface $L$ is an (immersed) exact Lagrangian surface in $\symp(J^1M)$, and we call surfaces obtained in this manner {\bf good Lagrangian cobordisms} from $\Lambda_-$ to $\Lambda_+$.  Note that the conical ends condition on $\Sigma$ implies that $L$ has cylindrical ends and that the primitive is constant on top and bottom cylinders.  As long as the immersed Lagrangians are equipped with such primitives, the construction is reversible, so we have a bijection between conical Legendrian cobordisms and good Lagrangian cobordisms; see eg. \cite[Proposition 4.9]{PanRu1}.   Thus, conical Legendrian cobordisms generalize  exact Lagrangian cobordisms from the embedded case to the immersed case.  Conical Legendrian cobordisms can be concatenated in a way that generalizes the concatenation of exact Lagrangian cobordisms;  see \cite[Section 4.4]{PanRu1}.

 The functoriality of the LCH DGA extends to conical Legendrian cobordisms provided the induced maps are allowed to be immersed DGA maps.
\begin{theorem}[\cite{PanRu1}]  \label{thm:PanRu1}
A conical Legendrian cobordism $\Sigma$ from $\Lambda_-$ to $\Lambda_+$ induces an  immersed DGA map
$$M_\Sigma = \big(\alg(\Lambda_+) \stackrel{f}{\rightarrow} \alg(\Sigma) \stackrel{i}{\hookleftarrow} \alg(\Lambda_-) \big),$$
satisfying:
\begin{enumerate}
\item The DGA $\alg(\Sigma)$  is generated by Reeb chords of $\Sigma$ and Reeb chords of $\Lambda_-$.
\item When $\Lambda_{\pm}$ and $\Sigma$ are $\rho$-graded such that the $\Z/\rho$-valued Maslov potential of $\Sigma$ restricts to the Maslov potentials on $\Lambda_{\pm}$, all the DGAs $\alg(\Lambda_+), \alg(\Sigma),$ and $\alg(\Lambda_-)$ inherit $\Z/\rho$-gradings that are preserved by $f$ and $i$.  
\item When two conical Legendrian cobordisms $\Sigma$ and $\Sigma'$ from $\Lambda_-$ to  $\Lambda_+$ are conical Legendrian isotopic, their induced  immersed DGA maps are  immersed homotopic. 

\item When $\Sigma_1: \Lambda_1 \rightarrow \Lambda_2$ and $\Sigma_2:\Lambda_2 \rightarrow \Lambda_3$ are concatenated the immersed maps satisfy
\[
M_{\Sigma_2 \circ \Sigma_1} \simeq M_{\Sigma_1} \circ M_{\Sigma_2}  \quad \mbox{(immersed homotopy)}.
\]

\item When $\Sigma$ has no Reeb chords, i.e., when it corresponds to an embedded exact Lagrangian cobordism $L$, we have that $\alg(\Sigma)=\alg(\Lambda_-)$ and the $f$ map is DGA homotopic to the induced map $f_L$ from \cite{EHK}. 
 
\end{enumerate}

\end{theorem}

The construction of the immersed DGA map is summarized as follows.
Given a conical Legendrian cobordism $\Sigma$ from $\Lambda_-$ to $\Lambda_+$, we construct a Morse cobordism $\wt\Sigma$ (as in \cite{EHK}) by replacing the conical ends $j^1(s\cdot \Lambda_{\pm} +A_{\pm})$ with standard Morse ends $j^1(h_{\pm}(s)\cdot \Lambda_{\pm}+B_{\pm})$, where $h_{-}$ (resp. $h_+$) are positive Morse functions with a single minimum (resp. maximum). 
 The set of Reeb chords of  $\wt\Sigma$ is in bijection with the set of Reeb chords of $\Sigma$ and $\Lambda_{\pm}$.
The DGA  $\alg(\Sigma)$ is generated by the Reeb chords of $\wt\Sigma$ that correspond to Reeb chords of $\Sigma$ and $\Lambda_-$, and the differential is defined by the usual count of GFTs with respect to a suitable choice of metric $g$ on $\R_{>0}\times M$ having the form $g_\R\times g_\pm$ near the critical points of $h_\pm$ where $g_\R$ is the Euclidean metric on $\R_{>0}$ and $g_\pm$ are regular\footnote{Here, {\bf regular} means that (i) there are no GFTs for $\Lambda$ of negative formal dimension and (ii) all $0$-dimensional GFTs for $\Lambda$ are transversally cut out.  See \cite{Ekh}.} metrics used for computing $\alg(\Lambda_\pm)$.
Moreover, the DGA $\alg(\Lambda_-)$ is a sub-DGA of $\alg(\Sigma)$ and the map $i$ is the natural DGA inclusion map. 
Finally, the DGA map $f$ is defined by counting GFTs with positive puncture at one of the $\Lambda_+$ Reeb chords of $\wt\Sigma$ and with image to the left of the local maximum of $h_+$.

In \cite[Section 6.3]{PanRu1}, the construction of Theorem \ref{thm:PanRu1} is formulated as a functor
\[
F: \leg_{im}^\rho\to \dga_{im}^\rho
\]
called the {\bf immersed LCH functor}.  The category $\leg_{im}^\rho$ has objects $(\Lambda, g, \mu)$ consisting of an (embedded) $\rho$-graded Legendrian knot, $\Lambda \subset J^1M$, with Maslov potential, $\mu$, equipped with a choice of regular metric, $g$, on $M$.
Morphisms from $(\Lambda_-,g_-,\mu_-)$ to $(\Lambda_+, g_+,\mu_+)$ are conical Legendrian isotopy classes of $\rho$-graded conical Legendrian cobordisms, $\Sigma:\Lambda_-\rightarrow \Lambda_+$, with  Maslov potentials, $\mu$, extending the Maslov potentials $\mu_-$ and $\mu_+$.  

\begin{corollary} \label{cor:LCHfunctor} The correspondence
\[
\begin{array}{lcl}
(\Lambda, g, \mu) & \mapsto & \mathcal{A}(\Lambda,g, \mu),  \\
(\Sigma, \mu) & \mapsto & [M_\Sigma]
\end{array}
\]
where $[M_\Sigma]$ denotes the immersed homotopy class of the immersed map $M_\Sigma$ from Theorem \ref{thm:PanRu1} defines a contravariant functor $
F: \leg_{im}^\rho\to \dga_{im}^\rho.$
\end{corollary}

\subsection{Morse minimum cobordisms}  \label{sec:minimum}  
Let $[a_-,a_+] \subset \R$ be a closed interval.  A compact Legendrian $\Sigma_{\mathit{min}} \subset J^1([a_-,a_+]\times M)$ is called a  {\bf Morse minimum cobordism} from $\Lambda_-$ to $\Lambda_+$ if there are neighborhoods $U_- = [a_-, a_-+\delta)$ and $U_+ = (a_+ - \delta, a_+]$, for some $\delta>0$, such that $\Sigma_{\mathit{min}}$ has the form $j^1(h_{\pm}(s)\cdot \Lambda_{\pm}+B_{\pm})$ in $J^1(U_\pm \times M)$ where $h_\pm:U_\pm \rightarrow \R_{>0}$ are positive Morse functions with unique critical points that are non-degenerate local minima at $a_\pm$.  

As in \cite{EK}, the LCH DGA 
of a Morse minimum cobordism $\mathcal{A}(\Sigma_{\mathit{min}})$ is well-defined and can be computed using GFTs with respect to a regular metric $g$ on $[a_-,a_+]\times M$ having the form $g_\R \times g_{\pm}$ in a neighborhood of $\{a_\pm\} \times M$ where $g_{\pm}$ are regular metrics for $\Lambda_\pm$.  Moreover, there are DGA inclusions
\[
i_\pm: \mathcal{A}(\Lambda_\pm) \hookrightarrow \mathcal{A}(\Sigma_\mathit{min})
\]
obtained via identifying Reeb chords of $\Lambda_\pm$ with the Reeb chords of $\Sigma_\mathit{min}$ located above $\{a_\pm\}\times M$.  See also \cite[Section 5.3]{PanRu1}.

\begin{construction}  \label{const:conical}
Given $\Sigma_{\mathit{min}}$ one can form an associated conical Legendrian cobordism $\Sigma_{\mathit{conic}} \subset J^1(\R_{>0}\times M)$ by shifting the interval $[a_-, a_+]$ into $[a_-+s_0, a_++s_0] \subset \R_{>0}$,
 and modifying the ends to have the form $j^1(\widehat{h}_{\pm}(s) \cdot \Lambda_\pm+B_\pm)$ with 
\[
\widehat{h}_{-}: (0, a_-+s_0+\delta) \rightarrow \R_{>0}, \quad \mbox{and} \quad \widehat{h}_{+}: (a_++s_0-\delta, +\infty) \rightarrow \R_{>0}
\]
chosen as follows:
\begin{itemize}
\item The function $\widehat{h}_{-}$ has no critical points, is increasing on $(0, a_-+s_0+\delta)$, and satisfies 
\[
\begin{array}{lr}
\widehat{h}_{-}(s) = s, & \quad \mbox{ for $s \in (0, a_-+s_0]$}, \\
\widehat{h}_{-}(s) = h_{-}(s-s_0), & \quad \mbox{ near $s= a_-+s_0+\delta$.}
\end{array}
\]
\item The function $\widehat{h}_{+}$ has a unique critical point that is a non-degenerate local minimum at $s = a_+ +s_0$, and satisfies
\[
\begin{array}{lr}
\widehat{h}_{+}(s) = h_+(s-s_0), & \quad \mbox{ for $s\in (a_++s_0-\delta,a_++s_0]$}, \\
\widehat{h}_{+}(s) = s, & \quad \mbox{ for $s \gg 0$.}
\end{array}
\]
\end{itemize}
Moreover, from a regular metric $g_\mathit{min}$ for $\Sigma_{\mathit{min}}$ defined on $[a_-,a_+]\times M$ and having the form $g_\mathbb{R}\times g_\pm$ near $\{a_\pm\}\times M$ we construct a metric $g_\mathit{conic}$ on $\R_{>0}\times M$ by shifting $g_{\mathit{min}}$ by $s_0$ in the $\R_{>0}$-direction and then extending to agree with $g_\mathbb{R}\times g_-$ on $(0, a_-+s_0]\times M$ and $g_\mathbb{R}\times g_+$ on $[a_++s_0,+\infty)\times M$.
\end{construction}

Note that the conditions on $\widehat{h}_{-}$ can be arranged with an appropriate choice of $s_0$.  See Figure \ref{fig:MorseMin}.

\begin{figure}[!ht]

\quad

\labellist
\normalsize
\pinlabel $\Sigma_{\mathit{min}}$ [t] at 78 -4
\pinlabel $\Sigma_{\mathit{conic}}$ [t] at 272 -4
\pinlabel $\widetilde{\Sigma}_{\mathit{conic}}$ [t] at 506 -4
\pinlabel $\Sigma$  at 64 60
\pinlabel $\Sigma$  at 260 60
\pinlabel $\Sigma$  at 498 60

\endlabellist

\centerline{\includegraphics[scale=.6]{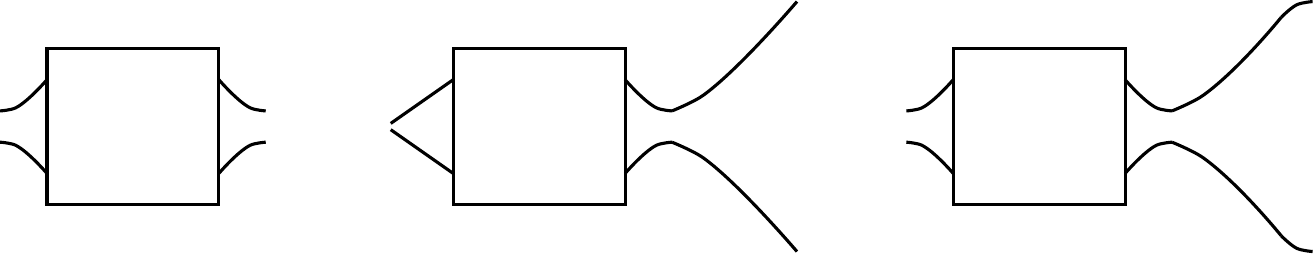}} 

\quad
 
\quad

\caption{A Morse minimum cobordism, an associated conical cobordism, and the cobordism $\widetilde{\Sigma}_{\mathit{conic}}$ used in the proof of Proposition \ref{prop:min}.
}   
\label{fig:MorseMin}
\end{figure}
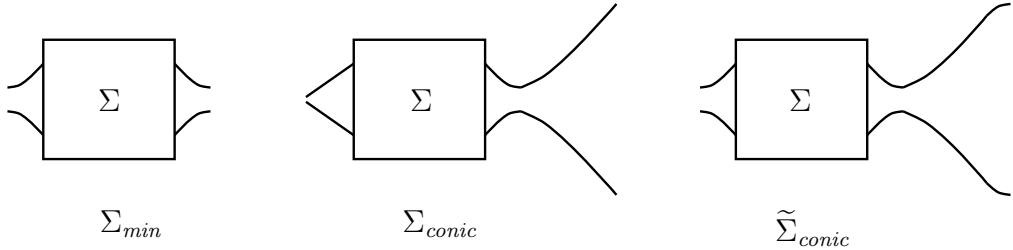

\begin{proposition} \label{prop:min}
When $\Sigma_{\mathit{min}} \subset J^1([a_-,a_+] \times M)$ is a Morse minimum cobordism from $\Lambda_-$ to $\Lambda_+$ and $\Sigma_{\mathit{conic}}$ is an associated conical cobordism  as in Construction \ref{const:conical}.  Then, the immersed LCH functor satisfies
\[
F(\Sigma_{conic}) = \left[\, \mathcal{A}(\Lambda_+) \stackrel{i_+}{\rightarrow} \mathcal{A}(\Sigma_{\mathit{min}}) \stackrel{i_-}{\hookleftarrow} \mathcal{A}(\Lambda_-)\,\right].
\]
\end{proposition}

\begin{proof}
A Morse cobordism $\widetilde{\Sigma}_{\mathit{conic}}$ (with a Morse minimum at the negative end and a Morse maximum at the positive end) that can be used to compute the induced immersed map 
\[
M_{\Sigma_{\mathit{conic}}} = \left(\mathcal{A}(\Lambda_+) \stackrel{f}{\rightarrow} \mathcal{A}(\Sigma_{\mathit{conic}}) \stackrel{i}{\hookleftarrow} \mathcal{A}(\Lambda_-)\right)
\]
 from Theorem \ref{thm:PanRu1} is obtained as follows:
\begin{enumerate}
\item At the  negative end of $\Sigma_{\mathit{conic}}$, replace $j^1(\widehat{h}_-(s) \cdot \Lambda_-+B_-)$ with $j^1(h_-(s-s_0) \cdot \Lambda_-+B_-)$.  This matches the negative end of $\Sigma_{\mathit{min}}$ shifted by $s_0$.
\item At the postive end,  replace the function $\widehat{h}_+$ with some  $\widetilde{h}_+:(a_++s_0-\delta, s_1+\delta] \rightarrow \R_{>0}$.  Here, $s_1$ is chosen large enough so that $\widehat{h}_+(s) = s $ for $s \geq s_1-\delta$ and we require
\begin{itemize}
\item $\widetilde{h}_+(s) = \widehat{h}_+(s)$ for $s \in (a_++s_0 -\delta, s_1]$,
\item $\widetilde{h}_+$ is increasing on $[s_1, s_1+\delta]$ and has a single critical point on this interval that is a non-degenerate local maximum at $s_1+\delta$.   
\end{itemize}
See Figure \ref{fig:MorseMin}.
\end{enumerate}
 By definition, the DGA $\mathcal{A}(\Sigma_\mathit{conic})$ is generated by those Reeb chords of $\widetilde{\Sigma}_{\mathit{conic}}$ that appear in the region where $s < s_1+\delta$.  These Reeb chords are the same as the Reeb chords of $\Sigma_{\mathit{min}}$ but shifted by $s_0$ in the $s$ direction.  Moreover, using metrics of the form $g_{\mathit{min}}$ and $g_{\mathit{conic}}$ as in Construction \ref{const:conical} to compute GFTs, because of the Morse minima all of the GFTs with positive punctures at these chords are contained in the region $[a_-+s_0,a_++s_0]\times M$ and therefore coincide (up to the shift in the $s$ direction) with the GFTs of $\Sigma_{\mathit{min}}$; see eg. \cite{EK}.  Thus, $\mathcal{A}(\Sigma_\mathit{conic}) = \mathcal{A}(\Sigma_\mathit{min})$ and the map $i$ agrees with $i_-$ by definition.   The GFTs that define the map $f$ (by definition) have their unique positive punctures at the Reeb chords located at $\{s_1+\delta\}\times M$, and (because of the local minimum of $\widetilde{h}_+$ at $s = a_++s_0$) have there images contained in $(a_++s_0, s_1+\delta) \times M$.  In this region, $\widetilde{\Sigma}_{\mathit{conic}} = j^1(\widetilde{h}_+(s) \cdot \Lambda_++B_+)$ and $\widetilde{h}_+$ strictly increases, so the computation of such GFTs is as in the case of the identity cobordism from $\Lambda_+$ to itself found in \cite[Proposition 6.15]{PanRu1}.  For each Reeb chord, $b$, of $\Lambda_+$ there is a single gradient trajectory that connects the Reeb chords of $\widetilde{\Sigma}_{\mathit{conic}}$ at $s=s_1+ \delta$ corresponding to $b$ to the Reeb chord of $\widetilde{\Sigma}_{\mathit{conic}}$ at $s=a_++s_0$ corresponding to $b$, and these are the only rigid GFTs.  It follows that $f = i_+$.  
\end{proof}

As discussed in the introduction, Morse minimum cobordisms may be used to define an equivalence relation on augmented Legendrians.  Refer to a triple $(\Lambda, \mu, \epsilon)$ consisting of a $\rho$-graded Legendrian $\Lambda \subset J^1M$ with Maslov potential, $\mu$, and a $\rho$-graded augmentation, $\epsilon:\mathcal{A}(\Lambda) \rightarrow \Z/2$, as a {\bf $\rho$-graded augmented Legendrian}.

\begin{definition}
Two $\rho$-graded augmented Legendrians $(\Lambda_i, \mu_i, \epsilon_i)$, $i =0,1$, in $J^1M$ are {\bf cobordant} if there exists a triple $(\Sigma, \mu, \alpha)$ consisting of a Morse minimum cobordism $\Sigma \subset J^1([0,1]\times M)$ from $\Lambda_0$ to $\Lambda_1$ together with a Maslov potential $\mu$ extending the $\mu_i$ and a $\rho$-graded augmentation $\alpha: \mathcal{A}(\Sigma) \rightarrow \Z/2$ satisfying $\alpha|_{\mathcal{A}(\Lambda_i)} \simeq \epsilon_i$ (DGA homotopy).  
\end{definition}

It is straightforward to see that cobordism defines an equivalence relation on $\rho$-graded augmented Legendrians.

We can now show that Theorem \ref{thm:main} implies Theorem \ref{thm:main2} from the introduction.

\begin{proof}[Proof of Theorem \ref{thm:main2}]
Given a $\rho$-graded augmented Legendrian $(\Lambda, \mu, \epsilon)$, assuming Theorem \ref{thm:main}, there exists a conical Legendrian cobordism with $\Z/\rho$-valued Maslov potential, 
\[
\Sigma: \emptyset \rightarrow \Lambda \quad \mbox{(if $\rho \neq 1$)}  \quad \mbox{or} \quad \Sigma: U \rightarrow \Lambda \quad \mbox{(if $\rho =1$)},
\]
together with a $\rho$-graded augmentation $\alpha:\mathcal{A}(\Sigma) \rightarrow \Z/2$ such that $\epsilon \simeq \alpha \circ f_{\Sigma}$.  Now, a Legendrian isotopy  that is compactly supported in the conical ends of $\Sigma$ modifies $\Sigma$ to have the form $\Sigma_{\mathit{conic}}$ for some Morse minimum cobordism $\Sigma_{\mathit{min}}$ as in Construction \ref{const:conical}.  Then, from Theorem  \ref{thm:PanRu1} there is an immersed DGA homotopy $M_{\Sigma} \simeq M_{\Sigma_{conic}}$ that (after using Proposition \ref{prop:min} to evaluate $M_{\Sigma_{conic}}$ and replacing the stable tame isomorphism $\varphi: \mathcal{A}(\Sigma_{\mathit{min}})*S \rightarrow \mathcal{A}(\Sigma_{\mathit{min}})*S'$ with its associated homotopy equivalence) gives rise to a DGA homotopy commutative diagram: 
\[
\xymatrix{ & \mathcal{A}(\Sigma_{\mathit{min}})  \ar[dd]^{h} \\ \alg(\Lambda_+) \ar[ru]^{i_+}  \ar[rd]^{f_\Sigma}  &  \\ & 
\mathcal{A}(\Sigma) } 
\]  
Then, we can compute
\[
\epsilon \simeq \alpha \circ f_\Sigma \simeq (\alpha \circ h) \circ i_+, 
\]
so that $(\Sigma_{\mathit{min}}, \alpha\circ h)$ provides the cobordism of $\rho$-graded augmented Legendrians from $\emptyset$ (if $\rho \neq 1$) or $(U, \epsilon_1)$ (if $\rho =1$) to $(\Lambda, \epsilon)$ as in the statement of Theorem \ref{thm:main2} where $\epsilon_1= \alpha \circ h \circ i_-$.    
\end{proof}

\begin{remark}
\begin{enumerate}
\item When $\rho$ is even, the cobordism $\Sigma$ and Legendrians are canonically oriented by the Maslov potential $\mu$.  If $\rho$ is odd, $\Sigma$ may be orientable or not.  In 
the odd case, a refined relation  of oriented cobordism for $\rho$-graded augmentation Legendrians arises from requiring that the $\Lambda_i$ and $\Sigma$ are additionally equipped with orientations.  We leave the computation of such oriented, odd-graded cobordism classes of augmented Legendrians in $J^1\R$ as an open problem.
\item Without augmentations, Legendrian cobordism classes in $J^1\R$  are computed as follows; see eg. \cite[Section 5.1]{ArnoldWave}.  Two Legendrians in $J^1\R$ are oriented cobordant if and only if they have the same rotation number, while any two Legendrians in $J^1\R$ are non-oriented cobordant.  In particular, Theorem \ref{thm:main2} implies the well known result of Sabloff, see \cite{Sabloff}, that if $\Lambda \subset J^1\R$ has an $\rho$-graded augmentation with $\rho$ even, then $\mathit{rot}(\Lambda) = 0$.  
\item As Legendrians that admit augmentations exhibit significantly more rigid behavior than general Legendrians, we do not see any {\it a priori} reason to expect that cobordism classes of $\rho$-graded augmented Legendrians should closely match the classical cobordism classes of Legendrians.  In fact, in the case of $J^1S^1$ we will show in \cite{PanRu3} that there are many examples of non-cobordant augmented Legendrians that become cobordant if one ignores the augmentations.
\end{enumerate}
\end{remark}

\section{Immersed DGA maps and augmentations} \label{sec:3}

A DGA morphism $f:\mathcal{A} \rightarrow \mathcal{B}$ contravariantly induces a map on homotopy classes of augmentations, $f^*:\mathit{Aug}(\mathcal{B})/{\sim} \rightarrow \mathit{Aug}(\mathcal{A})/{\sim}$.  In Section \ref{sec:augimmersed}, we consider analogous constructions for immersed DGA maps focusing on the {\it induced augmentation set} of an immersed DGA map, $M$, that is a subset $I(M) \subset \mathit{Aug}(\mathcal{B})/{\sim} \times \mathit{Aug}(\mathcal{A})/{\sim}$.
We show that the induced augmentation set 
induced by a conical Legendrian cobordism $\Sigma$ is a Legendrian invariant of $\Sigma$.  In addition, we make some observations about the form of the augmentation set in the case of Legendrian fillings and embedded Legendrian cobordisms. 
In Section \ref{sec:AugSet} we show that immersed augmentation sets compose as relations, and we record the effect of concatenating a conical Legendrian cobordism with the (invertible) Legendrian cobordism arising from a Legendrian isotopy.

\subsection{Induced augmentation sets}  \label{sec:augimmersed}
We work with $\Z/\rho$-graded DGAs, {\it with $\rho \geq 0$ understood to be fixed}.  As such, when the grading does not need to be emphasized we may refer to $\rho$-graded augmentations simply as augmentations.  We denote by $\aug_{\rho}(\alg)= \aug(\alg)$ the set of all ($\rho$-graded) augmentations of $\mathcal{A}$  to $\Z/2$, and we write $\aug(\alg)/{\sim}$ for the set of all DGA homotopy classes of augmentations.   In the case that $\mathcal{A} = \mathcal{A}(\Lambda)$ is the DGA of some $\rho$-graded Legendrian knot or cobordism, we may shorten these notations to $\aug(\Lambda)$ and $\aug(\Lambda)/{\sim}$.  A DGA map $f:\mathcal{A} \rightarrow \mathcal{B}$ induces a pullback map 
\[
f^*:\aug(\mathcal{B}) \rightarrow \aug(\mathcal{A}), \quad f^*\epsilon = \epsilon \circ f,
\]
and this gives a well-defined map on DGA homotopy classes also denoted as 
\[
f^*:\aug(\mathcal{B})/{\sim} \rightarrow \aug(\mathcal{A})/{\sim}.
\]

Let $M = \big(\mathcal{A}_1 \stackrel{f}{\rightarrow} \mathcal{B} \stackrel{i}{\hookleftarrow} \mathcal{A}_2 \big)$ be an immersed DGA map.   
 Then, the pullback construction results in maps
\[
\aug(\mathcal{A}_1) \stackrel{f^*}{\leftarrow} \aug(\mathcal{B}) \stackrel{i^*}{\rightarrow} \aug(\mathcal{A}_2) \quad \mbox{and}
\quad \aug(\mathcal{A}_1)/{\sim} \stackrel{f^*}{\leftarrow} \aug(\mathcal{B})/{\sim} \stackrel{i^*}{\rightarrow} \aug(\mathcal{A}_2)/{\sim}.
\]
The latter diagram is equivalent to the map 
\[
i^* \times f^*: \aug(\mathcal{B})/{\sim} \rightarrow (\aug(\mathcal{A}_2)/{\sim}) \times (\aug(\mathcal{A}_1)/{\sim})
\]
that we call the
 {\bf augmentation map} 
induced by $M$.  

\begin{definition}  The {\bf induced augmentation set}, $I(M)$, of an immersed DGA map $M = \big(\mathcal{A}_1 \stackrel{f}{\rightarrow} \mathcal{B} \stackrel{i}{\hookleftarrow} \mathcal{A}_2 \big)$ is the image of the augmentation map,
\[
I(M) = \mbox{Im}( i^* \times f^*) \subset  \aug(\alg_2)/{\sim} \times \aug(\alg_1)/{\sim}.
\]
\end{definition}

The induced augmentation set is an invariant of the immersed homotopy class of $M$.  

\begin{proposition}  \label{prop:indaug}
Suppose that $M = \big(\mathcal{A}_1 \stackrel{f}{\rightarrow} \mathcal{B} \stackrel{i}{\hookleftarrow} \mathcal{A}_2 \big)$ and $M' = \big(\mathcal{A}_1 \stackrel{f'}{\rightarrow} \mathcal{B}' \stackrel{i'}{\hookleftarrow} \mathcal{A}_2 \big)$ are immersed DGA maps that are immersed homotopic. Then, there is a bijection $h^*:\aug(\mathcal{B}')/{\sim} \stackrel{\cong}{\rightarrow} \aug(\mathcal{B})/{\sim}$ fitting into a commutative diagram
\begin{equation} \label{eq:augcomm2}
\xymatrix{ & \aug(\mathcal{B})/{\sim}  \ar[ld]_{f^*} \ar[rd]^{i^*} & \\ \aug(\mathcal{A}_1)/{\sim}  & & \aug(\mathcal{A}_2)/{\sim}  \\ & \aug(\mathcal{B}')/{\sim} \ar[uu]_{h^*}^\cong    \ar[lu]_{(f')^*} \ar[ru]^{(i')^*} & } \quad.
\end{equation}
In particular, the induced augmentation sets satisfy $I(M) = I(M')$.
\end{proposition}

\begin{proof}  There exists a diagram   
\begin{equation} \label{eq:augcomm1}
\xymatrix{ & \mathcal{B}*S \ar[dd]^{\varphi}_\cong  & \\ \alg_1 \ar[ru]^{f}  \ar[rd]^{f'}  & & \alg_2 \ar[lu]_{i} \ar[ld]_{i'}  \\ & \mathcal{B}'*S' & } 
\end{equation}
where $\varphi$ is a DGA isomorphism  
such that
\[
\varphi \circ f \simeq f' \quad \mbox{and} \quad \varphi \circ i = i'.
\]   
Let $h= \pi' \circ \varphi \circ \iota$, where $\iota: \mathcal{B} \rightarrow \mathcal{B} *S$ and $\pi':\mathcal{B}'*S' \rightarrow \mathcal{B}'$ are the inclusions and projections, be the associated homotopy equivalence from $\mathcal{B}$ to $\mathcal{B}'$.  
Then, since the maps such as $f: \alg_1 \rightarrow \mathcal{B}*S$ in (\ref{eq:augcomm1}) are implicitly understood to mean $\iota \circ f$, the diagram (\ref{eq:augcomm1}) leads to a similar homotopy commutative diagram (\ref{eq:augcomm1})'  with the vertical map replaced with $h: \B\rightarrow \B'$.  

Now, the association
\[
\begin{array}{lcl}
\mathcal{A} & \quad \leadsto \quad & \aug(\mathcal{A})/{\sim},  \\
 g:\mathcal{A} \rightarrow \mathcal{B} & \quad \leadsto \quad &  g^*:\aug(\mathcal{B})/{\sim} \rightarrow \aug(\mathcal{A})/{\sim}
\end{array}
\]
gives a well-defined contravariant functor from the category $\mathfrak{DGA}^\rho_{\mathit{hom}}$ (where morphisms are DGA homotopy classes of maps) to the category of sets.  In particular, since $h$ is a homotopy equivalence, $h^*$ is a bijection, and since (\ref{eq:augcomm1})' is homotopy commutative, the diagram (\ref{eq:augcomm2}) is indeed fully commutative.
\end{proof}

When $\Sigma:\Lambda_- \rightarrow \Lambda_+$ is a conical Legendrian cobordism with immersed DGA map, $M_\Sigma$, as in Theorem \ref{thm:PanRu1}, we write $I_{\Sigma} = I(M_\Sigma) \subset \aug(\Lambda_-)/{\sim} \times \aug(\Lambda_+)/{\sim}$ and refer to $I_\Sigma$  as the {\bf induced augmentation set} of $\Sigma$.

\begin{corollary}  
Suppose that $\Sigma, \Sigma': \Lambda_- \rightarrow \Lambda_+$ are conical Legendrian cobordisms related by a conical Legendrian isotopy.
Then, there is a commutative diagram
\begin{equation} \label{eq:comm2}
\xymatrix{ & \aug(\Sigma)/{\sim}  \ar[ld]_{f_\Sigma^*} \ar[rd]^{i_{\Sigma}^*} & \\ \aug(\Lambda_+)/{\sim}  & & \aug(\Lambda_-)/{\sim}  \\ & \aug(\Sigma')/{\sim} \ar[uu]_{h^*}^\cong    \ar[lu]_{f^*_{\Sigma'}} \ar[ru]^{i^*_{\Sigma'}} & } 
\end{equation}
i.e., $h^*$ is a bijection and $i^*_{\Sigma'} \times f^*_{\Sigma'} = \left(i^*_{\Sigma} \times f^*_{\Sigma}\right)  \circ h^*$.

In particular, the induced augmentation set
\[
I_{\Sigma}= \mbox{Im}( i^*_{\Sigma} \times f^*_{\Sigma}) \subset \aug(\Lambda_-)/{\sim} \times \aug(\Lambda_+)/{\sim},
\] is an invariant of $\Sigma$.
\end{corollary}

\begin{remark}
\begin{enumerate}
\item  To provide a more refined invariant of $\Sigma$, one can take multiplicities into account when considering $I_\Sigma$.  Eg., the function
\[
\aug(\Lambda_-)/{\sim} \times \aug(\Lambda_+)/{\sim} \rightarrow \Z_{\geq 0}, \quad ([\epsilon_-], [\epsilon_+]) \mapsto \left|( i^*_{\Sigma} \times f^*_{\Sigma})^{-1}\big(([\epsilon_-], [\epsilon_+])\big)\right|
\] 
is a conical Legendrian invariant of $\Sigma$.
  A similar invariant (using a normalized count of augmentations rather than homotopy classes) is studied for $1$-dimensional $\Sigma$ in \cite{Su}.
\item The set $\aug(\Lambda_{\pm})$ can be
 equipped with the additional structure of an $A_\infty$-category whose moduli space of objects is $\aug(\Lambda_{\pm})/{\sim}$, cf. \cite{BC, NRSSZ}, and one could hope for a further refinement of the augmentation map $i^*_{\Sigma} \times f^*_{\Sigma}$ to take this structure into account.  In the case of embedded cobordisms, such a refinement
is made in \cite{Pan1}.
\end{enumerate}
\end{remark}

\subsubsection{The embedded case}  \label{sec:embeddedcase}
When the conical Legendrian cobordism $\Sigma:\Lambda_-\rightarrow \Lambda_+$ corresponds to an embedded Lagrangian cobordism, $L$, i.e. when $\Sigma$ has no Reeb chords, $\mathcal{A}(\Sigma) = \mathcal{A}(\Lambda_-)$ and $i_\Sigma = \mathit{id}_{\mathcal{A}(\Lambda_-)}$.  Therefore, in this case the augmentation map $i^*_\Sigma\times f^*_\Sigma$ is determined by the induced augmentation set
\begin{equation}  \label{eq:embeddedSet}
I_\Sigma = \left\{([\epsilon], f^*_{\Sigma} [\epsilon]) \,|\,  [\epsilon] \in \aug(\Lambda_-)/{\sim}\right\} \subset \aug(\Lambda_-)/{\sim} \times \aug(\Lambda_+)/{\sim}
\end{equation}
which is simply the function $f^*_{\Sigma}$ viewed as a relation.

More generally, when $M = (\mathcal{A}_1 \stackrel{f}{\rightarrow} \mathcal{A}_2 \stackrel{\mathit{id}}{\hookleftarrow} \mathcal{A}_2)$ is the image of an ordinary DGA map $f:\mathcal{A}_1 \rightarrow \mathcal{A}_2$ under the functor from Proposition \ref{prop:IFunctor}, the induced augmentation set for $M$ is just the graph of $f^*$.

\subsubsection{Induced augmentations and immersed fillings}

Let $\Sigma$ be a conical Legendrian filling, 
or equivalently a good immersed Lagrangian filling.  Since $\Lambda_- = \emptyset$, 
 we have $\mathcal{A}(\Lambda_-) = \Z/2$.  Thus, $\aug(\Lambda_-) = \aug(\Lambda_-)/{\sim}$ consists of a single element, so that we can view $I_\Sigma$ as a subset of $\aug(\Lambda_+)$. 
To emphasize the analogy with the case of embedded Lagrangian fillings, given an augmentation $\alpha:\alg(\Sigma) \rightarrow \Z/2$, we use the notation
\[
\epsilon_{(\Sigma, \alpha)} : \alg(\Lambda_+) \rightarrow \Z/2, \quad \epsilon_{(\Sigma, \alpha)} = \alpha \circ f_{\Sigma}
\]
and refer to $\epsilon_{(\Sigma, \alpha)}$ as the {\bf augmentation induced by $\Sigma$ via $\alpha$}.  Thus, the induced augmentation set $I_\Sigma \subset \aug(\Lambda_+)$ consists of those augmentations of $\Lambda_+$ that can be induced by some choice of augmentation for $\Sigma$.

\subsection{Concatenation and induced augmentation sets}  \label{sec:AugSet}
Let $\mathfrak{Rel}$ denote the category whose objects are sets and morphisms are relations, i.e. a morphism $R \subset \mathit{Hom}_{\mathfrak{Rel}}(X,Y)$ is just a subset $R \subset X \times Y$.  Given relations $R \subset X\times Y$ and $S \subset Y \times Z$, there composition is
\[
S \circ R = \{(x,z) \in X \times Z\,|\, \exists y \in Y \mbox{ such that } (x,y) \in R \mbox{ and } (y,z) \in S \}.
\]
\begin{observation} \label{obs:relation}
\begin{enumerate}  
\item Any function $f:X \rightarrow Y$ defines a relation $\Gamma_f = \{(x,f(x))\,|\, x \in X \} \subset X \times Y$.
\item Given any $R \subset X \times Y$ and $g:Y \rightarrow Z$, we have
\[
\Gamma_g \circ R = \{ (x,g(y))\,|\, (x,y) \in R\} = (\mathit{id} \times g) (R). 
\]
\item If $f:X \rightarrow Y$ is a bijection, and $S \subset Y \times Z$,then
\[
S \circ \Gamma_f = (f^{-1} \times \mathit{id})(S).
\]
\end{enumerate}
\end{observation}

\begin{proposition} \label{prop:relfunctor} The induced augmentation set construction
gives a contravariant functor $\mathfrak{DGA}^\rho_{\mathit{im}} \rightarrow \mathfrak{Rel}$,
\[
\begin{array}{lcl}
\mathcal{A} & \leadsto & \aug(\mathcal{A})/{\sim} \\
\left[ M \right] & \leadsto & I(M).
\end{array}
\]
In particular, the induced augmentation sets for a pair of conical Legendrian cobordisms,  $\Sigma_i: \Lambda_{i+1} \rightarrow \Lambda_{i}$, $i=1,2$, satisfy
\begin{equation}  \label{eq:InducedComp}
I_{\Sigma_1 \circ \Sigma_2} = I_{\Sigma_1} \circ I_{\Sigma_2}.
\end{equation}
\end{proposition}
\begin{proof}
Let $M_1 = \big( \mathcal{A}_1 \stackrel{f_1}{\rightarrow} \mathcal{B}_1 \stackrel{i_1}{\leftarrow} \mathcal{A}_2 \big)$ and $M_2 = \big( \mathcal{A}_2 \stackrel{f_2}{\rightarrow} \mathcal{B}_2 \stackrel{i_2}{\leftarrow} \mathcal{A}_3 \big)$ be immersed maps. 
The set $I(M_2 \circ M_1)$ only depends on the immersed homotopy class of $M_2 \circ M_1$.  Thus, 
we can compute it using the immersed map $\mathcal{A}_1 \stackrel{f_1}{\rightarrow} \mathcal{D} \stackrel{i_2}{\leftarrow} \mathcal{A}_3$  as in Proposition \ref{prop:CompositionAlt} where $\mathcal{D} = \mathcal{B}_1 * \widehat{\mathcal{A}}_2 * \mathcal{B}_2$ with differential as in (\ref{eq:prop2-5b}).  That $I(M_2 \circ M_1) = I(M_1) \circ I(M_2)$ then follows from:

\medskip

\noindent{\bf Claim:}  
Let $X$ be the set of triples $(\alpha_1, \alpha_2, K)$ such that $\alpha_i \in \aug(\mathcal{B}_i)$ and $K: \mathcal{A}_2 \rightarrow \Z/2$ is a DGA homotopy operator from $i_1^*\alpha_1$ to $f_2^*\alpha_2$, i.e. a $(i_1^* \alpha_1, f_2^* \alpha_2)$-derivation with $i_1^* \alpha_1 - f_2^* \alpha_2 =  K \circ \partial$.  There is a bijection 
\[
X \rightarrow \aug(\mathcal{D}), \quad (\alpha_1, \alpha_2, K) \mapsto \alpha 
\]
where $\alpha:\mathcal{D} \rightarrow \Z/2$ satisfies $\alpha|_{\mathcal{B}_i} = \alpha_i$ and $\alpha(\widehat{a}_i) = K(a_i)$ for all generators $a_i \in \mathcal{A}_2$.

\medskip

To verify the claim note that $(\alpha \circ \partial)|_{\mathcal{B}_i} =0$ if and only if $\mathcal{B}_i$ is an augmentation.  In addition, since $\partial(\widehat{a}_i) = i_1(a_i) + f_2(a_i) + \Gamma \circ \partial a_i$ where $\Gamma(a_i) = \widehat{a_i}$ is an $(i_1, f_2)$-derivation, the equation $\alpha \circ \partial (\widehat{a}_i) = 0$ is equivalent to $i_1^*\alpha_1(a_i) + f_2^*\alpha_2(a_i) = K \circ \partial a_i$.  
\end{proof}

In the case that the corresponding exact Lagrangian cobordism is embedded, we get a simpler statement.

\begin{proposition}  \label{prop:LegIsotopyAug}
Suppose $\Sigma: \Lambda_- \rightarrow \Lambda_+$ is a conical Legendrian cobordism with embedded Lagrangian projection, and let $\Sigma': \Lambda'_- \rightarrow \Lambda_-$ and $\Sigma'': \Lambda_+ \rightarrow \Lambda_+''$.

\begin{enumerate}
\item Then, 
\[
I_{\Sigma \circ \Sigma'} = (\mathit{id} \times f_{\Sigma}^*)(I_{\Sigma'})
\]
\item  If $f^*_{\Sigma}:\aug(\Lambda_-)/{\sim} \rightarrow \aug(\Lambda_+)/{\sim}$ is a bijection, then 
\[
I_{\Sigma''\circ \Sigma} = ((f_{\Sigma}^*)^{-1} \times \mathit{id})(I_{\Sigma''}).
\]
\end{enumerate}
\end{proposition}
\begin{proof}
As discussed in \ref{sec:embeddedcase}, when $\Sigma$ is embedded $I_\Sigma$ is the relation $\Gamma_{f_{\Sigma}^*}$ associated to the function $f_{\Sigma}^*$.  Thus, the formulas from Observation \ref{obs:relation} can be applied.  

\end{proof}

An important case where (2) of Proposition \ref{prop:LegIsotopyAug} applies is when $\Sigma$ is induced by a Legendrian isotopy.  We now briefly review a version of this construction.  

Let $\Phi = \{\Lambda_s\}_{s \in \R_{>0}}$ be a Legendrian isotopy from $\Lambda_-$ to $\Lambda_+$  so that $\Lambda_s:\sqcup_{i=1}^cS^1 \rightarrow J^1M$ for $s \in \R_{>0}$ is a Legendrian embedding satisfying $\Lambda_s = \Lambda_-$ for $0<s\leq 1$ and $\Lambda_s = \Lambda_+$ for $s \gg 0$.  Writing 
\[
\Lambda_s(\theta) = \left(x(s, \theta), \, y(s, \theta), \, z(s, \theta)\right)
\]
there is a conical Legendrian cobordism 
\[
\Sigma_\Phi: \R_{>0} \times(\sqcup_{i=1}^cS^1) \rightarrow J^1(\mathbb{R}_{>0} \times M) = \{(s,x, u,y,z)\}
\] from $\Lambda_-$ to $\Lambda_+$ associated to $\Phi$ that is parametrized by
\[
\Sigma_\Phi(s, \theta) = \left(s, \, x(s,\theta), \, z(s, \theta) + s\cdot \frac{\partial z}{\partial s}(s,\theta) - s \cdot y(s, \theta) \cdot \frac{\partial x}{\partial s}(s,\theta), \, s \cdot y(s, \theta), \, s \cdot z(s, \theta) \right).
\]
 It can be shown that after reparametrizing by an appropriate orientation preserving diffeomorphism, $\alpha:\R_{>0} \rightarrow \R_{>0}$, the conical Legendrian cobordism corresponding to $\Phi' = \{\Lambda_{\alpha(s)}\}$ will not have Reeb chords.  Indeed, with our setup, one can take $\alpha(s) = s^a$ with $a>0$ suitably small.  See \cite[Section 6.1]{EHK} for a version of this construction for exact  Lagrangian cobordisms in $\mathit{Symp}(J^1M)$.  

\begin{corollary} \label{cor:LegIsotopyAug} Let $\Phi = \{\Lambda_s\}_{s \in \R_{>0}}$ be a Legendrian isotopy parametrized so that the conical Legendrian cobordism $\Sigma_\Phi$ does not have Reeb chords.  Then, $f^*_{\Sigma_\Phi}:\aug(\Lambda_-)/{\sim} \rightarrow \aug(\Lambda_+)/{\sim}$ is a bijection.  In particular, we have 
\[
I_{\Sigma_{\Phi} \circ \Sigma'} = (\mathit{id} \times f_{\Sigma_\Phi}^*)(I_{\Sigma'}) \quad  \mbox{and} \quad I_{\Sigma''\circ \Sigma_\Phi} = ((f_{\Sigma_{\Phi}}^*)^{-1} \times \mathit{id})(I_{\Sigma''}).
\]
whenever $\Sigma'$ and $\Sigma''$ are composable with $\Sigma_{\Phi}$.  
\end{corollary}

\begin{proof}  Let $\Phi^{-1}$ be be the inverse Legendrian isotopy $\{\Lambda_{1/s}\}_{s \in \R_{>0}}$ reparametrized, if necessary, to ensure that $\Sigma_{\Phi^{-1}}$ has no Reeb chords.  Then, the immersed LCH functor satisfies
\[
F(\Sigma_{\Phi}) = [M_{\Sigma_{\Phi}}] = I(f_{\Sigma_{\Phi}})  \quad   \mbox{and} \quad F(\Sigma_{\Phi^{-1}} ) = [M_{\Sigma_{\Phi^{-1}}}] = I(f_{\Sigma_{\Phi^{-1}}})
\]
where $I: \dga^\rho_{\mathit{hom}} \rightarrow \dga^\rho_{\mathit{im}}$ is the functor from Proposition \ref{prop:IFunctor}.  There is a clear conical Legendrian isotopy between $\Sigma_\Phi \circ \Sigma_{\Phi^{-1}}$ (resp. $\Sigma_{\Phi^{-1}} \circ \Sigma_\Phi$) and the identity cobordism $j^1(s\cdot \Lambda_+)$ (resp.  $j^1(s\cdot \Lambda_-)$).  From functoriality, it follows that $I(f_{\Sigma_{\Phi}} \circ f_{\Sigma_{\Phi^{-1}}})$ and $I(f_{\Sigma_{\Phi^{-1}}} \circ f_{\Sigma_{\Phi}})$ are identity morphisms in $\dga^\rho_{\mathit{im}}$, and since $I$ is injective on hom-sets we conclude that $f_{\Sigma_\Phi}$ and $f_{\Sigma_{\Phi^{-1}}}$ are homotopy inverses.
\end{proof}

\section{Immersed maps and the cellular DGA}  \label{sec:4}

In \cite{RuSu1, RuSu2, RuSu25}, a cellular version of Legendrian contact homology is introduced and shown to be equivalent to the usual LCH DGA in the case of closed Legendrian surfaces.  In Sections \ref{sec:4-1} and \ref{sec:4-2}, we briefly review the cellular DGA and extend its definition to the case of (compact) Legendrian cobordisms.  The DGA of the identity  cobordisms is computed in Section \ref{sec:4-3} and shown to be a mapping cylinder DGA.   In Section \ref{sec:4-4} a cellular version, $F_\comp$, of the immersed LCH functor is defined with a domain category consisting of $\rho$-graded Legendrians equipped with some additional data and with compact Legendrian cobordisms as morphisms.     
Finally, in Section \ref{sec:4-5}  
we state in Proposition \ref{prop:Fisomorph} a precise relationship between the immersed and cellular LCH functors, $F$ and $F_{\comp}$, that will allow us to work with $F_{\comp}$ in place of $F$ when considering induced augmentation sets; see Corollary \ref{cor:augmentationset}.  
Specifically, after unifying the domain categories by precomposing with suitable functors, $F$ and $F_{\comp}$ become isomorphic.  
\dr{This isomorphism is an extension of the isomorphism between the cellular and LCH DGAs for closed Legendrian surfaces from \cite{RuSu1, RuSu2, RuSu25}, and is presented here in a condensed manner.  A more detailed presentation appears in Appendices A and B of the preprint version of this article \cite{PanRu2Arxiv}.}

\subsection{Review of the cellular DGA}  \label{sec:4-1} The cellular DGA construction requires as input a Legendrian knot or surface equipped with a suitable polygonal decomposition of its base projection.    Recall that for a generic $1$-dimensional Legendrian knot $\Lambda \subset J^1M$ the singularities of the front projection, $\pi_{xz}(\Lambda) \subset M \times \R$, are cusp points and crossings points (i.e. transverse double points).  Generically, front projections of  Legendrian surfaces have crossing arcs and cusp edges as codimension 1 singularities and triple points (intersection of three smooth sheets of $\Lambda$), cusp-sheet intersections (where a smooth sheet intersects a cusp edge), and swallowtail points as codimension 2 singularities.  See eg. \cite{ArnoldWave}; the front singularities for surfaces are illustrated in \cite[Section 2.2]{RuSu1}.  We say that $\Lambda$ has {\bf generic front and base projections} if the front singularities are generic, and the base projections (to $M$) of the different classes of front singularities are all self transverse and transverse to one another.  

\begin{definition}  \label{def:decorated}
Let $\Lambda \subset J^1M$ be a closed Legendrian submanifold with $\dim \Lambda = 1$ or $2$ having generic front and base projections.
  A {\bf compatible polygonal decomposition} for $\Lambda$ is a polygonal\footnote{By polygonal decomposition, we mean a CW-complex decomposition such that the boundary of each $2$-cell consists of a sequence of vertices and edges (with repeats allowed).} decomposition, $\mathcal{E} = \{e^d_\alpha\}$, of the base projection of $\Lambda$,
\[
\sqcup_{d = 0}^2 \sqcup_\alpha e^d_\alpha  =  \pi_x(\Lambda)   \subset M,
\]
(where the superscript $0 \leq d \leq 2$ denotes the dimension of a cell) such that the base projection of the singular set of $\pi_{xz}(\Lambda)$ (crossings, cusps, swallow tail points, etc.) is contained in the $1$-skeleton of $\mathcal{E}$.  
In addition, we require that:
\begin{enumerate}
\item Each $1$-cell is assigned an orientation.  
\item Each $2$-cell is assigned an initial and terminal vertex, $v_0$ and $v_1$.  If $v_0 = v_1$, then a prefered direction around the boundary of the $2$-cell is also chosen.   
\item At each swallowtail point, $s_0$, the two polygonal corners that border the crossing arc near its endpoint at $s_0$ are labelled as $S$ and $T$.
\end{enumerate}
\end{definition}

Let  $(\Lambda, \mathcal{E})$ be a pair consisting of a closed Legendrian knot or surface, $\Lambda \subset J^1M$, together with a choice of compatible polygonal decomposition.  The {\bf cellular DGA} of $(\Lambda, \mathcal{E})$ will be denoted $\alg(\Lambda, \mathcal{E})$ or $\calg(\Lambda, \mathcal{E})$, and is defined as follows.

\medskip

\noindent {\bf Algebra:}  Given a cell, $e^d_\alpha$, denote by $\{S^\alpha_p\}$ the set of {\bf sheets} of $\Lambda$ above $e^d_\alpha$.  By definition, sheets above $e^d_\alpha$ are those components of $\Lambda \cap \pi_x^{-1}(e^d_\alpha)$ not contained in any cusp edge.  (Note: (i) Sheets are subsets of $\Lambda$, not $\pi_{xz}(\Lambda)$, so that, eg., a crossing arc of $\pi_{xz}(\Lambda)$ above a $1$-cell corresponds to two sheets. (ii) A swallowtail point {\it is} considered to be a sheet above a $0$-cell.)  The algebra $\alg(\Lambda, \mathcal{E})$ is the free unital, associative $\Z/2$-algebra whose generators are in bijection with triples $(e^d_\alpha, S^\alpha_i, S^\alpha_j)$ where $S^\alpha_i$ and $S^\alpha_j$ are sheets above $e^d_\alpha$ such that the inequality $z(S^\alpha_i) > z(S^\alpha_j)$ holds pointwise above $e^d_\alpha$.  We denote the generator associated to $(e^d_\alpha, S^\alpha_i, S^\alpha_j)$  as 
\[
a^\alpha_{i,j}, \quad b^\alpha_{i,j}, \quad \mbox{or} \quad  c^\alpha_{i,j}
\]
 when the dimension of $e^d_\alpha$ is $0$, $1$, or $2$ respectively.    

\medskip

\noindent{\bf Grading:}
 A $\Z/\rho$-grading on $\mathcal{A}(\Lambda, \mathcal{E})$ arises from a choice of $\Z/\rho$-valued Maslov potential, $\mu$, for $\Lambda$. 
  The $\Z/\rho$-grading of generators is
 \[
|a^\alpha_{i,j}| = \mu(S^\alpha_i) - \mu(S^\alpha_j)-1, \quad |b^\alpha_{i,j}| = \mu(S^\alpha_i) - \mu(S^\alpha_j),  \quad |c^\alpha_{i,j}| = \mu(S^\alpha_i) - \mu(S^\alpha_j)+1.
\]
(If $S^\alpha_i$ is a swallowtail point above a $0$-cell, then take $\mu(S^\alpha_i)$ to be the value of $\mu$ on the two sheets that cross near $S^\alpha_i$.)

\medskip

\noindent{\bf Differential:}  The differential $\partial: \mathcal{A}(\Lambda, \mathcal{E}) \rightarrow \mathcal{A}(\Lambda, \mathcal{E})$ is characterized on the generators of $\mathcal{A}(\Lambda, \mathcal{E})$ by matrix formulas whose precise form depends on the dimension of the associated cell $e^d_\alpha \in \mathcal{E}$.  We review these formulas here \emph{in the case that $\Lambda$ does not have swallowtail points}.  See also \cite{RuSu1, RuSu3}.

\begin{itemize} \item {\it For $0$-cells:}  We choose a linear ordering of the sheets of $\Lambda$ above $e^0_\alpha$ so that the $z$-coordinates appear in non-increasing order 
\[
z(S^\alpha_{\iota(1)}) \geq z(S^\alpha_{\iota(2)}) \geq \cdots \geq z(S^\alpha_{\iota(n)})
\]
and use it to place the generators, $a^\alpha_{p,q}$, into a matrix, $A$, whose $(i,j)$-entry is $a^\alpha_{\iota(i), \iota(j)}$ when $z(S^\alpha_{\iota(i)}) > z(S^\alpha_{\iota(j)})$ and is $0$ otherwise.  When $\partial$ is applied entry-by-entry to $A$ we have
\[
\partial A = A^2.
\]

\item {\it For $1$-cells:}  After a choice of linear ordering of sheets as above, we place the generators, $b^\alpha_{p,q}$, associated to a $1$-cell, $e^1_\alpha$, into an $n \times n$ matrix, $B$.  In addition, we form $n\times n$ {\bf boundary matrices}, $A_-$ and $A_+$, associated to the initial and terminal vertices, $e^0_-$ and $e^0_+$,  of $e^1_\alpha$.  The sheets above $e^0_\pm$ are identified with a subset of the sheets above $e^1_\alpha$ (since every sheet of $e^0_\pm$ belongs to the closure of a unique sheet of $e^1_\alpha$).  Using this identification, we place the $e^0_\pm$ generators, $a^\pm_{p,q}$, into the corresponding rows and columns of the $n\times n$ matrices $A_\pm$.  Whenever two sheets of $e^1_\alpha$, $S^\alpha_k$ and $S^\alpha_{k'}$, meet at a cusp point above $e^0_\pm$, we place a $2 \times 2$ block of the form 
\[
N = \left[ \begin{array}{cc} 0 & 1 \\ 0 & 0 \end{array} \right]
\]
on the diagonal at the location of the two (possibly non-consecutive) rows and columns of $A_\pm$ that correspond to $S^\alpha_k$ and $S^\alpha_{k'}$ with respect to the linear ordering of sheets of $e^1_\alpha$.  All other entries of $A_\pm$ are $0$.  The differential on the $b^\alpha_{p,q}$ satisfies
\[
\partial B = A_+ (I+B) +(I+B) A_-.
\]

\item {\it For $2$-cells:}  The sheets above a $2$-cell, $e^2_\alpha$, are already linearly ordered by descending $z$-coordinate, and using this ordering we place the generators $c^\alpha_{i,j}$ into an $n\times n$-matrix, $C$.  We form $n\times n$ boundary matrices, $A_{v_0}$ and $A_{v_1}$, associated to the initial and terminal vertices, $v_0$ and $v_1$, for $e^2_\alpha$ following the same procedure as in the $1$-cell case.  Additional boundary matrices $B_1, \ldots, B_l$ are associated to the $1$-cells that appear around the boundary of $C$.  We assume that the numbering is such that $B_1, \ldots, B_{j}$ (resp. $B_{j+1}, \ldots, B_l$) are the boundary matrices for the sequence of edges that appear along the path $\gamma_{+}$ (resp. $\gamma_-$) where $\gamma_\pm$ are the two paths that travel around the boundary of $e^2_\alpha$ (in the domain of the characteristic map) from $v_0$ to $v_1$.  When $v_0=v_1$ one of $\gamma_\pm$ is constant (as specified by the choice of preferred direction around $\partial e^2_\alpha$).  The $B_i$ are formed using the same procedure as for the boundary matrices associated to vertices except that the $2\times 2$ $N$ blocks that correspond to pairs of cusp sheets are replaced with $2\times 2$ blocks of $0$.  The differential satisfies
\[
\partial C = A_{v_1} C + C A_{v_0} + (I+B_{j})^{\eta_j} \cdots (I+B_1)^{\eta_1} +(I+B_l)^{\eta_l} \cdots (I+B_{j+1})^{\eta_{j+1}}
\] 
where $\eta_i$ is $+1$ (resp. $-1$) when the orientation of the $B_i$ edge agrees (resp. disagrees) with the orientation (from $v_0$ to $v_1$) of the corresponding path $\gamma_{\pm}$.  Note that since the $B_i$ are strictly upper triangular, $(I+B_{i})^{-1} = I +B_i + B_i^2 + \cdots +B_i^{n-1}$.

\end{itemize}

\begin{remark}
To define $\partial$ in the case that $\Lambda$ has swallowtail points, the following additions should be made.
\begin{itemize}
\item The definition of the boundary matrices $A_{\pm}$, $A_{v_0}$, and $A_{v_1}$ needs to be adjusted when the vertex is a swallowtail point of $\Lambda$.
\item Whenever a $2$-cell contains one of the corners labeled $S$ or $T$ at a swallowtail point, an additional matrix need to be inserted into the product $(I+B_j) \cdots (I+B_1)$ or $(I+B_{l}) \cdots (I+B_{j+1})$.
\end{itemize}
These details may be found in \cite{RuSu1} or \cite{RuSu3} and are not needed for the arguments that follow. 
\end{remark}

\begin{observation} \label{obs:cobord}
Whenever $\mathcal{E}' \subset \mathcal{E}$ is a (CW) sub-complex, the collection of generators associated to cells of $\mathcal{E}'$ form a sub DGA of $\calg(\Sigma, \mathcal{E})$.  Moreover, if $\mathcal{E}'$ is a decomposition of a curve or curve segment $C \subset M$ such that the restriction of 
$\Sigma$ to $C$ (as in Section \ref{sec:2-3}) is a $1$-dimensional Legendrian knot $\Lambda \subset J^1C$, then this sub-DGA is precisely $\calg(\Lambda, \mathcal{E}')$.  
\end{observation}

\subsection{The cellular DGA for compact Legendrian cobordisms}  \label{sec:4-2}
In extending the definition of the cellular DGA to Legendrian cobordisms, it is most natural to consider compact cobordisms rather than conical cobordisms.  The following definition coincides precisely with a standard notion of Legendrian cobordism  introduced by Arnold, cf. \cite{Arnold1, Arnold2, ArnoldWave}.  

\begin{definition} Given an interval $I \subset [0,1]$ and $\Lambda \subset J^1M$, we write $j^1(1_I \cdot \Lambda)$ for the {\bf Legendrian cylinder} on $\Lambda$ in $J^1(I \times M)$.  It is defined to be the product of $\Lambda \subset J^1M$ with the $0$-section in $T^*I$, i.e.
\begin{equation}  \label{eq:1I}
j^1(1_I \cdot \Lambda) = 0_{T^*I} \times \Lambda \subset T^*I \times J^1M \cong J^1(I \times M).  
\end{equation}
The notation is chosen to be consistent with that of Definition \ref{def:j1} with $1_I:I \rightarrow \R_{>0}$ denoting the constant function $1$.  

Let $\Lambda_0, \Lambda_1 \subset J^1M$.  A {\bf compact Legendrian cobordism} from $\Lambda_0$ to $\Lambda_1$, written $\Sigma: \Lambda_0 \rightarrow \Lambda_1$, is a Legendrian surface $\Sigma \subset J^1([0,1]\times M)$ that, for some $\epsilon >0$, $\Sigma$ agrees with the Legendrian cylinder 
$j^1(1_{[0, \epsilon]} \cdot \Lambda_0)$ in $J^1([0, \epsilon] \times M)$ and agrees with $j^1(1_{[1-\epsilon, 1]} \cdot \Lambda_1)$ in $J^1([1-\epsilon, 1], \times M)$.  
\end{definition}   

When $\Sigma: \Lambda_0 \rightarrow \Lambda_1$ is a compact Legendrian cobordism, we modify the definition of compatible polygonal decomposition for $\Sigma$ to require that $\pi_x(\Sigma) \cap \left(\{0\} \times M\right)$ and $\pi_x(\Sigma) \cap \left( \{1\} \times M \right)$ are (CW) sub-complexes of $\mathcal{E}$, that we denote $\mathcal{E}_0$ and $\mathcal{E}_1$.  Then, the definition of the cellular DGA extends immediately to give DGAs $\mathcal{A}(\Sigma, \mathcal{E})$ when $\Sigma$ is a compact Lagrangian cobordism.  Moreover, since $\mathcal{E}_0$ and $\mathcal{E}_1$ may be viewed as polygonal decompositions for $\Lambda_0$ and $\Lambda_1$, as in the Observation \ref{obs:cobord}, we have inclusion maps
\begin{equation} \label{eq:i0i1}
i_0:\mathcal{A}(\Lambda_0, \mathcal{E}_0) \rightarrow \mathcal{A}(\Sigma, \mathcal{E}) \quad \mbox{and} \quad i_1:\mathcal{A}(\Lambda_1, \mathcal{E}_1) \rightarrow \mathcal{A}(\Sigma, \mathcal{E}).
\end{equation}

\subsection{The DGA of a product cobordism} \label{sec:4-3} Given a $1$-dimensional Legendrian with compatible polygonal decomposition, $(\Lambda, \mathcal{E})$, we now compute the DGA of the product cobordism $\Sigma = j^1(1_{[0,1]} \cdot \Lambda) \subset J^1([0,1] \times M)$. This will be a useful ingredient in a few later arguments.

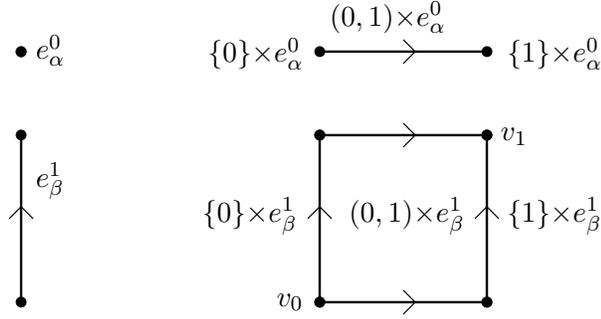
\begin{figure}

\quad 

\quad 

\quad

\labellist
\small
\pinlabel $e^0_\alpha$ [l] at 12 166
\pinlabel $e^1_\beta$ [l] at 12 84
\pinlabel $\{0\}{\times}e^0_\alpha$ [r] at 190 164
\pinlabel $\{1\}{\times}e^0_\alpha$ [l] at 308 164
\pinlabel $(0,1){\times}e^0_\alpha$ [b] at 238 172
\pinlabel $\{0\}{\times}e^1_\beta$ [r] at 186 62
\pinlabel $\{1\}{\times}e^1_\beta$ [l] at 308 62
\pinlabel $(0,1){\times}e^1_\beta$ at 250 62
\pinlabel $v_0$ [r] at 190 8
\pinlabel $v_1$ [l] at 304 112

\endlabellist

\centerline{\includegraphics[scale=.6]{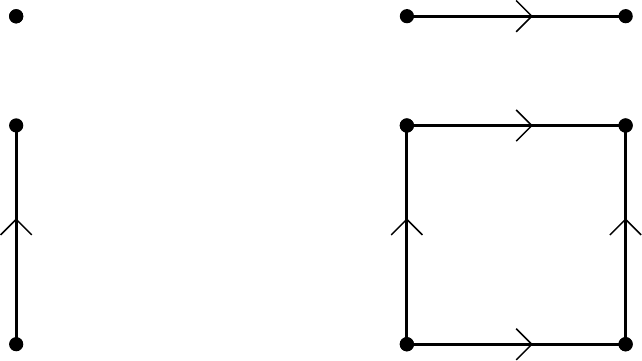}}

\caption{Notation for cells in $\mathcal{E}$ (left) and $\mathcal{E}'$ (right).
}   
\label{fig:Eprime2}
\end{figure}

As a preliminary, use $\mathcal{E}$ to form the {\bf product decomposition}, $\mathcal{E}'$, for $\pi_x(\Sigma) = [0,1] \times \pi_x(\Lambda)$ as follows:
 For each $d$-cell $e^d_\alpha \in \mathcal{E}$ (here, $d=0$ or $1$), decompose 
\[
[0,1] \times e^d_\alpha = \{0\} \times  e^d_\alpha \sqcup (0,1) \times e^d_\alpha \sqcup \{1\} \times e^d_\alpha
\]
 so that  $\{k\} \times e^d_\alpha$ are $d$-cells in $\mathcal{E}'$, for $k=0,1$, and $(0,1) \times e^d_\alpha $ is a $(d+1)$-cell in $\mathcal{E}'$.   Orient $1$-cells of the form $\{k\} \times e^1_\beta$ in the same direction as $e^1_\beta$, and those of the form $(0,1) \times e^0_\alpha$ using the standard orientation of $(0,1)$.  For each $2$-cell, $(0,1) \times e^1_\beta$ choose the initial and terminal vertices $v_0$ and $v_1$ to be $\{0\} \times e^0_-$ and $\{1\} \times e^0_+ $ where $e^0_-$ and $e^0_+$ are the initial and terminal vertices of $e^1_\beta$.  See Figure \ref{fig:Eprime2}. 
 
 Let us fix notation for the generators of the cellular DGA $\mathcal{A}(\Sigma, \mathcal{E}')$.  Observe that for any $e^d_\alpha \in \mathcal{E}$, the sheets of $\Lambda$ above $e^d_\alpha$ are in bijection with the sheets of $\Sigma$ above any of the $\{k\} \times e^d_\alpha$ or $(0,1) \times e^d_\alpha$ cells of $\mathcal{E}'$, so the generators associated to these cells are also in bijection with those of $e^d_\alpha$.  
 \begin{itemize}
 \item For a $0$-cell $e^0_\alpha \in \mathcal{E}$ with corresponding generators $a^\alpha_{i,j}$ we notate the generators corresponding to  $\{0\} \times e^0_\alpha $, $\{1\} \times e^0_\alpha$, and $(0,1) \times e^0_\alpha$ as $i_0(a^\alpha_{i,j})$, $i_1(a^\alpha_{i,j})$, and $b^\alpha_{i,j}$. 
\item For a $1$-cell $e^1_\beta \in \mathcal{E}$ with corresponding generators $b^\beta_{i,j}$ we notate the generators corresponding to  $\{0\} \times e^1_\beta$, $\{1\} \times e^1_\beta$, and $(0,1) \times e^1_\beta$ as $i_0(b^\beta_{i,j})$, $i_1(b^\beta_{i,j})$, and $c^\beta_{i,j}$. 
\end{itemize}

In Proposition \ref{prop:alg01}, the DGA of the product cobordism $\Sigma$ is described as a mapping cylinder DGA.  We now briefly review the relevant definitions, referring the reader to \cite[Section 2]{PanRu1} for a more thorough treatment in the present algebraic setting of triangular DGAs over $\Z/2$.  Let $f:(\mathcal{A}, \partial_\mathcal{A}) \rightarrow (\mathcal{B}, \partial_\mathcal{B})$ be a DGA map between based DGAs.  The {\bf standard mapping cylinder DGA} of $f$ is $(\mathcal{A} * \widehat{\mathcal{A}} * \mathcal{B}, \partial_\Gamma)$ where $\widehat{\mathcal{A}}$ has generators $\{\widehat{a_i}\}$ in correspondence with the generators $\{a_i\}$ of $\mathcal{A}$ but with the degree shift $|\widehat{a_{i}}| = |a_i|+1$.  The 
differential, $\partial_{\Gamma}$, satisfies
\[
\partial_{\Gamma}|_{\mathcal{A}} = \partial_{\mathcal{A}}, \quad \partial_{\Gamma}|_{\mathcal{B}} = \partial_{\mathcal{B}}, \quad \mbox{and} \quad \partial_\Gamma(\widehat{a_i}) = f(a_i) + a_i+ \Gamma \circ \partial_\mathcal{A}(a_i)
\]
where $\Gamma: \mathcal{A} \rightarrow \mathcal{A} * \widehat{\mathcal{A}} * \mathcal{B}$ is the unique $(f,\mathit{id}_\mathcal{A})$-derivation satisfying $\Gamma(a_i) = \widehat{a_i}$.  

\begin{proposition}  \label{prop:alg01}  Given $(\Lambda, \mathcal{E})$ with $\Lambda \subset J^1M$ and $\dim \Lambda = 1$, let $(\Sigma, \mathcal{E}')$ be the product cobordism, $\Sigma = j^1(1_{[0,1]}\cdot \Lambda) \subset J^1([0,1] \times M)$ equipped with the product decomposition, $\mathcal{E}'$.
\begin{enumerate}
\item For $k=0,1$, the maps $i_k:\alg(\Lambda, \mathcal{E}) \rightarrow \alg(\Sigma, \mathcal{E}')$ that extend the correspondence of generators are DGA isomorphisms from $\alg(\Lambda, \mathcal{E})$ onto the sub-DGAs $\mathcal{A}_k \subset \alg(\Sigma, \mathcal{E}')$ associated to the subcomplexes $ \{k\} \times \mathcal{E}\subset \mathcal{E}'$.
\item Identifying the sub-algebra of $\alg(\Sigma, \mathcal{E}')$ generated by the $b^\alpha_{i,j}$ and $c^\beta_{i,j}$ with $\widehat{\mathcal{A}_0}$ using the grading preserving bijection $\widehat{i_0(a^\alpha_{i,j})} \leftrightarrow b^\alpha_{i,j}$ and $\widehat{i_0(b^\beta_{i,j})} \leftrightarrow c^\beta_{i,j}$ 
gives a DGA isomorphism 
\[
\varphi:(\alg(\Sigma, \mathcal{E}'), \partial) \stackrel{\cong}{\rightarrow} (\alg_0 * \widehat{\alg_0} * \alg_1, \partial_\Gamma)
\]
with the standard mapping cylinder DGA of the map $i_1 \circ i_0^{-1}: \mathcal{A}_0 \rightarrow \mathcal{A}_1$.
\end{enumerate}
\end{proposition}
\begin{proof}
(1) is obvious. For (2), $\varphi$ is clearly an algebra isomorphism which is the identity on the $\mathcal{A}_k$, so we just to check that $\partial_\Gamma  \circ \varphi (x) = \varphi \circ \partial (x)$ when $x = b^\alpha_{i,j}$ or $c^\beta_{i,j}$.  From the definition of $\partial_\Gamma$,
 we have 
\begin{align}
\label{eq:align1} \partial_\Gamma \widehat{i_0(a^\alpha_{i,j})}  & = i_1(a^\alpha_{i,j}) +i_0(a^{\alpha}_{i,j}) + \Gamma \circ \partial_\Lambda a^\alpha_{i,j}  \\
\label{eq:align2} \partial_\Gamma \widehat{i_0(b^\beta_{i,j})}  & = i_1(b^\beta_{i,j}) +i_0(b^\beta_{i,j}) + \Gamma \circ \partial_\Lambda b^\beta_{i,j} 
\end{align} 
where $\partial_\Lambda$ is the differential on $\alg(\Lambda, \mathcal{E})$ and $\Gamma:\alg(\Lambda, \mathcal{E}) \rightarrow \alg_1 * \widehat{\alg_0} * \alg_0$ is the $(i_1, i_0)$-derivation satisfying $\Gamma(x) = \widehat{i_0(x)}$ on generators.

Using matrix notation, this allows us to compute
\begin{align*}
\varphi \circ \partial( B_\alpha) &= \varphi\left(i_1(A_\alpha)(I+B_\alpha) + (I+B_\alpha) i_0(A_\alpha)\right) \\
 & =  i_1(A_\alpha)  + i_1(A_\alpha) \widehat{i_0(A_\alpha)} + i_0(A_\alpha) + \widehat{i_0(A_\alpha)} i_0(A_\alpha) \\
 & = i_1(A_\alpha) + i_0(A_\alpha) +  \Gamma( A_\alpha^2) \\
 & = \partial_\Gamma( \widehat{i_0(A_\alpha)}) = \partial_\Gamma \circ \varphi(B_\alpha), 
\end{align*}
and
\begin{align*}
\partial_\Gamma \circ \varphi(C_\beta) & = \partial_\Gamma( \widehat{i_0(B_\beta)}) = i_1(B_\beta) + i_0(B_\beta) + \Gamma(\partial B_\beta) \\
 & = i_1(B_\beta) + i_0(B_\beta) + \Gamma(A_+(I+ B_\beta) + (I+B_\beta) A_-) \\
 & = i_1(B_\beta) + i_0(B_\beta) + \Gamma(A_+)(I+ i_0(B_\beta)) +i_1(A_+) \widehat{i_0(B_\beta)} \\  
 &  \quad + \widehat{i_0(B_\beta)} i_0(A_-) + (I+ i_1(B_\beta)) \Gamma(A_-) \\
 & = \varphi\bigg(  i_1(B_\beta) + i_0(B_\beta) + B_+(I+ i_0(B_\beta))  +i_1(A_+) C_\beta  \\
 & \quad + C_\beta i_0(A_-) + (I+i_1(B_\beta)) B_-\bigg) \\
 & = \varphi\bigg( i_1(A_+) C_\beta+  C_\beta i_0(A_-) + (I+B_+)(I+ i_0(B_\beta))  + (I+i_1(B_\beta))(I+ B_-)\bigg)  \\ 
 & = \varphi \circ \partial(C_\beta).
\end{align*}
At the fourth equality, it should be observed that when $A_-$ and $A_+$ are the boundary matrices for $e^1_\beta$ associated to the vertices $e^0_-$ and $e^0_+$, the boundary matrices $B_-$ and $B_+$ for $(0,1) \times e^1_\beta$ associated to the edges $(0,1) \times e^0_-$ and $(0,1) \times e^0_+$ indeed satisfy $\varphi(B_\pm) = \Gamma(A_\pm)$.  [Note that when a $\left[\begin{array}{cc} 0 & 1 \\ 0 & 0 \end{array}\right]$ block appears on the diagonal of $A_\pm$ due to two sheets of $e^1_\beta$ meeting at a cusp above $e^0_\pm$, since $\Gamma(1) =0$ (any derivation has this property), the appropriate $\left[\begin{array}{cc} 0 & 0 \\ 0 & 0 \end{array}\right]$ block will appear in $\Gamma(A_\pm)$.]
\end{proof}

\subsection{Immersed DGA maps from cobordisms via the cellular DGA} \label{sec:4-4}  Our aim is to now define a cellular version of the immersed LCH functor  $F: \mathfrak{Leg}^\rho_{im} \rightarrow \mathfrak{DGA}^\rho_{\mathit{im}}$ from Corollary \ref{cor:LCHfunctor}.  Recall that the domain category for $F$ has $\rho$-graded Legendrians in $J^1M$ equipped with 
regular metrics as objects and has (conical Legendrian isotopy classes of) conical Legendrian cobordisms in $J^1(\R_{>0}\times M)$ (equivalently, good immersed Lagrangian cobordisms in $\mathit{Symp}(J^1M)$) as morphisms.  For the cellular construction we instead work with a  category of compact cobordisms.

With $M = \R$ or $S^1$ and $\rho\geq 0$ fixed, define a {\bf cellular Legendrian cobordism category}, denoted $\mathfrak{Leg}^\rho_{\mathit{cell}}$, whose objects are triples $(\Lambda, \mathcal{E}, \mu)$ consisting of a $1$-dimensional Legendrian link, $\Lambda \subset J^1M$, together with a choice, $\mathcal{E}$, of compatible polygonal decomposition, and a choice, $\mu$, of $\Z/\rho$-valued Maslov potential.  Morphisms from $(\Lambda_0, \mathcal{E}_0, \mu_0)$ to $(\Lambda_1, \mathcal{E}_1, \mu_1)$ 
are equivalence classes of   
compact Legendrian cobordisms $\Sigma \subset J^1([0,1] \times M)$ from $\Lambda_-$ to $\Lambda_+$ having generic base and front projection and equipped with a $\Z/\rho$-valued Maslov potential extending $\mu_0$ and $\mu_1$.  Here, two cobordisms are considered equivalent if their front and base projections are  {\bf combinatorially equivalent}.  That is, $\Sigma$ and $\Sigma'$ are equivalent if there is a homeomorphism $\phi:([0,1]\times M) \times \R \rightarrow ([0,1] \times M) \times \R$ with $\phi(\pi_{xz}(\Sigma)) = \phi(\pi_{xz}(\Sigma'))$ that is a composition  of (i) a  homeomorphism, $\phi_1$, of the $[0,1] \times M$ factor, and (ii) a homeomorphism, $\phi_2$, that preserves the $[0,1] \times M$ factor.  Moreover, $\phi_1$ and $\phi_2$ should be isotopic to the identity and equal  to the identity in a neighborhood of the boundary.  
 
\begin{remark}
\begin{enumerate}
\item  As with the category $\mathfrak{Leg}^\rho_{\mathit{im}}$, in the definition of morphisms $\Sigma$ is not equipped with any additional structure beyond a choice of Maslov potential, eg. $\Sigma$ is not equipped with a polygonal decomposition. 
 
 \item In contrast to $\mathfrak{Leg}^\rho_{\mathit{im}}$, we do NOT  allow general Legendrian isotopies of $\Sigma$ in the equivalence relation used to define morphisms.  This is because from the initial definition of our cellular LCH functor we will not check directly that the assignment of immersed DGA maps to Legendrian cobordisms  factors through general Legendrian isotopies.   However, it will later be established in Corollary \ref{cor:LegInv}  that the induced immersed DGA maps (considered up to immersed homotopy) are indeed Legendrian invariants of $\Sigma$.
 \end{enumerate}
 \end{remark}
 The next proposition defines the {\bf cellular LCH functor}, $F_{\mathit{cell}}$.  
\begin{proposition} \label{prop:4-7}
There is a well-defined contravariant functor $F_{\mathit{cell}}: \mathfrak{Leg}^\rho_{\mathit{cell}} \rightarrow \mathfrak{DGA}^\rho_{im}$ given by
\[
\begin{array}{lcl}
(\Lambda, \mathcal{E}) & \leadsto & \calg(\Lambda, \mathcal{E}), \\
\Sigma: (\Lambda_0, \mathcal{E}_0) \rightarrow (\Lambda_1, \mathcal{E}_1) & \leadsto & \big[\calg(\Lambda_1, \mathcal{E}_1) \stackrel{i_1}{\rightarrow} \calg(\Sigma, \mathcal{E}') \stackrel{i_0}{\hookleftarrow} \calg(\Lambda_0, \mathcal{E}_0)\big] \end{array}
\]
where $\mathcal{E}'$ is any choice of  compatible polygonal decomposition for $\Sigma$ that restricts to $\mathcal{E}_{k}$ on $\{k\}\times M$, $k=0,1$, and the maps $i_0$ and $i_1$ are as in (\ref{eq:i0i1}).  
\end{proposition}

\begin{proof}
To see that $F_{\mathit{cell}}$ is well-defined it is enough to verify independence of the choice of $\mathcal{E}'$.  (Since modifying $(\Sigma, \mathcal{E}')$ in a manner that preserves the combinatorics of the front and base projections does not affect the cellular DGA.)

In \cite[Section 4.2-4.4]{RuSu1}, it is shown in the case when $\Sigma$ is a closed surface that any  compatible polygonal decompositions $\mathcal{E}'$ and $\mathcal{E}''$ for $\Sigma$ can be made the same after some sequence of the following modifications and their inverses (see \cite{RuSu1} for the precise meanings).
\begin{enumerate}
\item Subdivide a $1$-cell.
\item Subdivide a $2$-cell.
\item Delete a $1$-valent edge.
\item Switch the $S$ and $T$ decorations at a swallowtail point.
\end{enumerate}
The proof extends with minor adjustments to the case of a compact cobordism.  Moreover, if $\mathcal{E}'$ and $\mathcal{E}''$ agree above $\partial([0,1]\times M)$, then the polygonal decomposition of $\partial([0,1] \times M)$ can be left unchanged throughout the sequence of modifications.  [To see this, follow the proof of \cite[Theorem 4.1]{RuSu1}, but treat all $0$-cells  (resp. $1$-cells) in $\partial([0,1] \times M)$ as if they belong to what is notated there as $\Sigma_2$ (resp. $\Sigma_1$).]

If $\mathcal{E}'$ and $\mathcal{E}''$ are related by one of the modifications (1)-(4), then \cite[Theorems 4.2-4.5]{RuSu1} provide  stable tame isomorphisms $\varphi:\calg(\Sigma, \mathcal{E}')*S' \rightarrow \calg(\Sigma, \mathcal{E}'')*S''$.  Moreover, in all cases $\varphi$ can be seen to restrict to the identity on the sub-algebras $\calg(\Lambda_{k}, \mathcal{E}_{k})$, $k=0,1$.  [For (1)-(3), $\varphi$ is defined using the construction of Proposition \ref{prop:quotient}, and generators from $\calg(\Lambda_{k}, \mathcal{E}_{k})$ are never among the generators that are canceled.  For (4), $\varphi$ is the identity on all generators except for $1$-cells with endpoints at swallowtail points.]   Thus, we will have identities $\varphi \circ i_{k} = i_{k}$, $k=0,1$, so that $\varphi$ provides the required immersed DGA homotopy.

To verify that $F_{\mathit{cell}}$ preserves composition, suppose that $\mathcal{E}'$ and $\mathcal{E}''$ are compatible polygonal decompositions for a pair of composable (compact) cobordisms, $\Sigma': (\Lambda_0, \mathcal{E}_0) \rightarrow (\Lambda_1, \mathcal{E}_1)$ and $\Sigma'':(\Lambda_0, \mathcal{E}_0) \rightarrow (\Lambda_1, \mathcal{E}_1)$.  Since $\mathcal{E}'$ and $\mathcal{E}''$ agree with $\mathcal{E}_1$ along $\Lambda_1$, they can be glued to form a well-defined compatible polygonal decomposition, $\mathcal{E}' \cup \mathcal{E}''$, on $\Sigma'' \circ \Sigma'$.  Moreover, when this decomposition is used for computing it is readily verified from Definition \ref{def:immersedcomp} that $F_{\mathit{cell}}( \Sigma'' \circ \Sigma') = F_{\mathit{cell}}( \Sigma') \circ  F_{\mathit{cell}}( \Sigma'')$.

Finally, we check that $F_{\mathit{cell}}$ preserves identities.  The identity morphism in $\mathit{Hom}_{\mathfrak{Leg}^\rho_{\mathit{cell}}}((\Lambda, \mathcal{E}), (\Lambda, \mathcal{E}))$ is the product cobordism, $\Sigma = j^1(1_{[0,1]} \cdot \Lambda)$, and for computing $F_{\mathit{cell}}(\Sigma)$ we equip it with the product decomposition, $\mathcal{E}'$.  The computation from Proposition \ref{prop:alg01} shows that the morphism $F_{\mathit{cell}}(\Sigma)$ is the immersed homotopy class of the immersed DGA map
\[
D = \big(\mathcal{A}(\Lambda, \mathcal{E}) \stackrel{i_1}{\rightarrow} (\alg_0 * \widehat{\alg_0} * \alg_1, \partial_\Gamma) \stackrel{i_0}{\hookleftarrow} \mathcal{A}(\Lambda, \mathcal{E}) \big)
\]
where $\alg_0$ and $\alg_1$ are two copies of $\mathcal{A}(\Lambda, \mathcal{E})$ and $(\alg_0 * \widehat{\alg_0} * \alg_1, \partial_\Gamma)$ is the standard mapping cylinder DGA for the identity map on $\mathcal{A}(\Lambda,\mathcal{E})$.  On the other hand, the identity morphism in $\mathfrak{DGA}^\rho_{\mathit{im}}$ is just $[I]$ where $I$ is the immersed DGA map
\[
I = \big(\mathcal{A}(\Lambda, \mathcal{E})  \stackrel{\mathit{id}}{\rightarrow} \mathcal{A}(\Lambda, \mathcal{E}) \stackrel{\mathit{id}}{\leftarrow}  \mathcal{A}(\Lambda, \mathcal{E})\big).
\]
Note that the immersed map $D$ is obtained from two copies of $I$ precisely as in the statement of Proposition \ref{prop:CompositionAlt}, which therefore shows that $D$ is immersed homotopic to $I \circ I = I$ as required.

\end{proof}

\subsection{Isomorphism between the cellular and immersed LCH functors}  \label{sec:4-5} In this section, 
we state  a precise relationship between the immersed and cellular LCH functors, $F$ and $F_{\comp}$. 
%with the proof as well as some technical details postponed until the Appendices \ref{app:iso} and \ref{app:A}.

The functors $F:\leg_{\mathit{im}}^\rho\to \dga_{im}^\rho$ and $F_{\mathit{cell}}:\mathfrak{Leg}^\rho_{\mathit{cell}} \rightarrow \mathfrak{DGA}^\rho_{\mathit{im}}$  have slightly different domain categories.  In $\leg_{im}^\rho$ (resp. $\mathfrak{Leg}^\rho_{\mathit{cell}}$), the objects are (closed, $1$-dimensional) $\rho$-graded Legendrians in $J^1M$ 
equipped with a suitable choice of metric (resp. polygonal decomposition), while morphisms are equivalence classes of conical (resp. compact) $\rho$-graded Legendrian cobordisms.  For comparing the two functors, it is convenient to work with a third domain category \dr{whose objects again are $1$-dimensional Legendrians equipped with some additional data: an {\it admissible transverse decomposition}, $\mathcal{E}_{\pitchfork}$.  The role of $\mathcal{E}_{\pitchfork}$ is to simultaneously specify a compatible polygonal decomposition of $\Lambda$, $\mathcal{E}_{||}$, and a modification of $\Lambda$ by a Legendrian isotopy to $\widetilde{\Lambda}$ so that the cellular DGA of $(\Lambda, \mathcal{E}_{||})$ and the LCH DGA of $\widetilde{\Lambda}$ are related by a canonical homotopy equivalence (in fact, a stable tame isomorphism).  This is precisely the role of the {\it transverse square decompositions} from \cite[Section 3]{RuSu2} in the case of closed $2$-dimensional Legendrians.}

\begin{definition}
With $\rho \in \Z_{\geq 0}$ and $1$-dimensional $M$ fixed, define the {\bf transverse Legendrian category}, $\mathfrak{Leg}^\rho_{\pitchfork}$, to have objects $(\Lambda, \mathcal{E}_\pitchfork, \mu)$ where $\Lambda \subset J^1M$ is Legendrian, $\mu$ is a $\Z/\rho$-valued Maslov potential, and $\mathcal{E}_\pitchfork$ is an {\bf admissible transverse decomposition} for $\Lambda$ as defined in \cite[Appendix A, Definition A.1]{PanRu2Arxiv}.  Morphisms are compact $\rho$-graded Legendrian cobordisms up to combinatorial equivalence of front and base projections (as in Section \ref{sec:4-4}). 
\end{definition}

An admissible transverse decomposition of $\Lambda$ is a cellular decomposition of a neighborhood of the base projection, \dr{$\pi_x(\Lambda) \subset M$, that has cells transverse to the projection of the singular set of $\pi_{xz}(\Lambda)$  and is   
subject to some technical restrictions that are useful for the detailed proof of Proposition \ref{prop:Fisomorph} but are not relevant when applying the proposition.  For use of Proposition \ref{prop:Fisomorph} in the remainder of the article, it will be sufficient to note that any $1$-dimensional Legendrian $\Lambda \subset J^1M$ with generic front and base projections admits an admissible transverse decomposition. } 

In \cite[Appendix A]{PanRu2Arxiv}, two functors are constructed  
\[
\begin{array}{ccc}
G_\mathit{std}: \mathfrak{Leg}^\rho_{\pitchfork} \rightarrow \mathfrak{Leg}^\rho_{\mathit{im}},  & \quad \quad \mbox{and} \quad \quad & G_{||}: \mathfrak{Leg}^\rho_{\pitchfork} \rightarrow \mathfrak{Leg}^\rho_{\comp},  \\
(\Lambda, \mathcal{E}_\pitchfork) \mapsto (\widetilde{\Lambda}, \widetilde{g}) & &  (\Lambda, \mathcal{E}_\pitchfork) \mapsto (\Lambda, \mathcal{E}_{||})
\end{array}
\]
where $\widetilde{\Lambda}$, called the {\bf standard geometric model} for $\Lambda$ \dr{with respect to $\mathcal{E}_\pitchfork$}, is Legendrian isotopic to $\Lambda$ with front and base projections combinatorially equivalent to those of $\Lambda$. The actions on morphisms are as follows: 
\begin{itemize}
\item The morphism spaces in $\mathfrak{Leg}^\rho_{\pitchfork}$ and $\mathfrak{Leg}^\rho_{\comp}$ only depend on $\Lambda$ and are the same in both categories, and $G_{||}$ acts as the identity on morphisms.  
\item Given a compact cobordism $\Sigma$ from $\Lambda_0$ to $\Lambda_1$, we can choose any conical Legendrian cobordism $\Sigma' \subset J^1(\R_{>0} \times M)$ from $G_{\mathit{std}}(\Lambda_0)$ to $G_{\mathit{std}}(\Lambda_1)$ whose front and base projections are  combinatorially equivalent to $\Sigma$ (after truncating the conical ends).  The conical Legendrian isotopy type of $\Sigma'$ depends only on $\Sigma$, and we have $G_{\mathit{std}}([\Sigma]) = [ \Sigma']$. 
\end{itemize}

The standard geometric model, $G_\mathit{std}(\Lambda, \mathcal{E}_\pitchfork)= (\widetilde{\Lambda}, \widetilde{g})$, is constructed so that the Reeb chords of $\widetilde{\Lambda}$ closely match the generators of the cellular DGA, $\mathcal{A}(\Lambda, \mathcal{E}_{||})$.  In fact, $\mathcal{A}(\widetilde{\Lambda}, \widetilde{g})$ is isomorphic to a stabilization of $\mathcal{A}(\Lambda, \mathcal{E}_{||})$, and an explicitly defined homotopy equivalence
\[
q_{\Lambda}: \mathcal{A}(\widetilde{\Lambda}, \widetilde{g}) \rightarrow \mathcal{A}(\Lambda, \mathcal{E}_{||})
\]
that we call the {\bf canonical quotient map} is identified in \cite[Section A.4]{PanRu2Arxiv}.  As $q_\Lambda$ is invertible in the DGA homotopy category $\mathfrak{DGA}^\rho_{\mathit{hom}}$, it follows that the immersed map 
\[
Q_{\Lambda} = \big[\mathcal{A}(\widetilde{\Lambda}, \widetilde{g}) \stackrel{q_\Lambda}{\rightarrow} \mathcal{A}(\Lambda, \mathcal{E}_{||}) \stackrel{\mathit{id}}{\hookleftarrow} \mathcal{A}(\Lambda, \mathcal{E}_{||}) \big]  \in \dga^\rho_{\mathit{im}}
\] 
(that is the image of $q_\Lambda$ under the functor from Proposition \ref{prop:IFunctor}) is an isomorphism in $\dga^\rho_{\mathit{im}}$.

\begin{proposition}  \label{prop:Fisomorph}
The canonical quotient map construction, 
\[
(\Lambda, \mathcal{E}_\pitchfork) \in \mathfrak{Leg}^\rho_\pitchfork   \, \mapsto \, Q_{\Lambda} \in Hom_{\dga^\rho_{\mathit{im}}}(\mathcal{A}(\widetilde{\Lambda}, \widetilde{g}), \mathcal{A}(\Lambda, \mathcal{E}_{||}))
\] 
gives an isomorphism (invertible natural transformation) from the functor that is the composition
\[
\mathfrak{Leg}^\rho_{\pitchfork} \stackrel{G_{std}}{\rightarrow} \mathfrak{Leg}^\rho_{im} \stackrel{F}{\rightarrow} \dga^\rho_{\mathit{im}}.
\]
to the functor
\[
 \mathfrak{Leg}^\rho_{\pitchfork} \stackrel{G_{||}}{\rightarrow} \mathfrak{Leg}^\rho_{\comp} \stackrel{F_{\mathit{cell}}}{\rightarrow} \dga^\rho_{\mathit{im}}.
\]
That is, the $Q_{\Lambda}$ are invertible, and for any compact Legendrian cobordism $\Sigma: (\Lambda_-, \mathcal{E}^{\pitchfork}_-) \rightarrow (\Lambda_+, \mathcal{E}^{\pitchfork}_+)$ with generic base and front projection 
we have a commutative diagram in $\dga^\rho_{\mathit{im}}$,
\begin{equation} \label{eq:QLambda}
\xymatrix{ \mathcal{A}(\widetilde{\Lambda}_+, \widetilde{g}_+)\ar[rr]^{F\circ G_{\mathit{std}}(\Sigma)} \ar[d]_{Q_{\Lambda_+}} & & \mathcal{A}(\widetilde{\Lambda}_-, \widetilde{g}_-) \ar[d]^{Q_{\Lambda_-}}  \\ \calg(\Lambda_+, \mathcal{E}^{||}_+) \ar[rr]_{F_{\mathit{cell}}\circ G_{||}(\Sigma)}  & & \calg(\Lambda_-, \mathcal{E}^{||}_-) }.
\end{equation}
\end{proposition}

In summary, the following diagram is commutative up to a canonical isomorphism of functors:
\[
\xymatrix{ & \mathfrak{Leg}^\rho_{im}  \ar[rd]^{F} & \\ \mathfrak{Leg}^\rho_{\pitchfork} \ar[ru]^{G_{\mathit{std}}}  \ar[rd]^{G_{||}}  & & \dga^\rho_{\mathit{im}} \\ & \mathfrak{Leg}^\rho_{\comp} \ar[ru]^{F_{\mathit{cell}}} & } 
\]

\dr{A detailed proof of Proposition \ref{prop:Fisomorph} is presented over the course of \cite[Appendices A and B]{PanRu2Arxiv}.  It is an  extension of the isomorphism from \cite{RuSu2, RuSu25} for closed surfaces, and we provide here a summary of the key points.}

\begin{proof}[Sketch of Proof]
\dr{The technical requirements in the definition of admissible transverse decomposition are designed so that, with arbitrary $(\Lambda_-, \mathcal{E}_-^\pitchfork), (\Lambda_+, \mathcal{E}_+^\pitchfork) \in \mathfrak{Leg}^\rho_\pitchfork$ fixed, for any (compact) $\Sigma: \Lambda_- \rightarrow \Lambda_+$ the construction of \cite[Section 3]{RuSu2} for closed surfaces generalizes to produce a transverse square decomposition, $\mathcal{E}_\Sigma^\pitchfork$, that in a closed collar neighborhood of $\partial \Sigma = \Lambda_- \cup \Lambda_+$ agrees with the product of $\mathcal{E}_-^\pitchfork$ and $\mathcal{E}_+^\pitchfork$ with the standard CW decomposition of a closed interval.  Applying the square-by-square coordinate construction from \cite{RuSu25} using $\mathcal{E}_\Sigma^\pitchfork$ produces an initial geometric model for $\Sigma$,  $(\overline{\Sigma}, \overline{g})$.  The restriction of $(\overline{\Sigma}, \overline{g})$ to the negative boundary only depends on $(\Lambda_-, \mathcal{E}_-^\pitchfork)$, and this restriction is the definition of $G_{\mathit{std}}(\Lambda_-, \mathcal{E}_-^\pitchfork) = (\widetilde{\Lambda}_-, \widetilde{g}_-)$.  Similarly, $G_{\mathit{std}}(\Lambda_+, \mathcal{E}_+^\pitchfork)$ is the positive boundary of $(\overline{\Sigma}, \overline{g})$.  In addition, \cite[Section 3]{RuSu2} constructs from $\mathcal{E}_\Sigma^\pitchfork$ a compatible decomposition $\mathcal{E}_\Sigma^{||}$ whose restriction to the boundary is independent of $\Sigma$ and matches the $\mathcal{E}^{||}_{\pm}$ that appear in $G_{||}(\Lambda_\pm, \mathcal{E}_\pm^\pitchfork) = (\Lambda_\pm, \mathcal{E}_\pm^{||})$.  A technical point is that for computing the immersed DGA map $F \circ G_{std}(\Sigma) = [\alg(\widetilde{\Lambda}_+, \widetilde{g}_+) \rightarrow  \alg(\widetilde{\Sigma}, \widetilde{g}) \hookleftarrow \alg(\widetilde{\Lambda}_+, \widetilde{g}_+)]$ we need to pick a specific $(\widetilde{\Sigma}, \widetilde{g})$ that is a conical version of $\Sigma$.  This is done \cite[Appendix B]{PanRu2Arxiv} by carefully modifying $(\overline{\Sigma}, \overline{g})$ in the collar neighborhood of the boundary to be a Morse minimum cobordism and then applying the Construction \ref{const:conical}.  As verified in \cite[Appendix B]{PanRu2Arxiv}, the square-by-square computation of LCH from \cite{RuSu2} still applies to produce an explicit stable tame isomorphism relating $\alg(\widetilde{\Sigma}, \widetilde{g})$ and $\alg(\Sigma, \mathcal{E}_\Sigma^{||})$.  Moreover, \cite[Appendix A.6]{PanRu2Arxiv} uses this stable tame isomorphism to verify the commutativity of (\ref{eq:QLambda}).}   
\end{proof}   

\begin{remark}
An arbitrary Legendrian $\Lambda \subset J^1M$
 is related to one 
in the image of $G_{std}$ by a Legendrian isotopy (that preserves the combinatorial appearance of the front projection).  Thus Proposition \ref{prop:Fisomorph} allows the cellular DGA to be used for computing immersed maps $F(\Sigma)$ induced by a conical Legendrian cobordism $\Sigma:\Lambda_-\rightarrow \Lambda_+$ after possibly changing the Legendrians at the ends from $\Lambda_\pm$ to $\Lambda'_{\pm}$ by Legendrian isotopy.  Composing $\Sigma$ with cobordisms $\Sigma_+:\Lambda_+ \rightarrow \Lambda_+'$ and $\Sigma_-: \Lambda_-'\rightarrow \Lambda_-$ induced by the isotopies modifies $\Sigma$ to a cobordism $\Sigma': \Lambda_-' \rightarrow \Lambda_+'$ to which the cellular computation of Proposition \ref{prop:Fisomorph} can be applied directly.  
 Moreover, if one wants to work with the original  $\Lambda_\pm$, then the maps $F(\Sigma_-)$ and $F(\Sigma_+)$ can be 
explicitly computed as in \cite[Section 6]{EHK}.
\end{remark}

The following corollary of Proposition \ref{prop:Fisomorph} shows that the immersed DGA maps associated to a compact cobordism $\Sigma$ by the cellular functor, $F_{\comp}$, only depend on the Legendrian isotopy type of $\Sigma$.

\begin{corollary} \label{cor:LegInv}
Given $(\Lambda_0, \mathcal{E}_0), (\Lambda_1, \mathcal{E}_1) \in \mathfrak{Leg}^\rho_{\comp}$ and compact cobordisms $\Sigma$ and $\Sigma'$.  If $\Sigma$ and $\Sigma'$ are Legendrian isotopic rel. boundary, then $F_\comp(\Sigma) = F_\comp(\Sigma')$ in $\mathfrak{DGA}^\rho_{\mathit{im}}$. 
\end{corollary}

\begin{proof}

\noindent {\bf Case 1.}  Suppose $(\Lambda_k,\mathcal{E}_k)= G_{||}(\Lambda_k,\mathcal{E}_k^\pitchfork)$  for some admissible transverse decompositions $\mathcal{E}_k^\pitchfork$ for $\Lambda_k$, $k=0,1$. 

\medskip

In this case, Proposition \ref{prop:Fisomorph} allows us to compute in $\mathfrak{DGA}^\rho_{\mathit{im}}$,
\[
F_{\comp}(\Sigma) = Q_{\Lambda_1} \circ F( G_{\mathit{std}}(\Sigma)) \circ Q_{\Lambda_0}^{-1} = Q_{\Lambda_1} \circ F( G_{\mathit{std}}(\Sigma')) \circ Q_{\Lambda_0}^{-1} = F_{\comp}(\Sigma').
\]  
[At the second equality we used that $F(G_{\mathit{std}}(\Sigma))$ only depends on the conical Legendrian isotopy type of $G_{\mathit{std}}(\Sigma)$.]

\medskip

\noindent {\bf Case 2.}  For general $\mathcal{E}_0$ and $\mathcal{E}_1$.

\medskip

In this case, fix some choice of admissible transverse decompositions $\mathcal{E}_k^\pitchfork$ for $\Lambda_k$, $k=0,1$.  Let $I_k$ denote the product cobordism $j^1(1_{[0,1]} \cdot \Lambda_k)$ viewed as a morphism in $\mathit{Hom}_{\mathfrak{Leg}^\rho_{\comp}}( (\Lambda_k, \mathcal{E}_k), (\Lambda_k, G_{||}(\mathcal{E}_k^\pitchfork))$, and let $\Sigma_\pitchfork$ and $\Sigma_\pitchfork'$ denote $\Sigma$ and $\Sigma'$ viewed as morphisms in $\mathit{Hom}_{\mathfrak{Leg}^\rho_{\comp}}( (\Lambda_0, G_{||}(\mathcal{E}_0^\pitchfork)), (\Lambda_1, G_{||}(\mathcal{E}_1^\pitchfork))$.  Using Case 1, we compute
\begin{align*}
F_{\comp}(\Sigma) = F_{\comp}(I_1^{-1} \circ \Sigma_\pitchfork \circ I_0) &= F_{\comp}(I_1^{-1}) \circ F_{\comp}(\Sigma_\pitchfork) \circ F_{\comp}(I_0) \\
 &= F_{\comp}(I_1^{-1}) \circ F_{\comp}(\Sigma_\pitchfork') \circ F_{\comp}(I_0) = F_{\comp}(\Sigma').
\end{align*}

\end{proof}

As in Section \ref{sec:augimmersed}, given an immersed DGA map $M= \left(\mathcal{A}_1 \stackrel{f}{\rightarrow} \mathcal{B} \stackrel{i}{\hookleftarrow} \mathcal{A}_2 \right)$ we write $I(M) = \mathit{Im}(i^*\times f^*) \subset \mathit{Aug}_\rho(\mathcal{A}_1)/{\sim} \times \mathit{Aug}_\rho(\mathcal{A}_2)/{\sim}$ for the {\it induced augmentation set} of $M$. According to Proposition \ref{prop:indaug}, $I(M)$ only depends on the immersed homotopy class of $M$.   

\begin{corollary}  \label{cor:augmentationset}  Let $\Sigma \subset J^1([0,1]\times M)$ be a compact Legendrian cobordism from $\Lambda_0$ to $\Lambda_1$.  Let $\mathcal{E}_k^\pitchfork$ be an admissible transverse decomposition for $\Lambda_k$, $k=0,1$, and write $G_{std}(\Lambda_k,\mathcal{E}_k^\pitchfork) = (\widetilde{\Lambda}_k, \widetilde{g}_k)$ and $G_{||}(\Lambda_k,\mathcal{E}_k^\pitchfork) =(\Lambda_k,\mathcal{E}_k^{||})$.     
Then, the induced augmentation sets 
$I_\Sigma := I(F(G_{\mathit{std}}(\Sigma)))$ and $I^{\mathit{cell}}_{\Sigma} := I(F_\mathit{cell}(G_{||}(\Sigma)))$ are related by 
\begin{equation} \label{eq:qlambda}
(q_{\Lambda_0}^* \times q_{\Lambda_1}^*)\big(I^{\mathit{cell}}_{\Sigma} \big) = I_\Sigma
\end{equation}
where 
\[
q_{\Lambda_0}^* \times q_{\Lambda_1}^*:\aug_\rho(\Lambda_0, \mathcal{E}_0^{||})/{\sim} \times \aug_\rho(\Lambda_1, \mathcal{E}_1^{||})/{\sim} \stackrel{\cong}{\rightarrow} \aug_\rho(\widetilde{\Lambda}_0, \widetilde{g}_0)/{\sim} \times \aug_\rho(\widetilde{\Lambda}_1, \widetilde{g}_1)/{\sim}
\]
 is the bijection induced by the canonical quotient maps,  $q_{\Lambda_k}: \alg(\widetilde{\Lambda}_k, \widetilde{g}_k) \rightarrow \alg(\Lambda_k, \mathcal{E}_k^{||})$.
\end{corollary}

\begin{proof}
  Applying the functor from Proposition \ref{prop:relfunctor} to (\ref{eq:QLambda}) produces the diagram of relations
\begin{equation} \label{eq:QLambdaRel}
\xymatrix{ \aug(\widetilde{\Lambda}_1), \widetilde{g}_1)/{\sim}  & \aug(\widetilde{\Lambda}_0, \widetilde{g}_0)/{\sim} \ar[l]_{I_\Sigma}  \\ \aug(\Lambda_1, \mathcal{E}^{||}_1)/{\sim} \ar[u]^{I(Q_{\Lambda_1})}  & \aug(\Lambda_0, \mathcal{E}_0)/{\sim} \ar[l]_{I^{\mathit{cell}}_\Sigma} \ar[u]_{I(Q_{\Lambda_0})} }.
\end{equation}
As noted in \ref{sec:embeddedcase}, the induced augmentation sets associated to $Q_{\Lambda_{k}}$ are just the graphs of $q^*_{\Lambda_{k}}$, so Observation \ref{obs:relation} (2) and (3) can be applied to arrive at (\ref{eq:qlambda}).
\end{proof}

\section{Computations via Morse complex families}  \label{sec:MCF}
With Corollary \ref{cor:augmentationset} we have seen that for a compact Legendrian cobordism, $\Sigma$, the two versions of the 
induced augmentation sets, $I_\Sigma^{\mathit{cell}}$ and $I_\Sigma$, defined via the cellular and ordinary versions of the LCH functor, $F_{\mathit{cell}}$ and  $F$, are equivalent invariants of $\Sigma$ (in a canonical way).  
In this section, we focus on obtaining methods for computing $I_\Sigma^{\mathit{cell}}$ based on \dr{an extension of  correspondences between augmentations and Morse complex families (abbr. MCFs), cf. \cite{Henry, HenryRu2, RuSu3}, to the case of $2$-dimensional Legendrian cobordisms.   In Section \ref{sec:5-2}, we review the definition of MCFs and include the present context of $2$-dimensional Legendrian cobordisms.  In Section \ref{sec:5-4}, we make use of a bijection between the set of equivalence classes of MCFs and DGA homotopy classes of augmentions of the cellular DGA for $1$-dimensional Legendrians to establish a characterization of induced augmentation sets in terms of MCFs;  see Proposition \ref{prop:MCFcomp}.  The section concludes in Section \ref{sec:Aform} with a discussion of $A$-form MCFs, as defined by Henry \cite{Henry}, that are particularly convenient for computing induced augmentation sets in explicit examples.}

%%Augmentations of the cellular DGA can be conveniently viewed as \emph{chain homotopy diagrams} (abbr. CHDs) which are certain homotopy commutative diagrams of chain complexes that have their form specified by the base and front projections of a Legendrian knot or surface.    CHDs for $2$-dimensional Legendrians were introduced in \cite{RuSu3} where they were shown to be closely connected with \emph{Morse complex families} (abbr. MCFs) which are combinatorial objects motivated by Morse theory.  In the 1-dimensional case, Morse complex families have been studied in the works of Henry and Rutherford, \cite{Henry, HenryRu1, HenryRu2}.  In particular, for a Legendrian knot, $\Lambda \subset J^1\R$, \cite{Henry, HenryRu2} establishes a bijection between $\aug(\Lambda)/{\sim}$ and equivalence classes of MCFs for $\Lambda$. 

%In Section \ref{sec:5-1}, we briefly review the correspondence between augmentations and CHDs in the context of compact Legendrian cobordisms. In Sections \ref{sec:5-2}-\ref{sec:5-3MCF}, we review MCFs and their connection with CHDs in the context of cobordisms.  In Section \ref{sec:5-4},    
 %establishes a characterization of induced augmentation sets in terms of MCFs.   The section concludes in Section \ref{sec:Aform} with a discussion of $A$-form MCFs, as defined by Henry, that are particularly convenient for computing induced augmentation sets in explicit examples.

\subsection{Morse complex families}  \label{sec:5-2}
We now review the definition of Morse complex families for Legendrian knots and surfaces, as in \cite{Henry, RuSu3}, allowing for the case of Legendrian cobordisms.  %Morally, a Morse complex family can be thought of as a CHD with all chain maps and chain homotopies factored into elementary pieces.  

\subsubsection{The 1-dimensional case} Let $(\Lambda, \mu) \subset J^1M$ be a $1$-dimensional $\rho$-graded Legendrian knot.
 Recall that we use the respective notations $\pi_x:J^1M \rightarrow M$ and $\pi_{xz}:J^1M \rightarrow M \times \R$ for the base and front projections.

\begin{definition} 
A {\bf $\rho$-graded Morse complex family} (abbr. {\bf MCF}) for $(\Lambda, \mu)$ 
 is a pair $\mathcal{C} = (H, \{d_\nu\})$ consisting of the following items:
\begin{enumerate}
\item[(1)] A {\bf handleslide set} $H$ which is a finite collection of points in $M$ such that each $p\in H$ is equipped with lifts $u(p), l(p) \in \Lambda$ not belonging to cusp edges and satisfying $\pi_x(u(p)) = \pi_x(l(p)) = p$, $z(u(p)) > z(l(p))$, and $\mu(u(p)) = \mu(l(p))$.  We refer to $p \in H$ as a (formal) handleslide of $\mathcal{C}$ with upper and lower endpoints at $u(p)$ and $l(p)$.
\end{enumerate}
Denote by $\Lambda_{\mathit{sing}}$ the base projection of the singular set of $\pi_{xz}(\Lambda)$ (crossings and cusps), and let $\{R_\nu\}$ be the collection of path components of  $M \setminus (H \cup \Lambda_{\mathit{sing}})$.  Note that above each $R_\nu$, the sheets of $\Lambda$ are totally ordered by their $z$-coordinates, so they can be enumerated as 
\[
S^\nu_1, \ldots, S^\nu_{n_\nu} \quad \mbox{with } z(S^\nu_1) > \cdots > z(S^\nu_{n_{\nu}}).
\]
We then have $\Z/\rho$-graded vector spaces $V(R_\nu) = \mbox{Span}_{\Z/2}\{S^\nu_i\}$ \dr{where the grading is via the Maslov potential $|S^\nu_i|= \mu(S^\nu_i)$}.  We will sometimes omit the superscript $\nu$ in the notation for sheets.

\begin{enumerate}
\item[(2)] The second piece of information defining an MCF is a collection of strictly upper triangular 
%(as in (\ref{eq:strictly})) 
differentials $\{d_\nu\}$ of degree $+1$ mod $\rho$ making each $(V(R_\nu), d_\nu)$ into a $\Z/\rho$-graded cochain complex.  \dr{Here, the strictly upper triangular condition is that the matrix coefficient $\langle d_\nu S^\nu_q, S^\nu_p\rangle$ is $0$ unless $z(S^\nu_p) > z(S^\nu_q)$.}
\end{enumerate}
Moreover, $H$ and $\{d_\nu\}$ are required to satisfy the following Axiom \ref{ax:1dim}.
\end{definition}

\begin{axiom} \label{ax:1dim}
Let $R_0$ and $R_1$ be connected components of $\pi_{x}(\Lambda) \setminus (H \cup \Lambda_{\mathit{sing}})$ that border each other at a  point $p \in H \cup \Lambda_{\mathit{sing}}$.

\begin{enumerate}
\item If $p \in H$ is a \emph{handleslide}  with upper endpoint $u(p) \in S_u$ and $l(p) \in  S_l$, then the handleslide map 
\[
h_{u,l}( S_i) = \left\{\begin{array}{cr} S_i, & i \neq l, \\ S_u+S_l, & i = l, \end{array} \right. 
\] 
gives a chain isomorphism $h_{u,l} :(V(R_0), d_0)  \stackrel{\cong}{\rightarrow} (V(R_1), d_1)$.
 
\item If $p$ is the base projection of a \emph{crossing} between sheets $S_k$ and $S_{k+1}$, then the permutation map
\[
Q(S_k) = S_{k+1}, \,  Q(S_{k+1}) = S_k, \, Q(S_{i}) = S_i, i \notin \{k,k+1\}
\]
is a chain isomorphism, $Q :(V(R_0), d_0)  \stackrel{\cong}{\rightarrow} (V(R_1), d_1)$.
\item Suppose that $p$ is a \emph{cusp} point such that the two sheets of the cusp, $S^1_{k}$ and $S^1_{k+1}$, exist above $R_1$ but not above $R_0$.  Then, identifying sheets of $R_0$ and $R_1$ whose closures intersect at $p$ gives an inclusion $\iota: V(R_0) \hookrightarrow V(R_1)$ with
\[
V(R_1) = \iota(V(R_0)) \oplus V_{\mathit{cusp}}, \quad V_{\mathit{cusp}} = \mbox{Span}_{\Z/2}\{S^1_k, S^1_{k+1}\}.
\]
We require that $\iota \circ d_0 = d_1 \circ \iota$ and $d_1 S^1_{k+1} = S^1_k$, i.e. as a complex $(V(R_1), d_1)$ is the split extension of $(V(R_0), d_0)$ by $(V_{\mathit{cusp}}, d')$ where $d'$ maps the lower sheet of the cusp to the upper sheet of the cusp. 
\end{enumerate}
\end{axiom}

\begin{remark} \label{rem:6-8}
\begin{enumerate}
\item The requirement in Axiom \ref{ax:1dim} (3) used here is stronger than in some other versions of the definition, eg. in \cite{Henry, HenryRu1, HenryRu2}. In the terminology of \cite{Henry}, with our definition we only consider Morse complex families with ``simple births and deaths''.  The collection of equivalence classes of MCFs is not affected by making this restriction; see \cite[Proposition 3.17]{Henry}.

\item If $ \Lambda \subset J^1 \R$, then the collection of differentials $\{d_\nu\}$ is uniquely determined by $\Lambda$ and $H$, since one can work from left to right and apply Axiom \ref{ax:1dim} to determine the $\{d_\nu\}$ inductively.  However, it should be noted that not every handleslide set $H$ for $\Lambda$ will produce an MCF.  For example, consider the case when $R_0$ and $R_1$ are seperated by a crossing point, and  $d_0$ has already been determined.  If $\langle d_0 S_{k+1}, S_k \rangle \neq 0$, then the differential $d_1 = Q \circ d_0 \circ Q^{-1}$ required by Axiom \ref{ax:1dim} (2) will not be strictly upper triangular.
\end{enumerate}
\end{remark}

In figures, we picture handleslides as vertical segments connecting $u(p)$ and $l(p)$ on the front projection of $\Lambda$.  When convenient, differentials $d_\nu$ may be indicated above $R_{\nu}$ by drawing a dotted arrow from $S_i$ to $S_j$ whenever $\langle d_\nu S_j, S_i \rangle =1$.  See for example Figures \ref{fig:Endpoints} and \ref{fig:8_21} below.

\subsubsection{The 2-dimensional case}

Consider a 2-dimensional $\rho$-graded compact Legendrian surface, $(\Sigma, \mu) \subset J^1M$ where $\dim M = 2$. 
We allow for $\partial M \neq \emptyset$, but in this case require that $\Sigma$ is properly embedded, eg. $\Sigma$ could be a compact Legendrian cobordism in $J^1([0,1] \times \R)$ or $J^1([0,1] \times S^1)$.  We again use the notation $\Sigma_{\mathit{sing}} \subset M$ for the base projection of the singular set of $\pi_{xz}(\Sigma)$ (which now includes crossings, cusps, and swallowtail points).

\begin{definition}  \label{def:MCF2d}
A {\bf $\rho$-graded MCF} for $\Sigma$ is a triple $(H_0, H_{-1}, \{d_\nu\})$ consisting of:
\begin{enumerate}
\item A set of {\bf super-handleslide points}, $H_{-1} \subset \mathit{Int}(M)\setminus \Sigma_{\mathit{sing}}$, equipped with upper and lower endpoint lifts to $\Sigma$  satisfying $\mu(u(p)) = \mu(l(p))-1$ for all $p \in H_{-1}$.
\item A collection of {\bf handleslide arcs} which is an immersion $H_0: X \rightarrow M$ of a compact $1$-manifold, transverse to $\partial M$.
Aside from the exceptions at boundary points specified in Axiom \ref{ax:endpoints} below, we require that $H_0$ is transverse to $H_0 \cup \Sigma_{\mathit{sing}}$ and that the only self intersections are at transverse double points in $\mathit{Int}(M)$.  In addition, $H_0$ should be equipped with upper and lower endpoint lifts defined on the interior of $X$ 
\[
u,l:\mathit{Int}(X) \rightarrow \Sigma, \quad \pi_x\circ u = \pi_x \circ l = H_0, 
\]
satisfying
\[
 u(s) > l(s) \quad \mbox{and} \quad \mu(u(s)) = \mu(l(s)) \quad \forall s \in \mathit{Int}(X).
\]
Endpoints of handleslide arcs must satisfy Axiom \ref{ax:endpoints} below.
\item A collection of differentials $\{d_\nu\}$ of degree $+1$ mod $\rho$ making each $(V(R_\nu), d_\nu)$ into a $\Z/\rho$-graded cochain complex where $\{R_\nu\}$ is the collection of path components of $M \setminus ( H_{-1} \cup H_0(X) \cup \Sigma_{\mathit{sing}})$.  Note that $H_{-1} \cup H_0(X) \cup \Sigma_{\mathit{sing}}$ is a stratified subset of $M$ whose $1$-dimensional strata are \emph{handleslide arcs}, \emph{crossing arcs}, and \emph{cusp arcs}.  For each such $1$-dimensional stratum, we require that the complexes $(V(R_0), d_0)$ and $(V(R_1), d_1)$ that appear adjacent to the stratum are related as in Axiom \ref{ax:1dim} (1)-(3). 
\end{enumerate}
\end{definition}

Before stating Axiom \ref{ax:endpoints} we introduce some convenient terminology.  Above $M\setminus \Sigma_{\mathit{sing}}$ sheets of $\Sigma$ are locally labelled $S_1, \ldots, S_n$ with descending $z$-coordinate.  We will refer to a handleslide arc that has its upper and lower lift on sheets $S_i$ and $S_j$ as an {\bf $(i,j)$-handleslide}.  We caution that this terminology only applies locally since the labeling of upper and lower sheets of a handleslide arc may change when the arc passes through $\Sigma_{\mathit{sing}}$.  We use a similar terminology for super-handleslide points.

\begin{axiom} \label{ax:endpoints}
Endpoints of handleslide arcs can occur at transverse intersections of $H_0$ and $\partial M$.  All other endpoints of handleslide arcs occurring in the interior of $M$ are as specified in the following:
\begin{enumerate}
\item Suppose that $p \in M$ is a transverse double point for $H_0|_{\mathit{Int}(X)}$ where a $(u_1,l_1)$-handleslide arc crosses a $(u_2,l_2)$-handleslide arc.  If $l_1 = u_2$, then at $p$ there is a single endpoint of a $(u_1, l_2)$-handleslide arc.
\item Let $p \in H_0$ be a $(u,l)$-super-handleslide point, 
 and fix some region $R_\nu$ bordering $p$. 
\begin{itemize}
\item For every $i < u$ such that $\langle d_\nu S_u, S_i \rangle=1$, there is a single $(i,l)$-handleslide arc with an endpoint at $p$.
\item For every $l < j$ such that $\langle d_\nu S_j, S_l \rangle=1$, there is a single $(u,j)$-handleslide arc with an endpoint at $p$.
\end{itemize} 
\item Let $p \in \Sigma_{\mathit{sing}}$ be the base projection of an upward swallowtail point, i.e. one where in the front projection the crossing arc appears below the third sheet that
that limits to the swallowtail point.  The two cusp arcs in $\Sigma_{\mathit{sing}}$ that meet at $p$ divide a small neighborhood $N \subset M$ of $p$ into halfs $N_-$ and $N_+$ above which $\Sigma$ has $(n-2)$ or $n$ sheets respectively.  Let $k$ be such that above $N_-$ (resp. above $N_+$) 
the sheets that border the swallowtail point are numbered $S_k$ (resp. $S_k$, $S_{k+1}$, $S_{k+2}$) with   
$S_{k+1}$ and $S_{k+2}$ meeting along the crossing arc, $c$, that ends at $p$ and is contained in $N_+$.   
Let $(V(R_0), d_0)$ be the complex associated to the region, $R_0$, that borders $p$ on the same side as $N_-$.
\begin{itemize} 
\item There are two $(k+1,k+2)$-handleslide arcs contained in $N_+$ and ending at $p$, one on each side of $c$.
\item For each $i<k$ such that  $\langle d_0S_k, S_i \rangle=1$ there is a $(i,k)$-handle slide arc in $N_+$ with endpoint at $p$.
\end{itemize}
When $p$ is a downward swallowtail, the requirement is symmetric.  The sheets $S_k$ in $N_-$ (resp. $S_k, S_{k+1}, S_{k+2}$ in $N_+$) correspond now to sheets $S_{l-2}$ in $N_-$ (resp. $S_{l}, S_{l-1}, S_{l-2}$).  There are now two $(l-2,l-1)$-handleslide arcs in $N_+$  on opposite sides of $c$ with endpoints at $p$, and for each $l-2< j$ with $\langle d_0 S_j, S_{l-2} \rangle=1$ there is an $(l, j+2)$-handleslide arc in $N_+$ with endpoint at $p$.
\end{enumerate}
\end{axiom}

The three types of endpoints for handleslide arcs are pictured in Figure \ref{fig:Endpoints}.

\begin{remark}
\begin{enumerate}
\item In Axiom \ref{ax:endpoints} (2), if the condition holds for some region $R_\nu$ bordering $p$, then it holds for every region that borders $p$.
\item In Axiom \ref{ax:endpoints} (3), in a neighborhood of a swallowtail point $p$, assume that a differential $d_0$ has been assigned to the region $R_0$ that borders $p$ in $N_-$.  Then, as long as handleslide arcs are placed in $N_+$ as specified by the axiom, there is always a unique way to assign differentials to regions in $N_+$ to produce an MCF near $p$.

\item See \cite[Section 6.1]{RuSu3} for a further discussion of constructions of MCFs.  
\end{enumerate}
\end{remark}

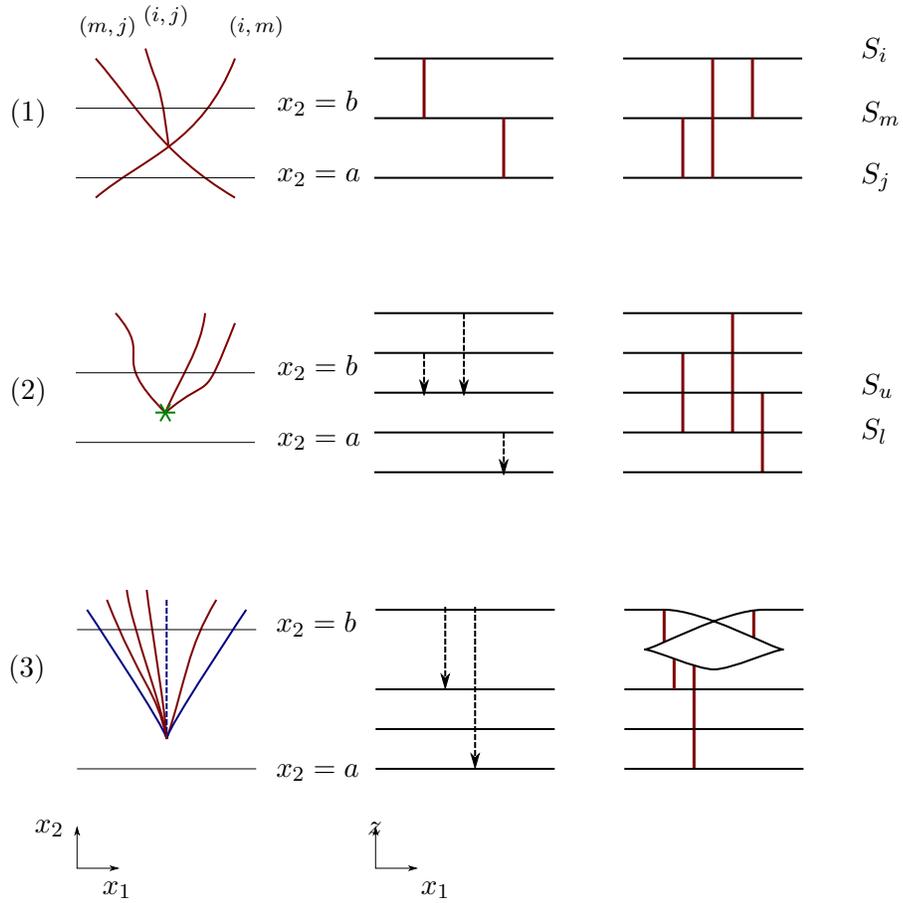
\begin{figure}

\quad

\quad

\labellist
\small
\pinlabel (1) [r] at -16 64
\tiny
\pinlabel $(m,j)$ [b] at 24 116
\pinlabel $(i,j)$ [b] at 68 124
\pinlabel $(i,m)$ [b] at 136 116
\small
\pinlabel $S_i$ [l] at 586 112
\pinlabel $S_m$ [l] at 586 64
\pinlabel $S_j$ [l] at 586 16
\pinlabel $x_2=a$ [l] at 146 16
\pinlabel $x_2=b$ [l] at 146 72
\endlabellist
\centerline{\includegraphics[scale=.5]{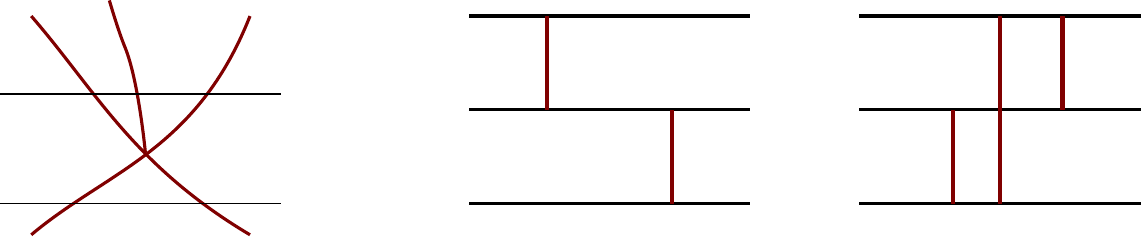}} 

\quad

\quad

\quad

\labellist
\small
\pinlabel (2) [r] at -16 64
\pinlabel $S_{u}$ [l] at 586 68
\pinlabel $S_{l}$ [l] at 586 34
\pinlabel $x_2=a$ [l] at 146 26
\pinlabel $x_2=b$ [l] at 146 82
\endlabellist

\centerline{\includegraphics[scale=.5]{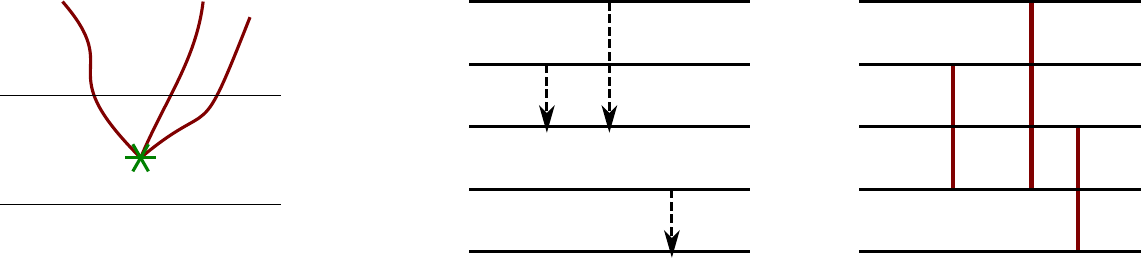}}

\quad

\quad

\quad

\labellist
\small
\pinlabel (3) [r] at -16 152
\pinlabel $x_1$ [t] at 32 -2
\pinlabel $x_2$ [r] at -2 32
\pinlabel $x_1$ [t] at 272 -2
\pinlabel $z$ [r] at 238 32
\pinlabel $x_2=a$ [l] at 146 74
\pinlabel $x_2=b$ [l] at 146 186

\endlabellist

\centerline{\includegraphics[scale=.5]{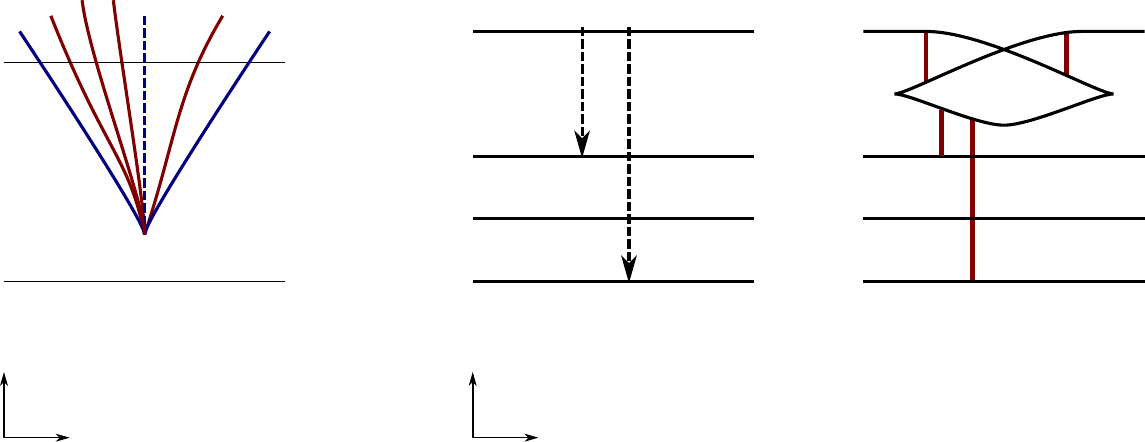}}

\caption{The three possible endpoints for handleslide arcs specified by Axiom \ref{ax:endpoints} (1)-(3).    The left column depicts the base projection (to $M$) of $H_0$ (in red), $H_{-1}$ (a green star) and $\Sigma_{\mathit{sing}}$, (in blue).  
The center and right column depict slices of the front projection, $\pi_{xz}(\Sigma)$, at $x_2=a$ and $x_2=b$ with dotted black arrows  from $S_i$ to $S_j$ used to indicate when $\langle d S_j, S_i \rangle = 1$.  In (3), a downward swallowtail point is pictured.
}   
\label{fig:Endpoints}
\end{figure}

Note that when $\Sigma$ and $M$ are $2$-dimensional and have boundary, an MCF $\mathcal{C}$ for $\Sigma$ restricts to an MCF for the $1$-dimensional Legendrian $\Sigma|_{\partial M} \subset J^1(\partial M)$.  

\begin{definition} \label{def:MCFequiv} Let $\mathcal{C}_0$ and $\mathcal{C}_1$ be two MCFs for a 1-dimensional Legendrian knot $\Lambda \subset J^1M$.  We say that $\mathcal{C}_0$ and $\mathcal{C}_1$ are {\bf equivalent} if there exists an MCF $\mathcal{C}$ for the product cobordism $j^1(1_{[0,1]}\cdot\Lambda)  \subset J^1([0,1]\times M )$ such that $\mathcal{C}|_{\{0\} \times M} = \mathcal{C}_0$ and $\mathcal{C}|_{\{1\} \times M} = \mathcal{C}_1$.  
\end{definition}

\begin{remark} \label{rem:6-13}
The statement of Definition \ref{def:MCFequiv} is slightly different than the definition of equivalence in \cite{Henry, HenryRu1, HenryRu2}.  These earlier works did not consider MCFs for $2$-dimensional Legendrians and instead present a collection of elementary moves on 1-dimensional MCFs with two MCFs $\mathcal{C}_0$ and $\mathcal{C}_1$ declared to be equivalent if they can be related by a sequence of elementary moves.  
 The two versions of the definitions are equivalent, since the elementary moves simply describe the bifurcations that occur along the slices, $\mathcal{C}|_{\{t\} \times M}$, of a generic $2$-dimensional MCF $\mathcal{C}$ for $j^1(1_{[0,1]}\cdot \Lambda)$ as $t$ increases from $0$ to $1$.  For instance, as numbered in \cite[Section 4.3]{HenryRu1},  Move 1 corresponds to passing a critical point of $t \circ H_0$; Moves 2-6 correspond to a transverse self intersection of $H_0$ with Move 6 demonstrating the handleslide endpoints required by Axiom \ref{ax:endpoints} (1); Moves 7, 8, and 10 are transverse intersections of $H_0$ with the base projection of a crossing arc of $j^1(1_{[0,1]}\cdot \Lambda)$; Moves 9 and 11 are transverse intersections of $H_0$ and the base projection of a cusp arc; and Move 15 corresponds to passing a super-handleslide point.  The Moves 12-14 do not occur in our case since we only consider MCFs with ``simple births and deaths'', and this does not affect the resulting equivalence classes.  Indeed, when two such MCFs with only ``simple births and deaths'' are related by a sequence of the Moves 1-15, possibly passing through some MCFs with non-simple births or deaths, we can always arrive at an alternate sequence of moves avoiding non-simple cusps by treating the implicit handleslides discussed in \cite{HenryRu1} as explicit handleslides located near a cusp.  See \cite[Proposition 3.17]{Henry}. 
\end{remark}

\subsection{MCFs and augmentations}  \label{sec:5-4}
The following proposition characterizes the (cellular) induced augmentation set of a compact Legendrian cobordism (see Section \ref{sec:AugSet} for the definition) in terms of MCFs.

\begin{proposition}  \label{prop:MCFcomp}   
\begin{enumerate} 
\item  For any $(\Lambda, \mathcal{E}, \mu)$ with $\Lambda \subset J^1M$ and $\dim \Lambda =1$,  there exists a canonical bijection 
%map $\Phi: \mathit{MCF}_\rho(\Lambda, \mathcal{E}) \rightarrow \CHD_\rho(\Lambda, \mathcal{E}) = \aug_\rho(\Lambda, \mathcal{E})$ from Proposition \ref{prop:surjective} induces a bijection
\[
\widehat{\Phi} :  \mathit{MCF}_\rho(\Lambda)/{\sim} \rightarrow  \aug_\rho(\Lambda, \mathcal{E})/{\sim}
\]
between equivalence classes of $\rho$-graded MCFs for $(\Lambda, \mu)$ and DGA homotopy classes of $\rho$-graded augmentations of the cellular DGA $\mathcal{A}(\Lambda, \mathcal{E})$.  

\item  Let $\Sigma \subset J^1([0,1]\times M)$ be a compact $\rho$-graded Legendrian cobordism from $(\Lambda_0, \mu_0, \mathcal{E}_0)$ to $(\Lambda_1, \mu_1, \mathcal{E}_1)$.
Using the canonical bijections $\widehat{\Phi}:\mathit{MCF}_\rho(\Lambda_i)/{\sim} \stackrel{\cong}{\rightarrow} \aug_\rho(\Lambda_i, \mathcal{E}_i)/{\sim}$, the induced augmentation set 
\[
I^\comp_\Sigma \subset \aug_\rho(\Lambda_0, \mathcal{E}_0)/{\sim} \times \aug_\rho(\Lambda_1, \mathcal{E}_1)/{\sim}
\]
 viewed as a subset 
\[
I^{\mathit{MCF}}_\Sigma \subset \mathit{MCF}_\rho(\Lambda_0, \mathcal{E}_0)/{\sim} \times \mathit{MCF}_\rho(\Lambda_1, \mathcal{E}_1)/{\sim}
\]
satisfies
\[
I^{\mathit{MCF}}_\Sigma = \left\{ ([\mathcal{C}|_{\Lambda_0}], [\mathcal{C}|_{\Lambda_1}]) \, \middle|\, \mathcal{C} \in \MCF_\rho(\Sigma) \right\}.  
\]
\end{enumerate}
\end{proposition}

\dr{A detailed proof of this, including review of the relevant parts of \cite{RuSu3}, appears in the preprint version of this article; see  \cite[Proposition 5.17 and 5.19]{PanRu2Arxiv}.  We provide here a sketch.}

\begin{proof}[Sketch of proof]   \dr{Both (1) and (2) are consequences of a straightforward extension of the correspondences between MCFs and augmentations from \cite[Propositions 5.5 and 6.4]{RuSu3} for closed surfaces to the case of $2$-dimensional cobordisms.  These results produce augmentations of $\mathcal{A}(\Lambda, \mathcal{E})$ (that are viewed in \cite{RuSu3} as {\it Chain Homotopy Diagrams}) from MCFs that are suitably transverse to $\mathcal{E}$ (aka {\it nice} MCFs), and vice-versa.}  

\dr{In particular, in the $1$-dimensional case we have a map $\Phi$ from nice MCFs to augmentations of $\mathcal{A}(\Lambda, \mathcal{E})$.  That $\Phi$ induces a bijection $\widehat{\Phi}$ as in (1) is seen from applying the MCF/augmentation correspondences to the product cobordism $j^1(1_{[0,1]}\cdot \Lambda)$.  Indeed, the above Proposition \ref{prop:alg01} shows that two augmentations of $\mathcal{A}(\Lambda, \mathcal{E})$ are homotopic if and only if they extend to an augmentation of $j^1(1_{[0,1]}\cdot \Lambda)$, and this characterization matches the definition of equivalence for MCFs.  The second statement (2) follows since in the cobordism case the correspondences can be made to commute with the restrictions to the boundary components.}

%Moreover, the map from MCFs to augmentations is compatible with restricting MCFs or augmentations to the $1$-skeleton.  agrees with constructions can be made inverse

%(that are viewed in \cite{RuSu3} as certain {\it Chain Homotopy Diagrams})  $\MCFs   
%In \cite{RuSu3}, augmentations of the cellular DGA are viewed as {\it Chain Homotopy Diagrams} that are certain diagrams of chain complexes 

%Two augmentations are homotopic if and only if they extend to an augmentation on ???.  

%by viewing the values of an augmentation 

%The bijection between augmentations

%In the $1$-dimensional case, \cite[Section 5]{RuSu3} constructs a map $\Phi$ from MCFs satisfying a mild transversality condition with respect $\mathcal{E}$ to augmentations of $\mathcal{A}(\Lambda, \mathcal{E})$. Constructions in the $2$-dimensional case of MCFs from augmentations

%Augmentations $\rightleftarrow$ CHDs holds for cobordisms in such a way that pulling back under $i_k$ maps corresponds to restricting the CHD to $\Lambda_k$, for $k=0,1$.  In the $1$-dimensional case, \cite[Section 5]{RuSu3} constructs a map from MCFs satisfying a mild transversality condition with respect $\mathcal{E}$ to augmentations of $\mathcal{A}(\Lambda, \mathcal{E})$.   $1$-skeleton o,  Using the computation of the cellular DGA for the product cobordism $j^1()$ from Proposition, allows shows to see that the map is a bijection on equivalence classes.

%For (1), a map from 
%This is an extension of the correspondence between MCFs and augmentations for closed surfaces from \cite{RuSu3} to the cobordism case...  
\end{proof}

\subsection{$A$-form MCFs}\label{sec:Aform} From
Proposition \ref{prop:MCFcomp} and Corollary \ref{cor:LegInv}, we see that $I^{\mathit{MCF}}_\Sigma$ is invariant of $\Sigma$ up to Legendrian isotopy rel. $\partial \Sigma$.  Moreover, $I^{\mathit{MCF}}_\Sigma$  is equivalent in a canonical way to both the cellular induced augmentation set $I^{\mathit{cell}}_\Sigma$ and, via Corollary \ref{cor:augmentationset}, to the induced augmentation set
$I_{\Sigma}$ of the conical version of $\Sigma$, $G_{\mathit{std}}(\Sigma)$.  
 For computing $I^{\mathit{MCF}}_\Sigma$ in explicit examples, it is convenient to have an efficient characterization of the sets, $\mathit{MCF}_\rho(\Lambda)/{\sim}$, 
not requiring the cellular DGA which often has a large number of generators.   
For Legendrian links in $J^1\R$, work of Henry \cite{Henry, HenryRu2}, which we now review,  gives an explicit bijection $\MCF_\rho(\Lambda)/{\sim} \leftrightarrow \aug_\rho(\Lambda_\res)/{\sim}$ where $\Lambda_\res$ is the Ng resolution of $\Lambda$.  Recall that the DGA $\mathcal{A}(\Lambda_\res)$ has generators in bijection with the crossings and right cusps of the front projection of $\Lambda$; see \cite[Section 2]{NgCompute}.  

\begin{definition}[\cite{Henry}] \label{def:Aform}
A $\rho$-graded MCF $\mathcal{C}$ for a Legendrian link $\Lambda \subset J^1\R$ is in {\bf $A$-form} if its handleslides occur only in the following locations:
\begin{enumerate}
\item Handleslides connecting the two crossing strands appear immediately to the left of {\it some subset} of the crossings of $\pi_{xz}(\Lambda)$.
\item In the case $\rho=1$, handleslides connecting the two strands that meet at a right cusp appear immediately to the left of {\it some subset} of the right cusps of $\pi_{xz}(\Lambda)$.
\item Near a right cusp,  any number of additional handleslides may appear connecting the upper (resp. lower) strand of the right cusp to a strand above (resp. below) the right cusp.  These handleslides are located to the right of the handleslide from (2), if it exists.
\end{enumerate} 
See Figure \ref{fig:Aform}.
\end{definition}

\labellist
\small
\endlabellist

\begin{figure}
\centerline{ \includegraphics[scale=.6]{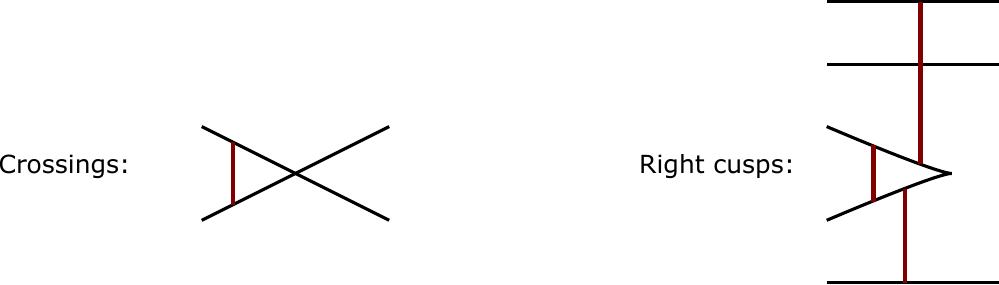} }
\caption{ The allowed locations for handleslides in an $A$-form MCF. }
\label{fig:Aform}
\end{figure}

\begin{remark}
The definition appears slightly different from that found in \cite{Henry, HenryRu2} as the handleslides from (3) are implicit handleslides in these references; see Remarks \ref{rem:6-8} (1) and \ref{rem:6-13}. Moreover, an $A$-form MCS is uniquely determined by its handleslides of type (1) and (2), since the handleslides of type (3) are determined by the form of the complexes $(R_\nu, d_\nu)$ near right cusps.
\end{remark}

Given an $A$-form MCF, $\mathcal{C}$, for $\Lambda \subset J^1\R$ we can define an algebra homomorphism
\[
\epsilon_{\mathcal{C}}: \mathcal{A}(\Lambda_\mathit{res}) \rightarrow \Z/2
\]
where for each crossing or right cusp $b_i \in \mathcal{A}(\Lambda_\mathit{res})$ we set $\epsilon_\mathcal{C}(b_i)$ to be $1$ (resp. $0$) if a handleslide connecting the two strands of the crossing or cusp appears (resp. does not appear) to the left of $b_i$.

\begin{proposition}[\cite{Henry, HenryRu2}] \label{prop:Aform}
For any $\Lambda \subset J^1\R$, the correspondence $\mathcal{C} \mapsto \epsilon_{\mathcal{C}}$ defines a surjective map from the set of $\rho$-graded $A$-form MCFs for $\Lambda$ to $\aug_\rho(\Lambda_\res)$.  Moreover, this induces a bijection
\[
\mathit{MCF}_\rho(\Lambda)/{\sim} \stackrel{\cong}{\rightarrow} \aug_\rho(\Lambda_\res)/{\sim}.
\]
\end{proposition}
\begin{proof}
In \cite{Henry}, a map $\widehat{\Psi}:\mathit{MCF}_\rho(\Lambda)/{\sim} \rightarrow \aug_\rho(\Lambda_\res)/{\sim}$ is defined and shown to be surjective.  Moreover, this map has the property that when $\mathcal{C}$ is an $A$-form MCF we have $\widehat{\Psi}([\mathcal{C}]) = [\epsilon_\mathcal{C}]$, and it is shown that every MCF is equivalent to an $A$-form MCF.  Finally, in \cite{HenryRu2} it is shown that $\widehat{\Psi}$ is injective. 
\end{proof}

\section{Examples}  \label{sec:example}

In this section, we apply the methods developed in Section \ref{sec:MCF} to compute induced augmentations and augmentation sets for some particular Legendrian fillings.  In Section \ref{sec:6-1}, we give examples of oriented Legendrian fillings inducing augmentations that cannot be induced by any oriented 
embedded Lagrangian fillings.  In Section \ref{sec:6-2}, we prove Theorem \ref{thm:1-4} by constructing $2n$ Legendrian fillings of the Legendrian torus knot, $T(2,2n+1)$, each with a single degree $0$ Reeb chord, and distinguished by their induced augmentation sets.

\subsection{Initial examples}  \label{sec:6-1}

Before launching into examples, we summarize the method that we use for computing induced augmentations and induced augmentation sets arising from a (compact) Legendrian filling, $\Sigma$.  For a $1$-dimensional Legendrian, $\Lambda$, with admissible transverse decomposition $\mathcal{E}_\pitchfork$, Proposition \ref{prop:MCFcomp} and Corollary \ref{cor:augmentationset} provide a commutative diagram of bijections and inclusions
\begin{equation}  \label{eq:}
\begin{array}{ccccc}
\mathit{MCF}_\rho(\Lambda)/{\sim} & \cong & \mathit{Aug}_\rho(\Lambda, \mathcal{E}_{||})/{\sim} & \cong & \mathit{Aug}_\rho(\widetilde{\Lambda})/{\sim} \\  
\cup & & \cup & & \cup \\
I^\mathit{MCF}_\Sigma  & \cong & I^\comp_\Sigma & \cong & I_{\widetilde{\Sigma}}
\end{array}
\end{equation} 
where, in the top row, from left to right the sets are the equivalence classes of MCFs of $\Lambda$; homotopy classes of augmentations of the cellular DGA of $(\Lambda, \mathcal{E}_{||})$; and homotopy classes of augmentations of the LCH DGA $\mathcal{A}(\widetilde{\Lambda})$ where $\widetilde{\Lambda}$ is Legendrian isotopic to $\Lambda$.  (It is the standard geometric model from Section \ref{sec:4-5}.)  The bottom row consists of the 
(equivalent) versions of the induced augmentation set in these three different settings (where $\widetilde{\Sigma}= G_{\mathit{std}}(\Sigma)$ is $\Sigma$ with its positive end extended to become a conical cobordism).  
The MCFs in $I^\mathit{MCF}_\Sigma$ are those that arise from restricting a $2$-dimensional MCF, $\mathcal{C}_\Sigma$, for $\Sigma$ to $\Lambda$.   
Thus, elements of $I^\mathit{MCF}_\Sigma$ (and so also induced augmentations in $I_{\widetilde{\Sigma}}$) may be constructed by specifying a $2$-dimensional MCF for $\Sigma$, and then observing its restriction to $\Lambda$.  
To keep track of the different MCFs of $\Lambda$ we use the further bijection from Proposition \ref{prop:Aform} with augmentations of the Ng resolution,
\[
\mathit{MCF}_\rho(\Lambda)/{\sim} \, \cong \, \mathit{Aug}(\Lambda_{res})/{\sim}
\]
following Section \ref{sec:Aform}.
That is, once a $2$-dimensional MCF has been restricted to the Legendrian boundary $\Lambda$, if necessary, we transform it by an equivalence into $A$-form to arrive at 
an augmentation of $\Lambda_\mathit{res}$.
 
For the rest of the section, all of the augmentations we talk about are of $\Lambda_\mathit{res}$, and we will refer to Reeb chords of $\Lambda_\mathit{res}$ as crossings and right cusps of the front projection.

\begin{example} An augmentation that can be induced by a Legendrian filling
for a knot that does not have any oriented embedded Lagrangian fillings.
\end{example}

The work \cite[Section 8.3]{PanRu1} gives an infinite family of such knots without explicitly computing the induced augmentations. One of them is  the knot $4_1$ shown in Figure \ref{fig:4_1}
where a conical Legendrian filling is constructed by a clasp move  (see Figure \ref{fig:PinchMoves} below for a discussion of local moves for constructing Legendrian cobordisms) followed by Reidemeister I moves and an unknot move; the pictures are slices of the front projection of the corresponding compact filling, $\Sigma \subset J^1([0,1]\times M)$, with $t \in [0,1]$ constant.  As shown in \cite[Section 8.3]{PanRu1}, the clasp move can be realized by a conical Legendrian cobordism with a single Reeb chord, so that a conical version of $\Sigma$ (after an appropriate Legendrian isotopy) is a disk with a single double point in  the Lagrangian projection.
Figure \ref{fig:4_1} indicates the MCF for $\Sigma$ as a  movie of handleslides.  Note that the handleslides that appear after the Reidemeister I moves are as required in Axiom \ref{ax:endpoints} (3).    
Using the correspondence between A-form  MCFs and augmentations of $\Lambda_\mathit{res}$, the induced augmentation $\e_1$ sends
the crossings $b_1$ and $b_2$ to $1$ and all other Reeb chords to $0$.
Note that a $\Z$-valued Maslov potential for $4_1$ extends (uniquely) over $\Sigma$, and the MCF for $\Sigma$ is $0$-graded.   Correspondingly, $\epsilon_1$ is the unique
$0$-graded augmentation of $\Lambda_\mathit{res}$.
As discussed in \cite[Section 8.3]{PanRu1}, the $4_1$ knot does not 
have  any oriented embedded Lagrangian fillings due to a restriction on the Thurston-Bennequin number.

\begin{figure}[!ht]
\labellist
\small
\pinlabel $b_1$ at 95 45
\pinlabel $b_2$ at 95 63 
\pinlabel $\downarrow$ at 175  65
\pinlabel $\downarrow$ at 175  -5
\pinlabel $\emptyset$ at 175 -20
\endlabellist
\includegraphics[width=3in]{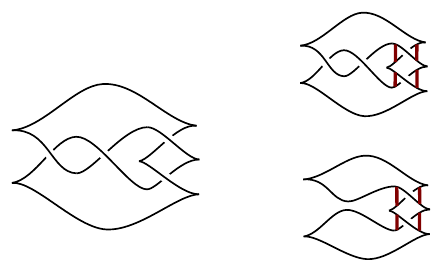}
\vspace{0.2in}

\caption{A conical Legendrian filling of Legendrian $4_1$ constructed by a movie. The arrow indicates the negative $t$ direction. Red vertical line segments indicates handle slides in the MCF.}
\label{fig:4_1}
\end{figure}
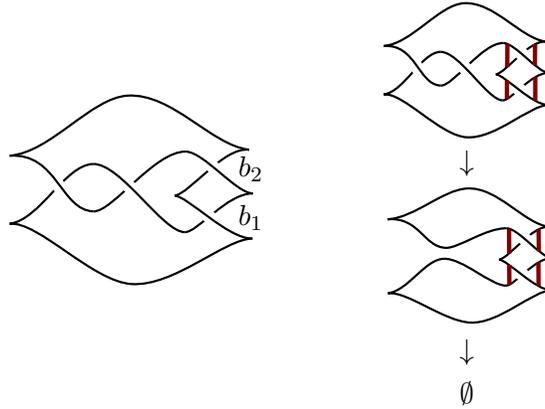

We note that $4_1$ does admit {\it non-orientable} embedded Lagrangian fillings.  Such a filling, constructed via applying two pinch moves (see Figure  \ref{fig:PinchMoves}) and then an unknot move, is pictured in Figure \ref{fig:4_1NonOrient}.  It induces the $1$-graded augmentation $\e_2$ that sends  $b_1, b_2, a$, and $c$ to $1$ and other Reeb chords to $0$.

\begin{question}  Can the augmentation $\e_1$ induced by the Legendrian filling from Figure \ref{fig:4_1} also be induced by a non-orientable embedded Lagrangian filling?
\end{question}
We suspect that the answer is no, but we do not know of any obstructions to such a non-orientable filling that are currently available in the literature.

\begin{figure}[!ht]
\labellist
\small
\pinlabel $a$ [t] at 18 32
\pinlabel $c$ [l] at 90 32 
\pinlabel $\emptyset$ at 268 36
\endlabellist
\includegraphics[width=3in]{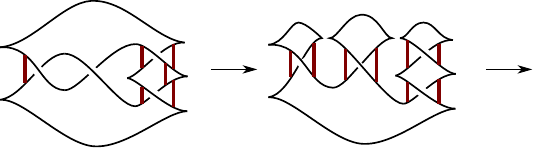}
\vspace{0.2in}

\caption{A non-orientable conical Legendrian filling of Legendrian $4_1$ with embedded Lagrangian projection.  A $1$-graded MCF inducing the augmentation $\e_2$ is pictured.}
\label{fig:4_1NonOrient}
\end{figure}

\begin{example} Two augmentations for the same knot, one can be induced by an embedded Lagrangian filling and one cannot.
\end{example}

Figure \ref{fig:8_21} shows a Legendrian of knot type
 $m(8_{21})$.
  The knot has a genus one embedded Lagrangian filling constructed by doing three pinch moves (see Figure \ref{fig:PinchMoves}) and closing up the two disjoint unknots
(see the first row in Figure \ref{fig:8_21}).
The associated MCF induces the augmentation $\epsilon_1$ that sends $b_1, b_4, b_5$ to $1$ and others to $0$.
 
 \begin{figure}[!ht]
\labellist
\small
\pinlabel $b_1$ at 35 85
\pinlabel $b_2$ at 70 90
\pinlabel $b_3$ at  115 88
\pinlabel $b_4$ at 35 35
\pinlabel $b_5$ at  73 37
\pinlabel $b_6$ at 113 37

\pinlabel $b_1$ at 35 248
\pinlabel $b_2$ at 70 253
\pinlabel $b_3$ at  115 253
\pinlabel $b_4$ at 35 200
\pinlabel $b_5$ at  73 198
\pinlabel $b_6$ at 113 198

\pinlabel $\emptyset$ at 600 60
\pinlabel $\emptyset$ at 600 220
\endlabellist
\hspace{-1in}\includegraphics[width=5in]{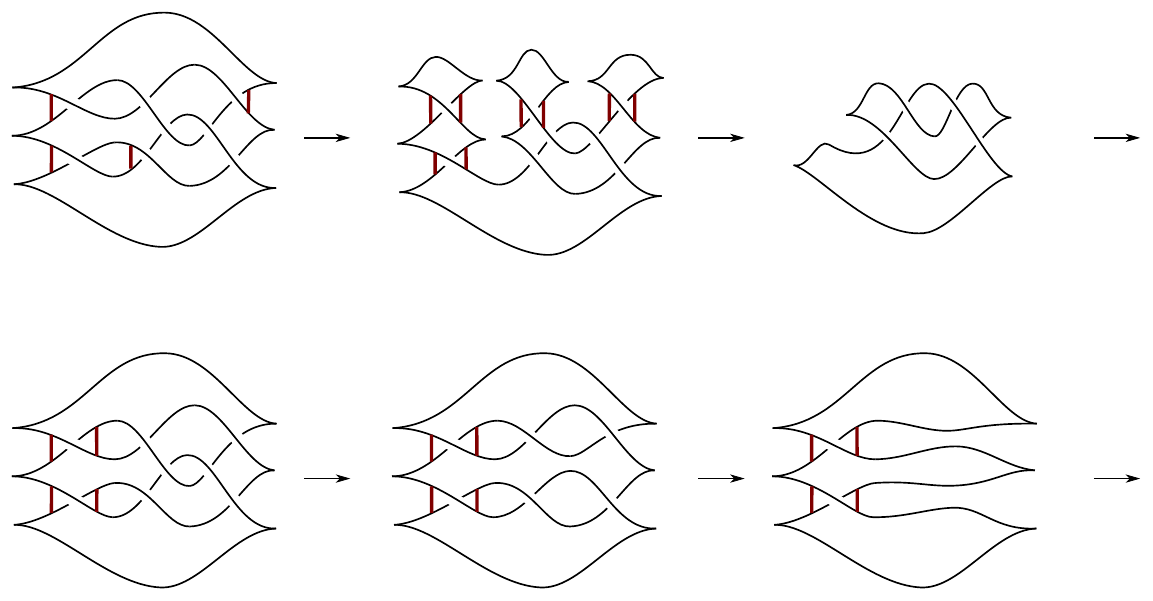}
\caption{The first row gives an embedded Lagrangian filling of $m(8_{21})$ constructed by a movie. The second row gives a conical Legendrian filling of $m(8_{21})$ through a movie. The arrow indicates negative $t$ direction.}
\label{fig:8_21}
\end{figure}
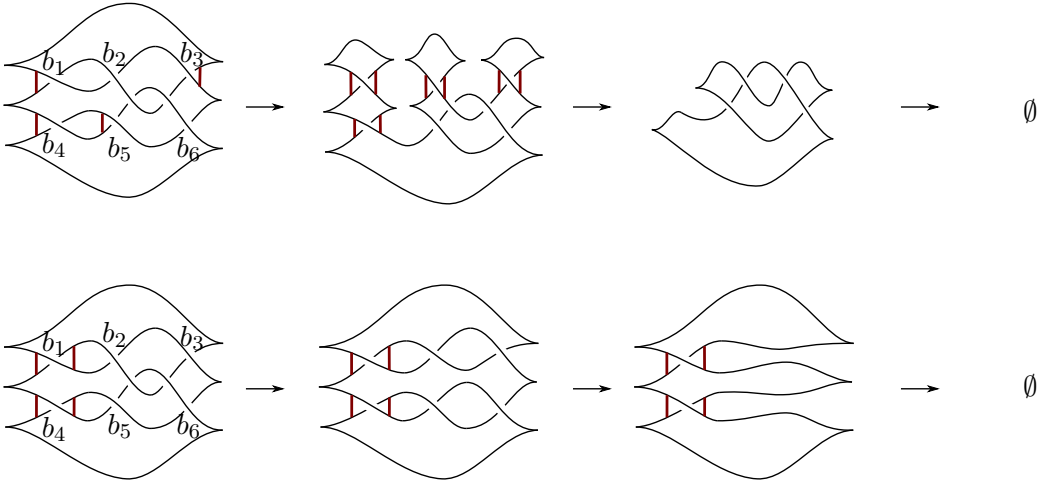

 One the other hand, one can construct a conical Legendrian filling through doing three clasp moves and then close up the unknot as shown in the movie in the second row of Figure \ref{fig:8_21}.
 This is a Legendrian disk with three Reeb chords. 
 The associated MCF as  shown in Figure \ref{fig:8_21} induces the augmentation
 $\e_2$ that sends $b_1,b_2,b_4$ and $b_5$ to $1$ and other Reeb chords to $0$.
 This augmentation cannot be induced by an embedded orientable Lagrangian filling since its linearized homology has Poincare polynomial
$t^{-1}+4+2t$.
 By Seidel's Isomorphism \cite{Ekh, DR}, if an augmentation was induced by an embedded orientable genus $g$ filling, using the $\Z/2$-grading on linearized homology arising from a $\Z/2$-valued Maslov potential on the filling (such a Maslov potential exists by orientability, see Remark \ref{rem:orient}), its Poincare polynomial with $\Z/2$-grading should be $t+2g$.

\subsection{Proof of Theorem \ref{thm:1-4}} \label{sec:6-2}

The theorem provides examples of knots with an arbitrary  (finite) number of immersed fillings, all having isomorphic DGAs,  
but distinguished by induced augmentation sets.

\medskip

\noindent {\bf Theorem \ref{thm:1-4}.}
{\it For each $n \geq 1$, there exists $2n$ distinct conical Legendrian fillings, $\Sigma_1, \ldots, \Sigma_{2n}$ of the max-$\mathit{tb}$ Legendrian torus knot $T(2, 2n+1)$
such that 
\begin{itemize}
\item[(i)] the $\Sigma_i$ are all orientable with genus $n-1$ and have $\Z$-valued Maslov potentials,
\item[(ii)] each $\Sigma_i$ has a single Reeb chord of degree $0$,
\item[(iii)] and the induced augmentation sets satisfy $I_{\Sigma_i} \neq I_{\Sigma_j}$ when $i \neq j$.
\end{itemize}
}

\medskip

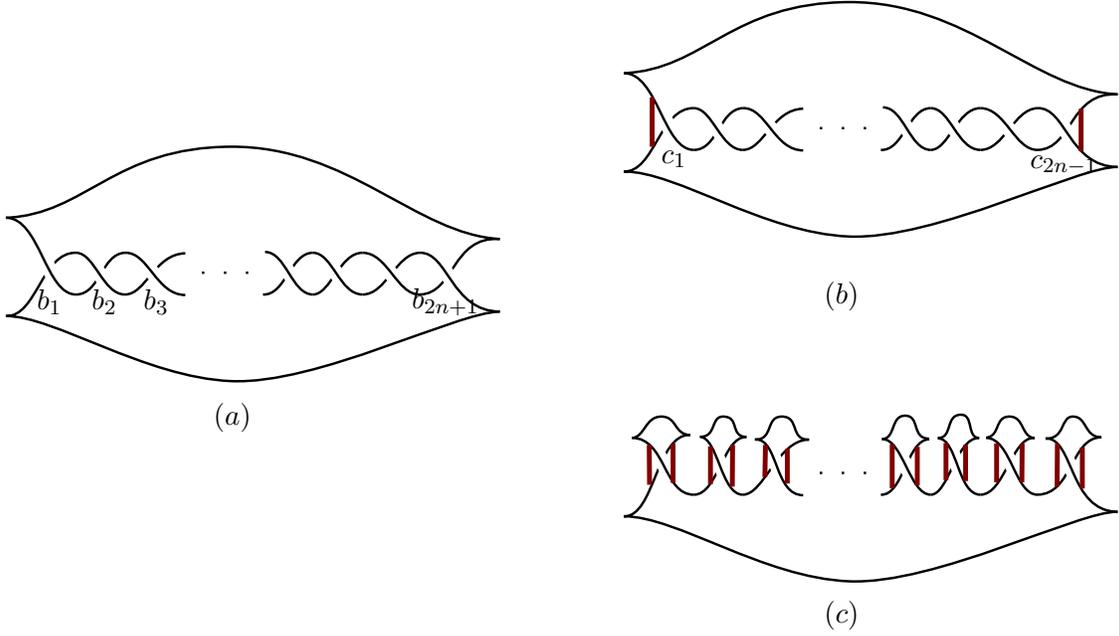
\begin{figure}[!ht]
\labellist
\small
\pinlabel $b_1$ at  20 98
\pinlabel $b_2$ at  38 98
\pinlabel $b_3$ at 55 98
\pinlabel $b_{2n+1}$ at 150 98
\pinlabel $(a)$ at 80 60
\pinlabel $(b)$ at  280 100
\pinlabel $(c)$ at 280 -5
\pinlabel $c_1$ at 225 145
\pinlabel $c_{2n-1}$ at 353 143
\endlabellist
\includegraphics[width=6in]{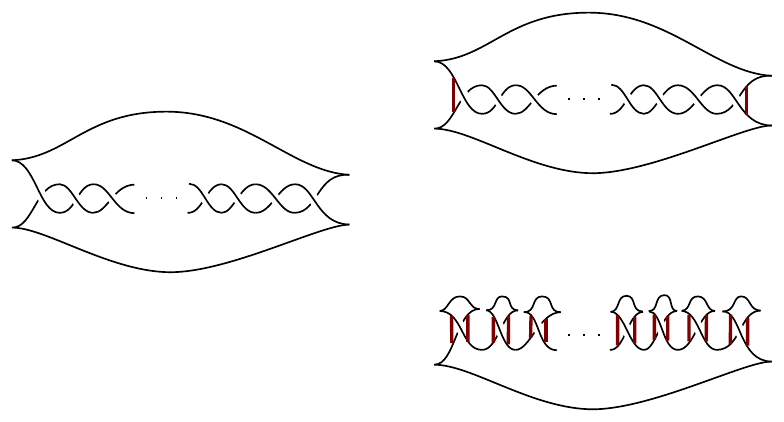}
\vspace{0.1in}

\caption{Part $(a)$ is a Legendrian $(2,2n+1)$ torus knot and $(b)$ is a Legendrian $(2, 2n-1)$ torus knot.  Part $(c)$ is the unknot got after doing $2n-2$ pinch moves to $T(2, 2n-1)$.}
\label{fig:torus}
\end{figure}

Note that each of the DGAs $\mathcal{A}(\Sigma_i)$ is $\Z$-graded with a single generator in degree $0$ and the differential is necessarily zero for grading reasons.
Thus, they are all isomorphic.

\begin{proof}

Consider the Legendrian torus knot $T(2, 2n+1)$ as shown in Figure \ref{fig:torus} part $(a)$. There are $2n$ ``eye shapes'' that one can do a clasp move on to get a $(2,2n-1)$ Legendrian torus knot as shown in part $(b)$.
One can then do $2n-2$ pinch moves to the $T(2,2n-1)$ knot and close the unknot in part $(c)$.
Let $\Sigma_i$ be the conical Legendrian filling of $T(2,2n+1)$ that is constructed by performing the first clasp move at the ``eye shape" between $b_i$ and $b_{i+1}$.
According to the MCF shown in Figure \ref{fig:torus} part $(b)$ and $(c)$, the pinching procedure on $T(2,2n-1)$ induces an augmentation of $T(2,2n-1)$ that only sends $c_1$ to $1$ and all others to $0$.

The clasp move produces a degree $0$ Reeb chord, $a$, (see Proposition 8.3 of \cite{PanRu1} for the degree computation) of the resulting conical Legendrian surface $\Sigma_i$, and the rest of the cobordism can be constructed without Reeb chords. Thus, $\alg(\Sigma_i)$ has  two augmentations, $a \mapsto 0$ and $a \mapsto 1$, and it follows that $I_{\Sigma_{i}}$ can have at most two elements.  In fact, $I_{\Sigma_{i}}$ does have two elements induced by the two  MCFs of the conical Legendrian surface shown in Figure \ref{fig:claspMCF}.  
For a subset $I$ of the set $\{1, \cdots 2n+1\}$, denote by $\e_{I}$ the algebra map from $\alg(T(2,2n+1))$ to $\Z/2$ that only sends $\{b_i| i\in I\}$ to $1$ and all other Reeb chords to $0$.
Then, Figure \ref{fig:claspMCF} shows $$I_{\Sigma_i}=\begin{cases}\{\e_{\{1\}}, \e_{\{3\}}\}, & \mbox{ if }i=1\\
\{\e_{\{1\}}, \e_{\{1,i, i+2\}}\}, & \mbox{ if }1<i<2n\\
\{\e_{\{1\}}, \e_{\{1, 2n\}}\}, & \mbox{ if }i=2n.
\end{cases}$$ 

\begin{figure}[!ht]
\labellist
\small
\pinlabel $(a)$ at 25 -5
\pinlabel $(b)$ at 105 -5
\pinlabel $(c)$ at 180 -5

\pinlabel $b_1$ at 13 111
\pinlabel $b_2$ at 33 111
\pinlabel $b_3$ at 43 111
\pinlabel $c_1$ at  43 65

\pinlabel $b_1$ at 13 47
\pinlabel $b_2$ at 33 47
\pinlabel $b_3$ at 43 47
\pinlabel $c_1$ at  43 3

\pinlabel $b_1$ at 88 111
\pinlabel $b_i$ at 95 111
\pinlabel $b_{i+1}$ at 115 111
\pinlabel $b_{i+2}$ at 125 111
\pinlabel $c_1$ at  88 65
\pinlabel $c_{i}$ at  125 65

\pinlabel $b_1$ at 88 47
\pinlabel $b_i$ at 95 47
\pinlabel $b_{i+1}$ at 115 47
\pinlabel $b_{i+2}$ at 125 47
\pinlabel $c_1$ at  88 3
\pinlabel $c_{i}$ at  125 3

\pinlabel $b_1$ at 160 111
\pinlabel $b_{2n-1}$ at 168 111
\pinlabel $b_{2n}$ at 177 111
\pinlabel $b_{2n+1}$ at 200 111
\pinlabel $c_1$ at  162 65
\pinlabel $c_{2n-1}$ at  172 65

\pinlabel $b_1$ at 160 47
\pinlabel $b_{2n-1}$ at 168 47
\pinlabel $b_{2n}$ at 177 47
\pinlabel $b_{2n+1}$ at 200 47
\pinlabel $c_1$ at  162 3
\pinlabel $c_{2n-1}$ at  172 3

\endlabellist
\includegraphics[width=6in]{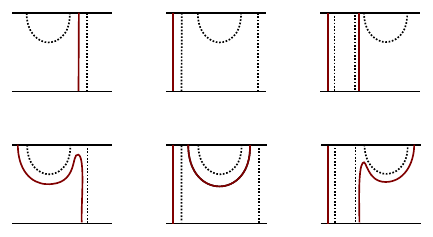}
\vspace{0.1in}

\caption{Two MCFs for the clasp move constructing $\Sigma_i$ pictured in the base projection.  These MCFs extend the MCF from Figure \ref{fig:torus} (b) and (c) to all of $\Sigma_i$.  Columns $(a), (b),(c)$ are for $i=1$, $1<i<2n$ and $i=2n$, respectively. The black dotted lines represent double points (in the front projection) and red lines represent handleslides.  }
\label{fig:claspMCF}
\end{figure}
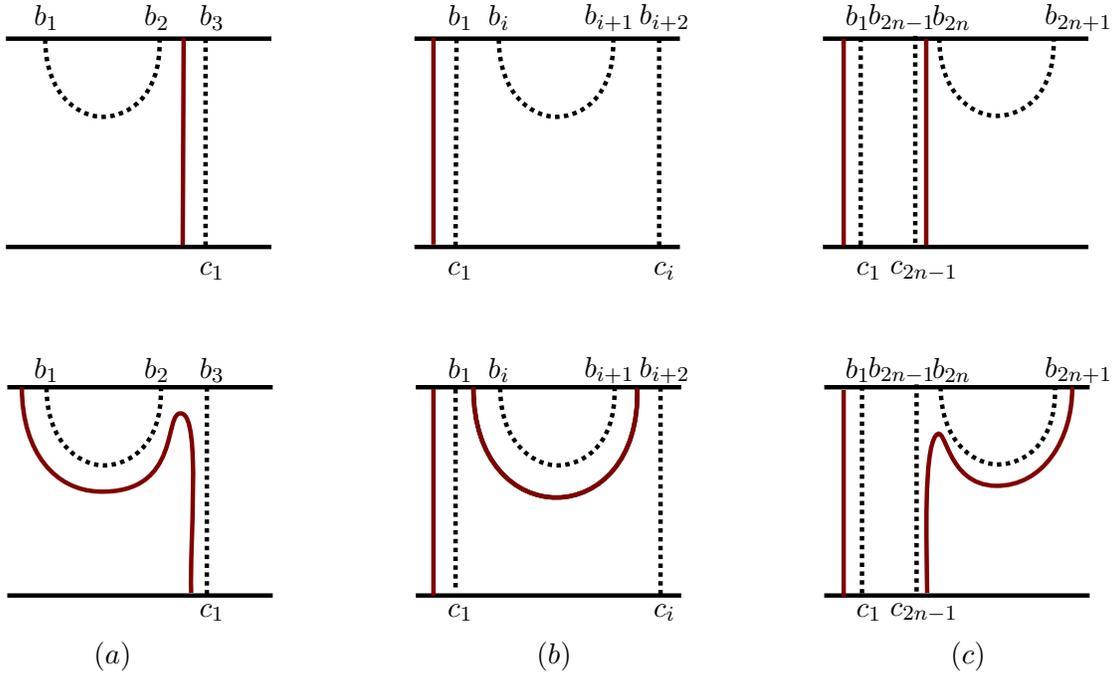

\end{proof}
\section{Every augmentation is induced by an immersed filling}
\label{sec:aug}

With the characterization of induced augmentation sets for Legendrian cobordisms in terms of MCFs now in place, the present section establishes Theorem \ref{thm:main} from the introduction.  For convenience, we repeat the statement. 

\medskip

\noindent {\bf Theorem 1.2.} {\it
Let $\Lambda \subset J^1\R$ have the $\Z/\rho$-valued Maslov potential $\mu$ where $\rho \geq 0$, and  let $\epsilon: \mathcal{A}(\Lambda) \rightarrow \Z/2$ be any $\rho$-graded augmentation. } 
\begin{enumerate}
\item {\it If $\rho \neq 1$, there exists a conical Legendrian filling $\Sigma$  of $\Lambda$ with $\Z/\rho$-valued Maslov potential extending $\mu$ together with a $\rho$-graded augmentation $\alpha:\alg(\Sigma) \rightarrow \Z/2$ such that $\epsilon \simeq \epsilon_{(\Sigma, \alpha)}$.  Moreover, if $\rho$ is even, then $\Sigma$ is orientable.}
\item {\it If $\rho =1$, then there exists a conical Legendrian cobordism $\Sigma: U \rightarrow \Lambda$ where $U$ is the standard Legendrian unknot with $\mathit{tb}(U) = -1$ together with an augmentation $\alpha:\alg(\Sigma) \rightarrow \Z/2$ such that $\epsilon \simeq f^*_\Sigma \alpha$.}
\end{enumerate}

\medskip

\begin{proof}[Proof of Theorem \ref{thm:main}]  Note that the orientability of $\Sigma$ when $\rho$ is even follows from the existence of the $\Z/\rho$-valued Maslov potential on $\Sigma$.  See Remark \ref{rem:orient}.

In view of Corollary \ref{cor:LegIsotopyAug}, if $\Lambda_1$ and $\Lambda_2$ are Legendrian isotopic and the statement holds for $\Lambda_1$, then it also holds for $\Lambda_2$.  Therefore, we can replace $\Lambda$ with its standard geometric model $G_{\mathit{std}}(\Lambda, \mathcal{E}_\pitchfork)$ with respect to some admissible transverse decomposition (see Section \ref{sec:4-5}), so that  Corollary \ref{cor:augmentationset} and Proposition \ref{prop:MCFcomp} apply to reduce Theorem \ref{thm:main} to \dr{Proposition \ref{prop:MCSextend} below}.

\end{proof}

\begin{proposition} \label{prop:MCSextend} Let  $\mathcal{C}_\Lambda$ be any $\rho$-graded MCF for $\Lambda \subset J^1\R$ with respect to the Maslov potential $\mu$.  
\begin{enumerate}
\item  If $\rho \neq 1$, then there exists a compact Legendrian filling, $\Sigma \subset J^1([0,1] \times \R )$, of $\Lambda$ with a Maslov potential extending $\mu$ such that  $\mathcal{C}_\Lambda$ can be extended to a $\rho$-graded MCF, $\mathcal{C}_\Sigma$, on $\Sigma$.  
\item If $\rho =1$, then the same statement holds except that $\Sigma$ is a compact Legendrian cobordism from the Legendrian unknot $U$ to $\Lambda$.
\end{enumerate}
\end{proposition}

We will prove Proposition \ref{prop:MCSextend} at the end of this section after establishing some preliminaries about extending MCFs over various elementary cobordisms.

\subsection{Normal rulings and $SR$-form MCFs}

In \cite{Henry}, Henry introduced a class of MCFs (the ``$SR$-form MCFs'') for $1$-dimensional Legendrian links that are in a clear many-to-one correspondence with the normal rulings of $\Lambda$.  Moreover, Henry proved that any MCF is equivalent to an $SR$-form MCF.  It will be convenient to make use of this result in proving Proposition \ref{prop:MCSextend}, so we briefly discuss the relevant definitions.

See any of \cite{HenryRu1, Fuchs, Sabloff} for a detailed definition of {\bf normal rulings} for Legendrian links in $J^1\R$.  Here, we recall that 
a normal ruling,  $\sigma$, for $\Lambda \subset J^1 \R$ (where $\Lambda$ has generic front and base projections) may be viewed as a continuous family of fixed point free involutions, $\sigma_{x_0}: \Lambda \cap \{x = x_0\} \rightarrow \Lambda \cap \{x = x_0\}$, defined for $x_0 \in \R \setminus \Lambda_{\mathit{sing}}$, i.e. for all values of $x_0$ such that $\pi_{xz}(\Lambda)$ has no crossing or cusp along the vertical line $x=x_0$.  For each connected component $R_\nu \subset \R \setminus \Lambda_{\mathit{sing}}$ there is a resulting involution on the sheets of $\Lambda$ above $R_\nu$, notated simply as 
\[
\sigma: \{S^\nu_i\} \rightarrow \{S^\nu_i\}, 
\]
which partitions $\{S^\nu_i\}$ into a collection of disjoint pairs (because of the fixed point free condition). 
  Moreover, if two regions $R_i$ and $R_{i+1}$ share a border at a cusp or crossing, then the pairings of sheets above $R_i$ and $R_{i+1}$ are required to be related in a standard way.  In particular, sheets that meet at a cusp point are paired near the cusp point, and every crossing of $\pi_{xz}(\Lambda)$ is either a {\bf departure}, a {\bf return}, or a {\bf switch} for $\sigma$ as in Figure \ref{fig:Rulings}.
The {\bf degree} of a crossing of $\pi_{xz}(\Lambda)$ with respect to a $\Z/\rho$-valued Maslov potential $\mu$ is defined to be the difference $\mu(u) - \mu(l) \in \Z/\rho$ where $u$ and $l$ denote the upper and lower strands at the left side of the crossing.  
We say that $\sigma$ is {\bf $\rho$-graded} with respect to $\mu$ if all switches of $\sigma$ have degree $0$ mod $\rho$. 
 On each region $R_\nu \subset \R \setminus \Lambda_{\mathit{sing}}$, we can use $\sigma$ to define a {\bf standard ruling differential} 
\begin{equation}  \label{eq:standardruling}
d^\sigma_\nu:V(R_\nu)  \rightarrow V(R_\nu), \quad d^\sigma_\nu S^\nu_j = \left\{ \begin{array}{cr} S^\nu_i, & \sigma(S^\nu_j) = S^\nu_i \, \mbox{and} \, z(S^\nu_i) > z(S^\nu_j), \\ 0, & \mbox{else}. \end{array} \right.
\end{equation}

\labellist
\small
\endlabellist

\begin{figure}
\centerline{ \includegraphics[scale=.6]{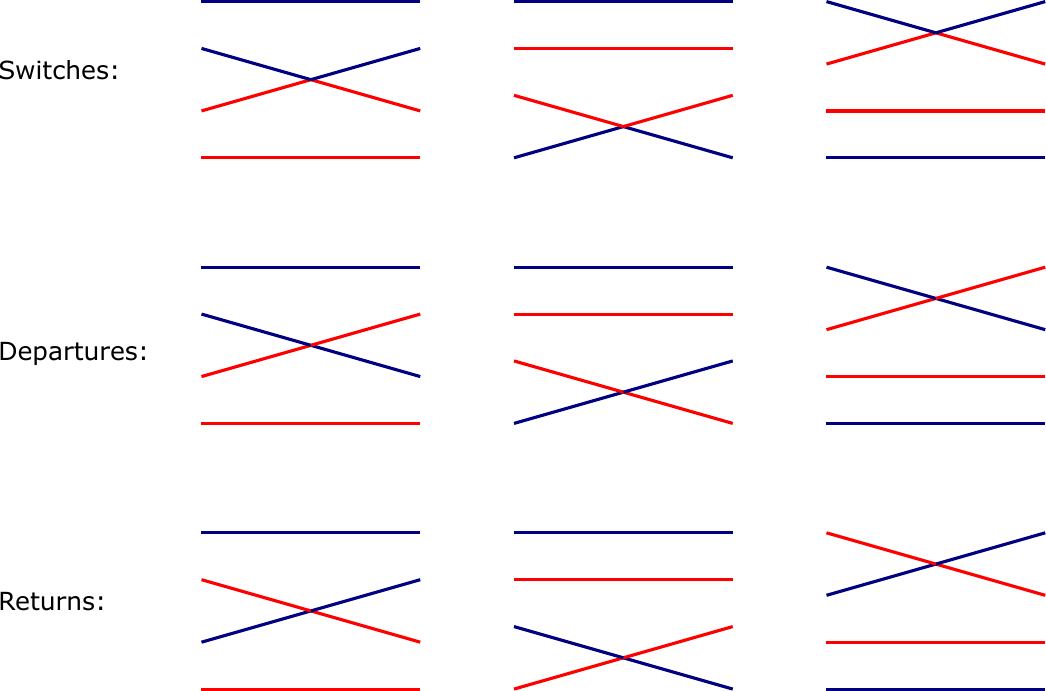} }
\caption{ Near a crossing the two sheets that cross cannot be paired with one another by a normal ruling $\sigma$.  The figure indicates the allowed configurations for the crossing sheets and their companion sheets.  (There may be other sheets of $\Lambda$ in between those pictured.) }
\label{fig:Rulings}
\end{figure}

\begin{definition}
A $\rho$-graded MCF $\mathcal{C} = (H, \{d_\nu\})$ for a $1$-dimensional Legendrian link $\Lambda$ is in {\bf $SR$-form} with respect to a $\rho$-graded normal ruling $\sigma$ for $\Lambda$ if the only handleslides of $\mathcal{C}$ are as follows:
\begin{enumerate}
\item  At \emph{every} switch of $\sigma$, handleslides are placed as specified in Figure \ref{fig:SRForm}.
\item  At \emph{some subset} of the degree $0$ (mod $\rho$) returns
 of $\sigma$, handleslides are placed as in Figure \ref{fig:SRForm}.
\item  In the $1$-graded case, at \emph{some subset} of the right cusps of $\Lambda$, a single handeslide connects the two cusp strands near the cusp.
\end{enumerate}
\end{definition}

\begin{figure}
\centerline{ \includegraphics[scale=.6]{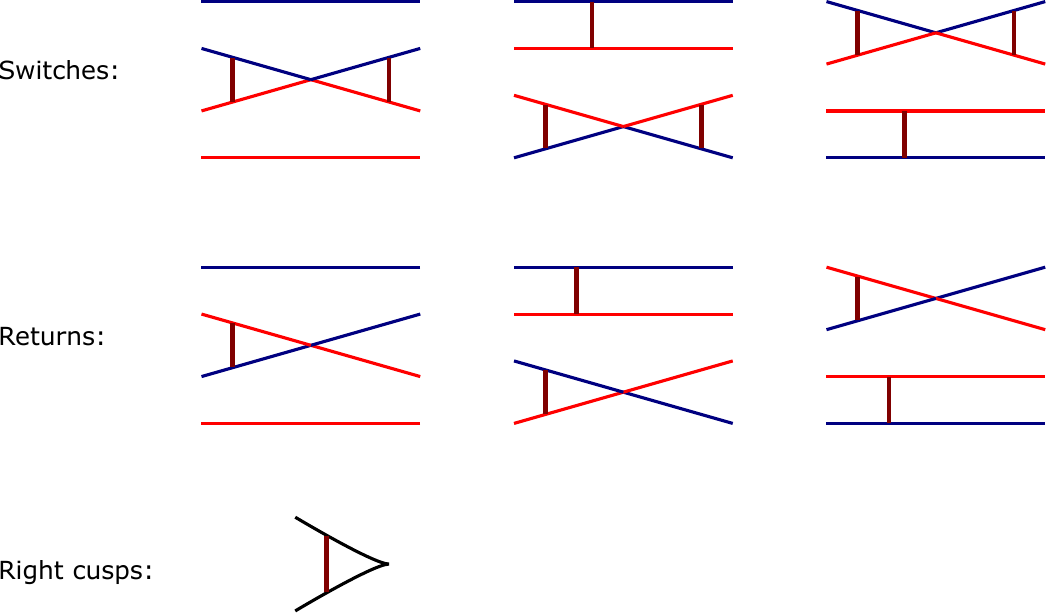} }
\caption{ The placement of handleslides near switches and some subset of the degree $0$ returns of $\sigma$ for an SR-form MCF.  When $\rho = 1$, handleslides are also placed next to some collection of right cusps as indicated.}
\label{fig:SRForm}
\end{figure}

\begin{remark}  \label{rem:SRForm}
\begin{enumerate}
\item It can be shown that for any $\rho$-graded normal ruling $\sigma$ of $\Lambda$ and any chosen subset of the $\rho$-graded returns of $\Lambda$ (and also of the right cusps of $\Lambda$ in the $1$-graded case), there is an $SR$-form MCF with handleslide set as in Figure \ref{fig:SRForm}.  
\item Whenever $\mathcal{C}= (H, \{d_\nu\})$ is an SR-form MCF and  $R_\nu \subset \R \setminus( H \cup \Lambda_{\mathit{sing}})$ is a region outside of the collection of handleslides  near switches, returns, and cusps, the differential $d_\nu$ agrees with the standard ruling differential $d^\sigma_\nu$ defined in (\ref{eq:standardruling}).
\end{enumerate}
\end{remark}

The following proposition is Theorem 6.17 from \cite{Henry}. 
\begin{proposition} \label{prop:SRform} Any $\rho$-graded MCF on a Legendrian link $\Lambda \subset J^1\R$ is equivalent to a SR-form MCF with respect to some $\rho$-graded normal ruling of $\Lambda$.
\end{proposition}

\subsection{Construction of MCFs}  \label{sec:Construct}
 The following results from \cite{Henry} and \cite{RuSu3} are useful for constructing MCFs.

\begin{proposition}  \label{prop:Handleslide}
Let $\Lambda \subset J^1M$ be a $1$-dimensional Legendrian, and let $\mathcal{C}_0 = (H, \{d_\nu\})$ be an $\rho$-graded MCF for $\Lambda$.  
\begin{enumerate}
\item Suppose that $H'$ is a handleslide set for the product cobordism $j^1(1_{[0,1]}\cdot\Lambda) \subset J^1([0,1] \times M)$ as in Definition \ref{def:MCF2d} (2) (and without  super-handleslide points) that agrees with $H$ when restricted to 
  $\{0\} \times \Lambda $ and satisfies Axiom \ref{ax:endpoints}.  Then, there is a unique $\rho$-graded MCF for $j^1(1_{[0,1]}\cdot\Lambda)$ that agrees with $\mathcal{C}_0$ on $\{0\}\times \Lambda$ and has handleslide set $H'$.
 \item Construct
a handleslide set $H'$ for $j^1(1_{[0,1]}\cdot\Lambda) \subset J^1([0,1] \times M)$ by extending each handleslide point $p$ of $\mathcal{C}_0$ to a handleslide arc along $[0,1] \times \{p\}$, and for some region $R_\nu \subset M \setminus (\Lambda_{\mathit{sing}} \cap H)$ and some $i < j$ such that $\mu(S_i^\nu) = \mu(S_j^\nu)-1$ placing a single $(i,j)$-super handleslide point, $q$, in the interior of $[0,1]\times R_\nu$ and adding handleslide arcs in $[0,1] \times R_\nu$ as specified by Axiom \ref{ax:endpoints} (2) that connect $q$ to $\{1\} \times R_\nu$ with monotonically increasing $[0,1]$ component. Then, there is a unique $\rho$-graded MCF for $j^1(1_{[0,1]}\cdot\Lambda)$ that agrees with $\mathcal{C}_0$ on $\{0\} \times \Lambda$ and has handleslide set $H'$.
\end{enumerate}

\end{proposition}

\begin{proof}  In all cases, it needs to be shown that it is possible to produce the required collection of differentials $\{d_\nu\}$ extending the given ones and so that Axiom \ref{ax:1dim} is satisfied.  (1) and (2) are both consequences of Proposition 3.8 from \cite{Henry} which shows that the differentials $\{d_\nu\}$ can be extended from $[0, t_0-\epsilon]\times M$ to $[0, t_0+\epsilon]\times M$ when a generic bifurcation of the $t$-slices of $H$ occurs at $t= t_0$.  (2) corresponds to  Henry's Move 17 in \cite{Henry}.  Alternatively, see Proposition 6.3 of \cite{RuSu3}.
\end{proof}

Next, we make a sequence of observations (B1)-\dr{(B6)} about extending MCFs along various types of elementary cobordisms that will form   the building blocks for the proof of Proposition \ref{prop:MCSextend}.  In all cases it should be observed that the Maslov potential for the given $1$-dimensional Legendrian extends in a unique way over the elementary cobordism, though we mostly leave this implicit.  Moreover, although we omit the terminology, all MCFs under consideration are $\rho$-graded.   
\begin{enumerate}
\item[{\bf (B1)}]  {\it If $\mathcal{C}_0$ and $\calC_1$ are equivalent MCFs on $\Lambda \subset J^1\R$, then there exists a MCF $\mathcal{C}$ on $j^1(1_{[0,1]}\cdot\Lambda) \subset J^1([0,1]\times \R)$ that restricts to $\mathcal{C}_i$ above $\{i\} \times \R$, for $i=0,1$.}
\end{enumerate}
\begin{proof}
This is the definition of equivalence.
\end{proof}

\begin{enumerate}
\item[{\bf (B2)}] {\it Let $\Lambda_t$, $0 \leq t \leq 1$ be a Legendrian isotopy, and consider the Legendrian surface $\Sigma \subset J^1([0,1] \times \R )$ such that  the $\{t\}\times \R$ slices  of the front projection
of $\Sigma$ are the front projections of the $\Lambda_t$.  Any MCF on $\Lambda_0$ extends over $\Sigma$.}
\end{enumerate}

\begin{figure}
\centerline{ \includegraphics[scale=.5]{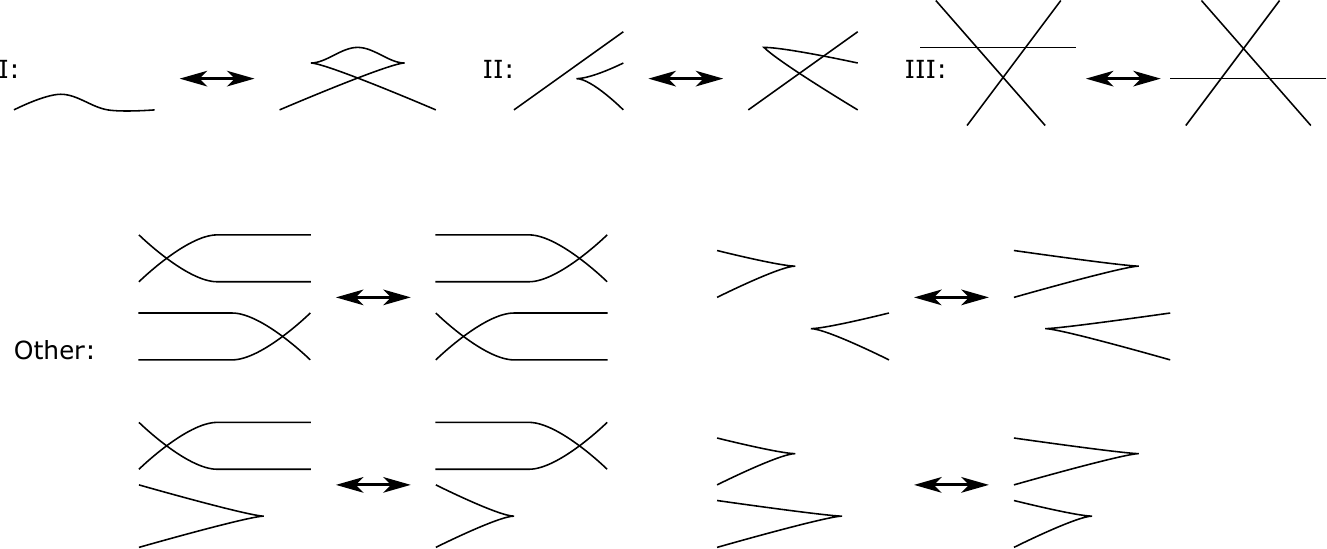}}
\caption{ Generic bifurcations of the front and base projections during a Legendrian isotopy include the Reidemeister moves I-III, and instances where two front singularities share the same $x$-coordinate (marked ``Other'' in the figure).  Vertical and horizontal reflections of all moves are allowed.  }
\label{fig:Moves}
\end{figure}

\begin{proof}
It suffices to consider the case where the isotopy contains a single bifurcation of the front or base projection, i.e. a single Reidemeister move or an instance at which two singularities (crossings or cusps) at different locations in the front diagram share the same $x$-coordinate.  See Figure \ref{fig:Moves}.

Before applying the move, we apply (B1) together with Henry's result (Proposition \ref{prop:SRform}) to assume that the given MCF $\mathcal{C}_0$ is in SR-form with respect to some normal ruling $\sigma$.  
\begin{itemize}
\item For a Type III move or any of the ``Other'' moves (both directions):  Use Proposition \ref{prop:Handleslide} to first move all of the  handleslides outside of the $x$-interval that contains the pictured part of the diagram where the move occurs.  Then, we apply the move, and check that there is a unique way to assign differentials to the new region(s) that arise from the move so that Axiom \ref{ax:1dim} is satisfied.  

This is a straightforward case-by-case check that we illustrate by considering the Type III move in detail.  (Other cases are left to the reader.)  The $t$-slices before (resp. after)  the move occurs intersect a sequence of $4$ regions of $\pi_x(\Sigma) \setminus (\Sigma_{\mathit{sing}} \cup H)$, which we label $R_A$, $R_B$, $R_C$, and $R_D$ (resp. $R_A$, $R_{B'}$, $R_{C'}$, and $R_D$);  see Figure \ref{fig:Type3}.  The differentials on $R_A$, $R_B$, $R_C$, and $R_D$ are already specified by $\mathcal{C}_0$.  Suppose that the three sheets that intersect at the triple point are numbered as $S^\nu_k, S^\nu_{k+1}, S^{\nu}_{k+2}$ above these regions.  (Recall that above any particular region we label sheets with decreasing $z$-coordinate, eg. the sheet labeled $S^A_k$ becomes $S^B_{k+1}$ when it passes through the crossing locus.)  Write $Q_{i\,i+1}$ for a linear map that interchanges the $i$ and $i+1$ sheets.  According to Axiom \ref{ax:1dim} (2) we must define the differentials on $R_{B'}$ and $R_{C'}$ so that the following maps are chain isomorphisms,
\begin{equation} \label{eq:QdA}
Q_{k+1\,k+2}: (V(R_{A}), d_A) \stackrel{\cong}{\rightarrow} (V(R_{B'}), d_{B'}), \quad \mbox{and} \quad Q_{k\,k+1}: (V(R_{B'}), d_{B'}) \stackrel{\cong}{\rightarrow} (V(R_{C'}), d_{C'}).
\end{equation}
In order for this to define a valid MCF we only need to check that Axiom \ref{ax:1dim} (2) is satisfied along the border between $R_{C'}$ and $R_D$, i.e., we need to show that 
\[
Q_{k+1\,k+2}:(V(R_{C'}), d_{C'}) \rightarrow (V(R_D), d_D) 
\]
is a chain map.  In view of (\ref{eq:QdA}), this is equivalent to the composition $Q_{k+1\,k+2} \circ Q_{k\, k+1} \circ Q_{k+1 \,k+2}: (V(R_A), d_A) \rightarrow (V(R_D), d_D)$ being a chain isomorphism. Since we know that Axiom \ref{ax:1dim} is satisfied by $\mathcal{C}_0$, when we pass from $R_A$ to $R_D$ by way of $R_B$ and $R_C$ we see that $Q_{k \, k+1} \circ Q_{k+1 \, k+2} \circ Q_{k\,k+1}:(V(R_A), d_A) \rightarrow (V(R_D), d_D)$ is a chain isomorphism, so we can just note the braid relation $Q_{k\,k+1} \circ Q_{k+1\, k+2} \circ Q_{k \,k+1} = Q_{k+1 \, k+2} \circ Q_{k \, k+1} \circ Q_{k+1\,k+2}$.

\labellist
\small
\pinlabel $x$ [l] at 38 2
\pinlabel $z$ [b] at 2 38
\pinlabel $x$ [l] at 476 2
\pinlabel $t$ [b] at 442 38
\pinlabel $R_A$  at 504 64
\pinlabel $R_{B}$ [b] at 534 0
\pinlabel $R_{B'}$ [t] at 534 134
\pinlabel $R_C$ [b] at 578 0
\pinlabel $R_{C'}$ [t] at 578 134
\pinlabel $R_D$ at 604 64

\endlabellist

\begin{figure}
\centerline{ \includegraphics[scale=.6]{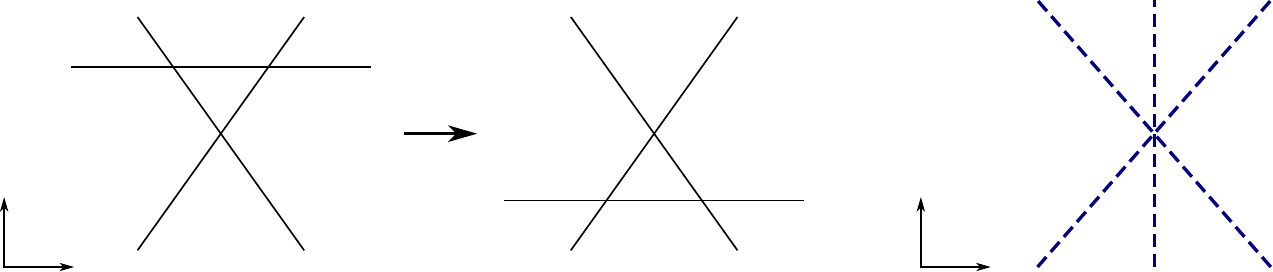} }
\caption{ Labeling of the regions in the base projection near a triple point.}
\label{fig:Type3}
\end{figure}

\item For a Type I move ($\leftarrow$ direction):   A crossing, $b$, a left cusp, $c_l$, and a right cusp, $c_r$, of $\Lambda_0$ all vanish at a swallowtail point, $s$, during the move.   The crossing, $b$, must be a switch for $\sigma$, so under the SR-form assumption, \dr{when $\rho\neq 1$}, the handleslides appearing near the crossing are exactly those required to have endpoints at the swallowtail point by Axiom \ref{ax:endpoints} (3).  Extend these to handleslide arcs with endpoints at $s$.  Let $\partial_{\nu_l}$ (resp.  $\partial_{\nu_r}$) denote the differential from $\mathcal{C}_0$ assigned to the region $R_{\nu_l}$ (resp. $R_{\nu_r}$) that borders $c_l$ on the left (resp. $c_r$ on the right).  After the Type I move, $R_{\nu_l}$ and $R_{\nu_r}$ merge to become a single region, so it is important to note that $\partial_{\nu_l}$ and $\partial_{\nu_r}$ agree.  This is the case since by Remark \ref{rem:SRForm} (2), they both agree with the standard ruling differential for $\sigma$. 

\dr{When  $\rho=1$, the $SR$-form may also have a handleslide connecting the strands of $c_r$, and this requires a preliminary step to remove the handleslide.   This is done using Proposition \ref{prop:Handleslide} to alter the handleslide set via the sequence of adjustments pictured in  \cite[Figure 16]{PanRu3}.}
  %The handleslide set sequence of two superhandleslides is performed... one to the left of the crossing and one to the right of the crossing.  If the switch is nested this will result in the appearance of one extra handleslide, but it can be slid out of the way before doing the Type 1 move, and will be left connecting the strands of the other ruling disk of the switch.}

\item For a Type I move ($\rightarrow$ direction):  To extend $\mathcal{C}_0$, we just add additional handleslide arcs with end points at the swallowtail point as required by Axiom \ref{ax:endpoints}.  Differentials for the new regions can be defined since the handleslide set on the slices after the swallow tail point occurs is in SR-form for the ruling $\sigma'$ obtained from $\sigma$ by making the new crossing into a switch.  Alternatively, see Proposition 6.2 of \cite{RuSu3}.

\item For a Type II move ($\leftarrow$ direction): Once again, the strategy is to use an equivalence to move all of the handleslides out of the $x$-inteval, $I$, where the front diagram is pictured and then perform the move.  It is then routine to check that the required differentials can be defined to complete the extension of $\mathcal{C}_0$ to $\mathcal{C}$.  The two pictured crossings of $\Lambda_0$ are a departure followed by a return, so (since $\mathcal{C}_0$ is in SR-form) the only handleslides that need to be moved are those that may appear near the return.  Figure \ref{fig:Type2} illustrates a $2$-dimensional MCF on $j^1(1_{[0,1]}\cdot\Lambda)$ that will remove all handleslides from $I$ as required.    This MCF involves a single super-handleslide point, so Proposition \ref{prop:Handleslide} (1) and (2) produce the required differentials.  (A similar procedure applies for the horizontally and/or vertically reflected versions of Move II.)

\item For a Type II move ($\rightarrow$ direction):  There are no handleslides in the interval where the move occurs.  We perform the  move and then check that the differentials may be extended.

\end{itemize}
\end{proof}

\labellist
\small
\pinlabel $x$ [l] at 38 154
\pinlabel $z$ [b] at 2 188
\pinlabel $x$ [l] at 292 2
\pinlabel $t$ [b] at 258 36
\endlabellist

\begin{figure}
\centerline{ \includegraphics[scale=.6]{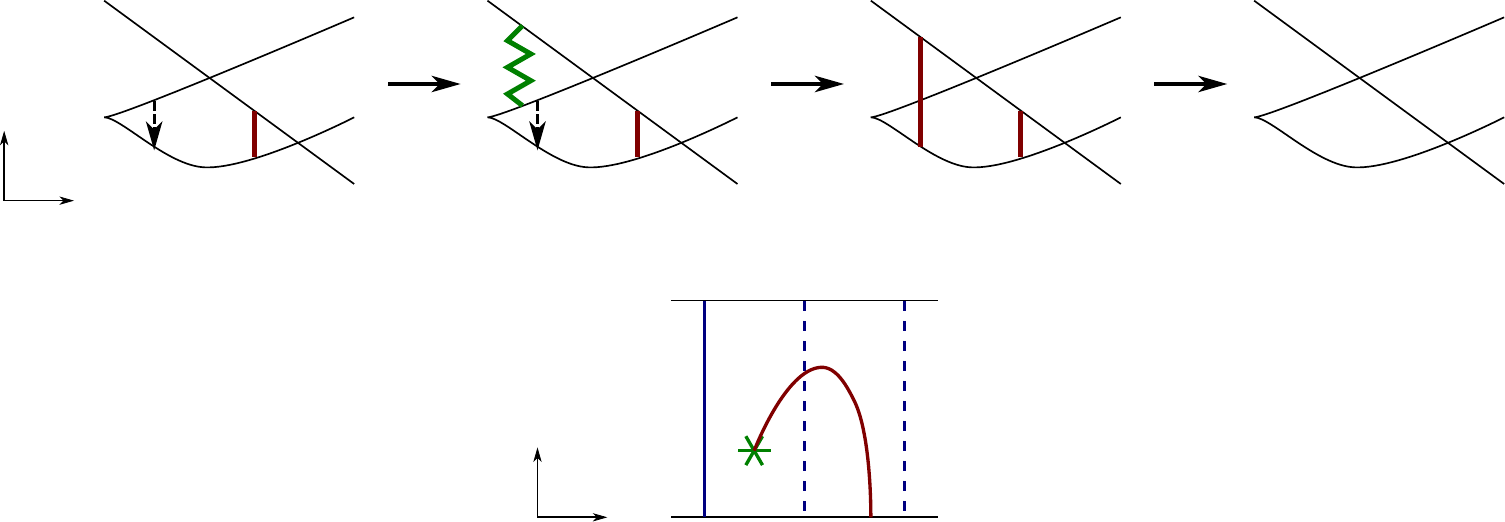} }
\caption{ Removing a handleslide at a return to prepare for a Type II move.  The top row illustrates the MCF on a sequence of $t$-slices with $t$ increasing.  The bottom row depicts the MCF in the base projection.  If the return is as in the 2nd or 3rd column of Figure \ref{fig:Rulings}, there will be a second handleslide arc with endpoint at the super-handleslide point.  Neither of its endpoints are on the cusp sheets, so it can be moved away from the pictured part of the front projection.}
\label{fig:Type2}
\end{figure}

\begin{remark} An
 alternate 
approach to establishing (B2) is made possible by the correspondence between MCFs and augmentations. When $\Phi$ is a Legendrian isotopy 
from $\Lambda_-$ to $\Lambda_+$, as in Section \ref{sec:AugSet}, there is an invertible conical Legendrian cobordism, $\Sigma_\Phi$, with embedded Lagrangian projection. 
Thus, because the induced augmentation set $I_\Sigma$ has the form found in equation (\ref{eq:embeddedSet}) the calculation of $I_\Sigma$ in terms of MCFs via Proposition \ref{prop:MCFcomp} and Corollary \ref{cor:augmentationset} implies that any MCF for $\Lambda_-$ can be extended over (a compact version of) $\Sigma$.
\end{remark}

Along with moves associated to $1$-dimensional Legendrian isotopies, the generic front bifurcations of the $t$-slices of a Legendrian cobordism include the {\bf Clasp Move}, the {\bf Pinch Move}, and the {\bf Unknot Move} as pictured in Figure \ref{fig:PinchMoves}.  These moves correspond to a local maximum or minimum in the $t$-direction for the crossing locus (in the case of the Clasp Move) or the cusp locus (for the Pinch and Unknot Move).  It will also be convenient to consider a (non-generic) {\bf Cusp Tangency Move} pictured in Figure \ref{fig:CuspMove} which can be realized by a combination of the Clasp Move with a Legendrian isotopy.

\labellist
\small
\pinlabel $x$ [l] at 84 2 
\pinlabel $z$ [b] at 52 34
\pinlabel $x$ [l] at 398 2 
\pinlabel $t$ [b] at 366 34
\large
\pinlabel $\emptyset$ [b] at 116 40
\endlabellist

\begin{figure}
\centerline{ \includegraphics[scale=.5]{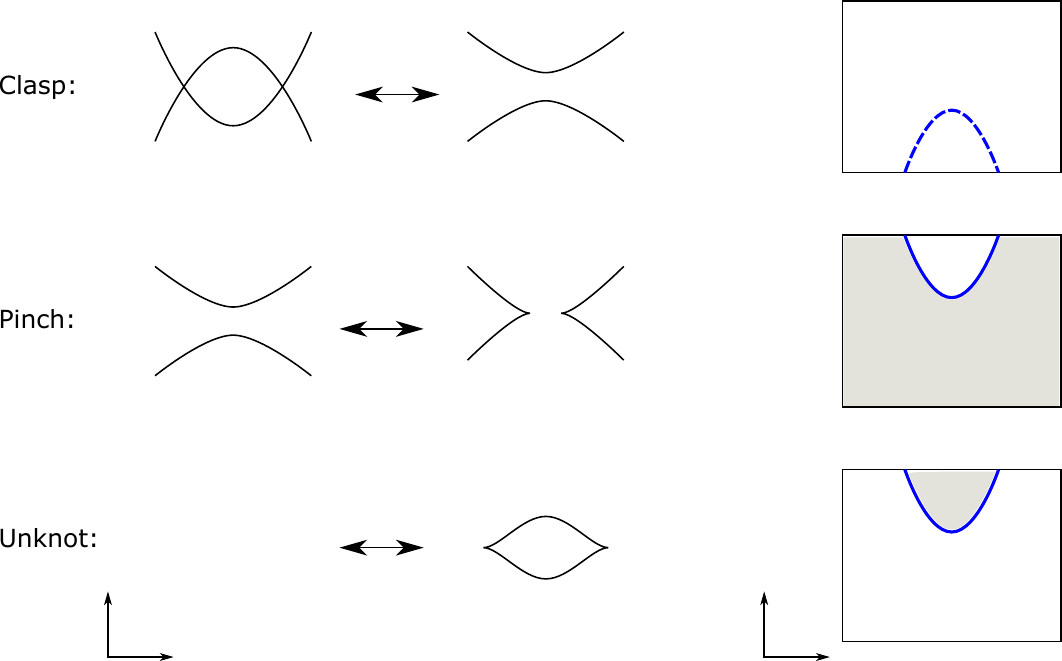} }
\caption{ The Clasp, Pinch, and Unknot Moves, pictured as front projection slices (left) and in the base projection of the corresponding $2$-dimensional cobordism (right).  Shading in the base projection indicates the region where the two cusp sheets exist.}
\label{fig:PinchMoves}
\end{figure}

\labellist
\small
\endlabellist

\begin{figure}
\centerline{ \includegraphics[scale=.5]{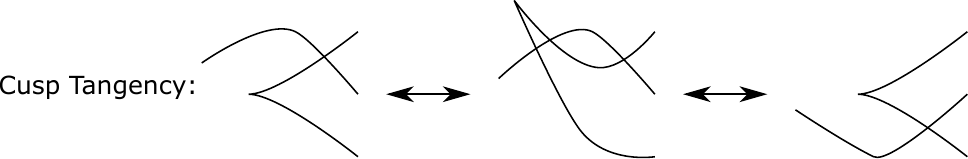} }
\caption{ The Cusp Tangency Move refers to the interchange of the far left and far right diagrams.  This is accomplished by a Type II Move and a Clasp Move, as pictured.}
\label{fig:CuspMove}
\end{figure}

\begin{enumerate}
\item[{\bf (B3)}]  
\begin{enumerate}
\item {\it An MCF can be extended along the $\rightarrow$ direction of the Clasp Move if there is no handleslide that connects the crossing sheets and has its $x$-coordinate between the two crossings.}  
\item {\it An MCF can be extended along the Cusp Tangency Move (either direction) provided there is no handleslide that connects the crossing strands and has its $x$-coordinate to the left of the crossing.} 
\end{enumerate}
\end{enumerate}

See Figure \ref{fig:Clasp}.

\labellist
\small
\endlabellist

\begin{figure}
\centerline{ \includegraphics[scale=.5]{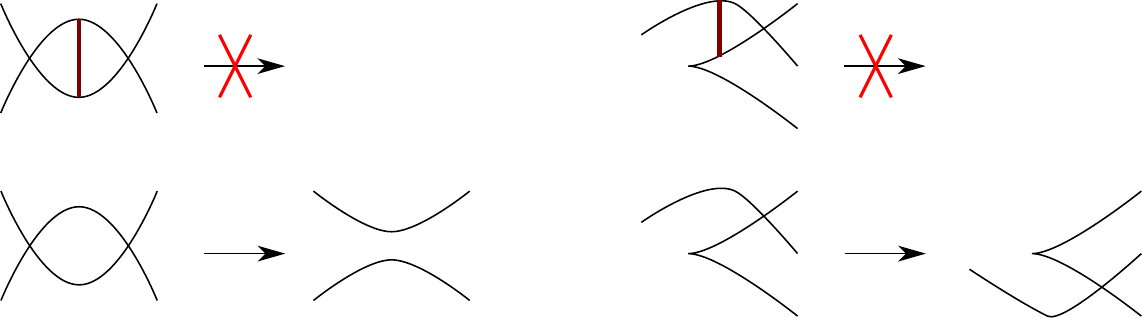} }
\caption{ The requirement on handleslides from (B3) that allows for MCFs to be extended through a clasp move (left) or cusp tangency (right).}
\label{fig:Clasp}
\end{figure}

\begin{proof}
Let $\Sigma$ be the Legendrian surface corresponding to the Clasp Move ($\rightarrow$), with $\Lambda_0$ and $\Lambda_1$ denoting the $1$-dimensional slices before and after the move.  Let $\mathcal{C}_0$ be an MCF for $\Lambda_0$ with no handleslides between the crossings.  Then, there is a sequence of three adjacent regions $R_A,R_B, R_C$ for $\mathcal{C}_0$ where $R_B$ is between the crossings and $R_A$ and $R_C$ are to the left and to the right.  Since there are no handleslides between the crossings, we can define a MCF for $\Sigma$ by extending all handleslide points of $\mathcal{C}_0$ along straight line segments in the $t$-direction.  Axiom \ref{ax:1dim} (2) for $\mathcal{C}_0$ shows that the differentials on the two regions $R_A$ and $R_B$ agree, so that there is a well defined differential on the common region for $\Sigma$ that contains them.  

The Cusp Tangency Move follows from the first case of (B3) since the hypothesis restricting the location of handleslides before the move implies that the move can be realized (in either direction) by a Type II Move followed by a Clasp Move ($\rightarrow$ direction) in such a way that there are no handleslides between the crossings when the Clasp Move is applied.
\end{proof}

\begin{enumerate}
\item[{\bf (B4)}] {\it An SR-form MCF for $\Lambda$ associated to a $\rho$-graded normal ruling $\sigma$ can be extended to a $\rho$-graded MCF on the elementary cobordism $\Sigma$ arising from applying a Pinch Move  ($\rightarrow$ direction) to adjacent sheets of $\Lambda$, $S^\nu_{k}$ and $S^\nu_{k+1}$, above a region where they are paired by $\sigma$.}  
\end{enumerate}

Note that the $\Z/\rho$-valued Maslov potential $\mu$ for $\Lambda$ extends over $\Sigma$ since the fact that $\sigma$ is $\rho$-graded implies that $\mu(S^\nu_{k}) = \mu(S^\nu_{k+1})+1$  mod $\rho$.

\begin{proof}

At the location of the pinch move, the differential for the SR-form MCF, $d_\nu$, is the standard ruling differential for $\sigma$.  (See Remark \ref{rem:SRForm} (2).)  Thus,  $d_{\nu} S_{k+1} = S_k$, and $S_{k}$ and $S_{k+1}$ do not appear in the differentials of the other generators.  That is, the complex for $(V(R_\nu), d_{\nu})$ splits as a direct sum $(V(R_\nu), d_{\nu}) = (C_1,d_1) \oplus (C_2,d_2)$ with $C_2$ spanned by $\{S_{k}, S_{k+1}\}$ and $C_1$ spanned by the rest of the sheets. Therefore, when we extend the MCF over the surface by using $(C_1,d_1)$ in the region where $S_k$ and $S_{k+1}$ do not exist, the Axiom \ref{ax:1dim} is satisfied. 
\end{proof}

\begin{enumerate}
\item[{\bf (B5)}] {\it An MCF $\mathcal{C}$ can be extended along the Unknot Move ($\leftarrow$ direction), provided there is no handleslide connecting the two sheets of the unknot.  (Note that such a handleslide cannot exist if $\mathcal{C}$ is $\rho$-graded and $\rho \neq 1$.)}
\end{enumerate}

\begin{proof}
Using (B1) we can assume $\mathcal{C}$ is in SR-Form on the slice that precedes the unknot move.  It then follows that (i) there are no handleslides with endpoints on the unknot sheets, and (ii) in the region where the unknot exists $(V(R_\nu), d_{\nu}) = (C_1,d_1) \oplus (C_2,d_2)$ where $C_1$ is spanned by the unknot sheets, $S^\nu_k$ and $S^\nu_{k+1}$, and $d_\nu S^\nu_{k+1} = S_k$.  Moreover, by Axiom \ref{ax:1dim} the differentials on the two regions adjacent to the unknot both agree with $d_2$. Thus, $\mathcal{C}$ extends in an obvious way over the surface.
\end{proof}

\begin{enumerate}
\item[{\bf (B6)}]  {\it Any MCF $\mathcal{C}$ can be extended along the $\leftarrow$ direction of the Pinch Move.}
\end{enumerate}

\begin{proof}
\dr{Again using (B1) we can assume $\mathcal{C}$ has no handleslides near the two cusps that are joined by the Pinch Move.  From Axiom \ref{ax:1dim} the differentials agree in the two regions where the cusp sheets exist, so that $\mathcal{C}$ extends over the surface.} 
\end{proof}

\begin{remark}  \label{rem:extramoves}
Although we will not need to use them in our proof of Theorem \ref{thm:main}, it is also easy to give necessary and sufficient conditions for an MCF to extend over the remaining moves.
\begin{enumerate}
\item Any MCF can be extended 
over 
%the Pinch Move ($\leftarrow$ direction) or 
the Unknot Move ($\rightarrow$ direction). (This is easy to see directly, but also follows from the fact that there are conical Legendrian cobordisms realizing these moves that have embedded Lagrangian projections; cf. \cite{EHK, BST}. The same comment applies to the $\leftarrow$ Pinch Move.)
%\item Any MCF can be extended over the Pinch Move ($\leftarrow$ direction) or 
%the Unknot Move ($\rightarrow$ direction). (This is easy to see directly, but also follows from the fact that there are conical Legendrian cobordisms realizing these moves that have embedded Lagrangian projections; cf. \cite{EHK, BST}.)
\item An MCF can be extended over the Clasp Move ($\leftarrow$ direction) if and only if $\langle d_\nu S_{k+1}, S_k \rangle = 0$ where $R_{\nu}$ is the region where the crossings will appear and $S_k$ and $S_{k+1}$ are the two crossing sheets.
\end{enumerate}
\end{remark}

\subsection{Construction of the Legendrian filling}

The front projection of a generic Legendrian $\Lambda \subset J^1\R$ can be represented as a word $w_\Lambda$ that is a product (left to right concatenation) of elementary Legendrian tangles, each one of which contains a single crossing or cusp.  We notate these {\bf elementary tangles} as $l^{n-2,n}_{k}, r^{n,n-2}_{k}$, and $\sigma^{n,n}_k$ in the case of a tangle with a left cusp, right cusp, or crossing respectively, where $n \geq 2$ and $ 1 \leq k \leq n-1$.  The superscripts indicate the number of strands of the tangle at its left and right vertical boundaries, while $k$ is the numbering of the upper of the two strands that is involved with the crossing or cusp when we number strands as $1$ to $n$ from top to bottom.  See Figure \ref{fig:Elementary}.  In the following, we suppress the superscripts from notation.  

\labellist
\tiny
\pinlabel $1$ [r] at -2 106
\pinlabel $2$ [r] at -2 84
\pinlabel $n-2$ [r] at -2 16
\pinlabel $1$ [l] at 76 142
\pinlabel $2$ [l] at 76 120
\pinlabel $k$ [l] at 76 98
\pinlabel $k+1$ [l] at 76 68
\pinlabel $n$ [l] at 76 2
\endlabellist

\begin{figure}
\centerline{ \includegraphics[scale=.6]{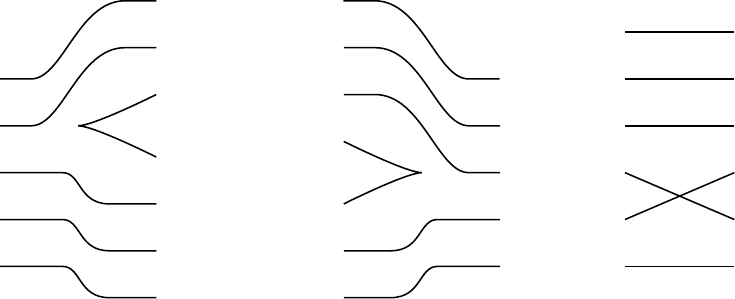} }
\caption{ The elementary tangles $l^{n-2,n}_{k}, r^{n,n-2}_{k}$, and $\sigma^{n,n}_k$.  The picture illustrates $l^{5,7}_3$, $r^{7,5}_4$, and $\sigma^{6,6}_4$ from left to right. }
\label{fig:Elementary}
\end{figure}

\begin{proof}[Proof of Proposition \ref{prop:MCSextend}]
Assume $\mathcal{C}_\Lambda$ is an $\rho$-graded MCF for $\Lambda \subset J^1\R$.  We first prove the proposition {\it assuming $\rho \neq 1$}, and then close by indicating the minor modifications to the proof when $\rho=1$.

Let $w_\Lambda$ be the word representing the front projection of $\Lambda$, and note that \dr{if $\Lambda \neq\emptyset$ then} $w_\Lambda$ can be written in the form
\begin{equation}  \label{eq:Yform}
w_\Lambda = X l_k \sigma_{k-1} \sigma_{k-2} \cdots \sigma_{k-s} \sigma_{k+1} \sigma_{k+2} \cdots \sigma_{k+t} Y
\end{equation}
where $s,t \geq 0$ and $Y$ contains no left cusps.  (For example, just take the $l_k$ term to be the left cusp of $\Lambda$ with largest $x$-coordinate, and put $s=t=0$.)  See Figure \ref{fig:wLambda}.

\labellist
\large
\pinlabel $X$  at 48 136 
\pinlabel $Y$  at 366 136
\endlabellist

\begin{figure}
\centerline{ \includegraphics[scale=.5]{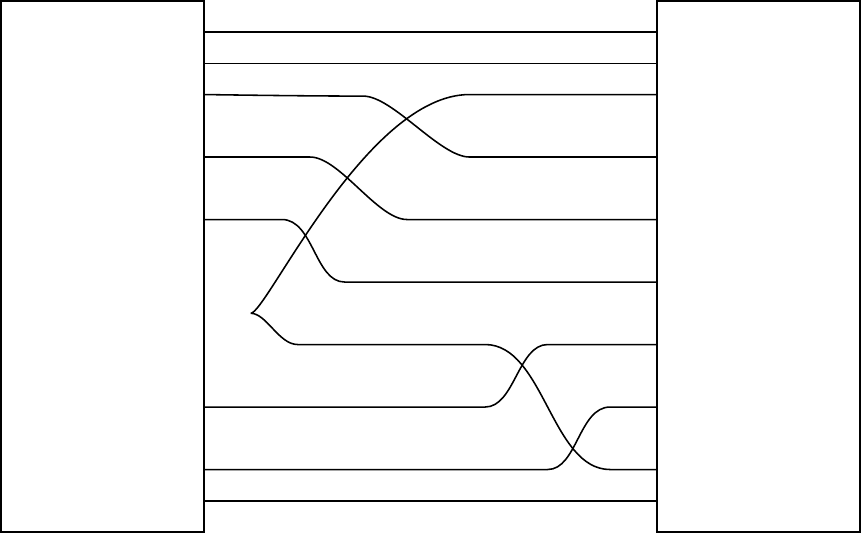} }
\caption{The front projection corresponding to the word $w_\Lambda = X l_k \sigma_{k-1} \sigma_{k-2} \cdots \sigma_{k-s} \sigma_{k+1} \sigma_{k+2} \cdots \sigma_{k+t} Y$ when $s =3$ and $t=2$.
 }
\label{fig:wLambda}
\end{figure}

We prove the following statement by a nested induction.  The outer induction is on $c$ and the inner induction is on $n$.

\medskip

\noindent{\bf Inductive Statement:}  Suppose that $w_\Lambda$ has $c$ left cusps, and if $c \geq 1$ assume that $w_\Lambda$ can be written in the form (\ref{eq:Yform}) such that the length of $Y$ is $|Y|=n$.  Then, there exists a compact Legendrian filling  $\Sigma\subset J^1([0,1]\times \R)$ with a $\rho$-graded MCF $\mathcal{C}$ such that $\mathcal{C}|_\Lambda = \mathcal{C}_\Lambda$.  
%If $w_\Lambda$ has $c$ left cusps and can be written in the form (\ref{eq:Yform}) such that the length of $Y$ is $|Y|=n$, then there exists a compact Legendrian filling  $\Sigma\subset J^1([0,1]\times \R)$ with a $\rho$-graded MCF $\mathcal{C}$ such that $\mathcal{C}|_\Lambda = \mathcal{C}_\Lambda$.  

\medskip

The case $c=0$ is trivial since $\Lambda = \emptyset$.  For fixed $c \geq 1$, \dr{assuming the statement holds for smaller values of $c$, we establish the statement via induction on $n$.}   The case $n=0$ is vacuously true, \dr{since $Y$ must have at least one right cusp.}  Assuming $n\geq 1$, we write $Y = z Y'$ \dr{(allowing the possibility that $Y'$ is the empty word)} where $z$ is an elementary tangle that is necessarily a right cusp or a crossing, $z= r_j$ or $z=\sigma_j$.  We consider cases depending on the vertical location of this right cusp (in Cases 1-5) or crossing (in Cases 6-10).  By symmetry
we can assume $j \leq k$.  In addition, using (B1) and Proposition \ref{prop:SRform} we can assume that $\mathcal{C}_\Lambda$ is in SR-form with respect to a $\rho$-graded normal ruling $\sigma$.

\medskip

\noindent {\bf Case 1:}  $z = r_{j}$ with $j\leq k-s-2$.  Then, $z$ is a right cusp above all the sheets with cusps or crossings in the product $l_k \sigma_{k-1} \sigma_{k-2} \cdots \sigma_{k-s} \sigma_{k+1} \sigma_{k+2} \cdots \sigma_{k+t}$.  So, a Legendrian isotopy of $\Lambda$  modifies the front diagram by 
\[
X l_k \sigma_{k-1} \sigma_{k-2} \cdots \sigma_{k-s} \sigma_{k+1} \sigma_{k+2} \cdots \sigma_{k+t} (r_{j}Y') \, \rightarrow \, (X r_{j}) l_{k-2} \sigma_{k-3} \sigma_{k-4} \cdots \sigma_{k-2-s} \sigma_{k-1} \sigma_{k} \cdots \sigma_{k-2+t} Y'.
\]
Use (B2) to extend $\mathcal{C}_\Lambda$ over the Legendrian isotopy.  Then, the inductive hypothesis on $n$ applies to complete the construction of $(\Sigma, \mathcal{C})$.

\medskip

\noindent {\bf Case 2:} $z = r_{k-s-1}$.  \dr{See Figure \ref{fig:wmoves} for a summary of this case.}  
Using a Legendrian isotopy, we have
\[
w_\Lambda \rightarrow  X l_k \sigma_{k-1} \sigma_{k-2} \cdots \sigma_{k-s} r_{k-s-1} \sigma_{k-1} \sigma_{k} \cdots \sigma_{k-2+t} Y'.
\]
Note that we must have $s \geq 1$, since otherwise a ``zig-zag'' $l_k r_{k-1}$ would occur in $\Lambda$.  This is not possible since MCFs do not exist for stabilized Legendrian links. 
The idea is to try to move the cusp $l_k$ up next to the $r_{k-s-1}$.  Consider the first crossing to the right of $l_k$, $x=\sigma_{k-1}$.

\begin{figure}[!ht]
\labellist
\small
\pinlabel $\vdots$ at  20 30
\pinlabel $\vdots$ at  20 175
\pinlabel $\vdots$ at  20 330
\pinlabel $\vdots$ at  20 470

\pinlabel $\vdots$ at  360 35
\pinlabel $\vdots$ at  360 175
\pinlabel $\vdots$ at  360 330
\pinlabel $\vdots$ at  360 470

\pinlabel $\vdots$ at  690 35
\pinlabel $\vdots$ at  690 110
\pinlabel $\vdots$ at  690 335
\pinlabel $\vdots$ at  690 475

\pinlabel $(A)$ at 470 260
\pinlabel $(B)$ at 470 -20
\endlabellist
\includegraphics[width=5in]{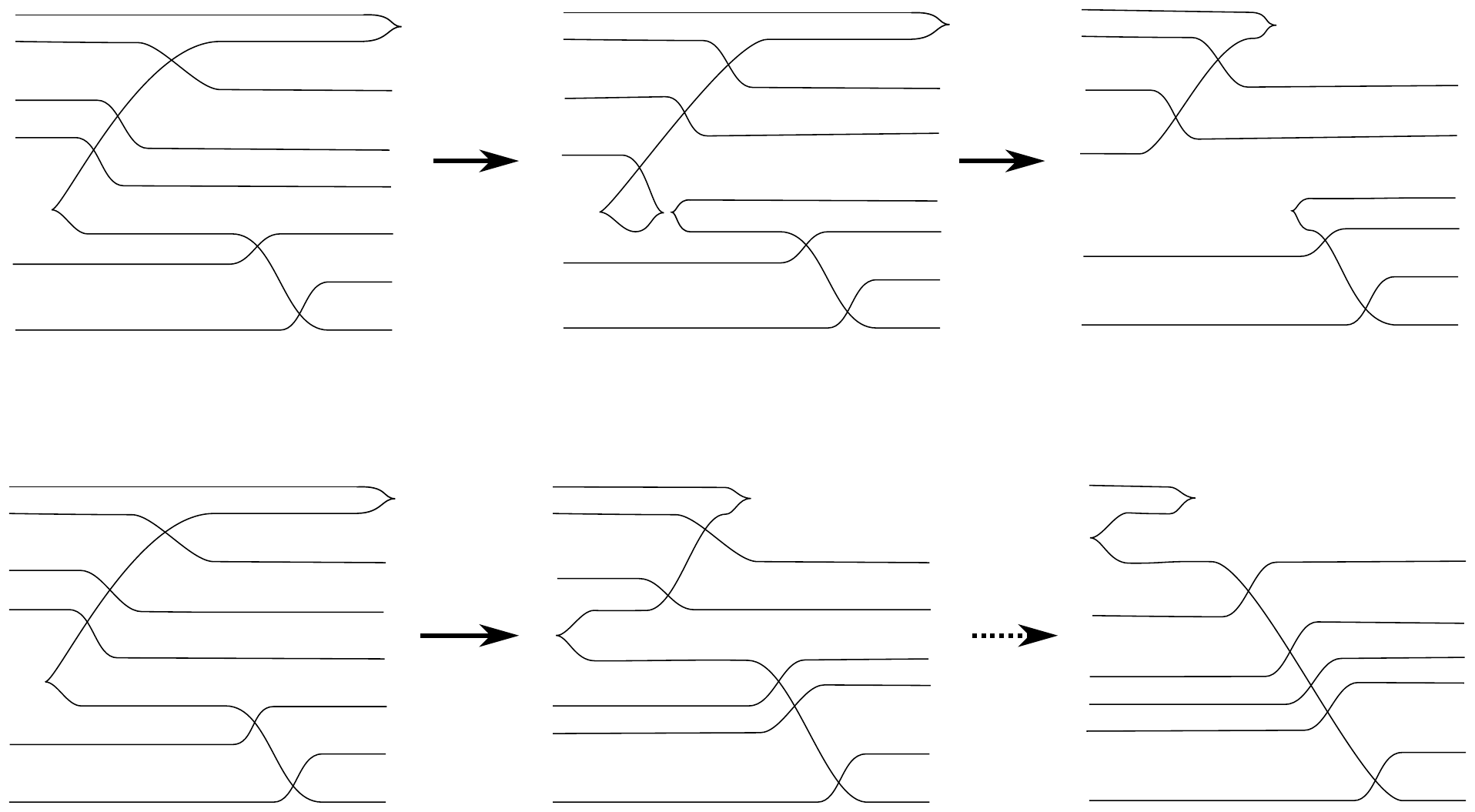}
\caption{\dr{The first line outlines the cobordism used in Case 2, Subcase A and the second line outlines the cobordism for Subcase B.  Subcase A will necessarily occur before the front at the far right of (B) is reached.}}
\label{fig:wmoves}
\end{figure}
\begin{itemize}
\item[Subcase A:]  $x$ is a switch for $\sigma$.  Since $\mathcal{C}$ is in SR-form for $\sigma$, we apply a $\rightarrow$ Pinch Move to the right of $x$ (extending the MCF using (B4)) followed by a Type I Reidemeister Move to produce an MCF on a surface with the slices
\begin{align*}
\cdots l_k \sigma_{k-1} \cdots \stackrel{(B4)}{\rightarrow} & \cdots l_k \sigma_{k-1} r_k l_k \cdots \rightarrow \cdots l_k \sigma_{k-2} \cdots \sigma_{k-s} r_{k-s-1} \sigma_{k-1} \sigma_{k} \cdots \sigma_{k-2+t} Y' \\
\rightarrow & \cdots \sigma_{k-2} \cdots \sigma_{k-s} r_{k-s-1} l_{k-2}\sigma_{k-1} \sigma_{k} \cdots \sigma_{k-2+t} Y'.
\end{align*}
Then, the inductive hypothesis on $n$ applies.

\item[Subcase B:]  $x$ is not a switch for $\sigma$.  Then, it is a departure, so since $\mathcal{C}_\Lambda$ is in SR-form, it has no handleslides connecting the crossing strands to the left of $x=\sigma_{k-1}$.  Then, we can apply (B3) to extend past the Cusp Tangency Move:
\begin{align*}
 X l_k \sigma_{k-1} \sigma_{k-2} \cdots \sigma_{k-s} r_{k-s-1} \sigma_{k-1} \sigma_{k} \cdots \sigma_{k-2+t} Y' \stackrel{(B3)}{\rightarrow} & X l_{k-1} \sigma_k \sigma_{k-2} \cdots \sigma_{k-s} r_{k-s-1} \sigma_{k-1} \sigma_{k} \cdots \sigma_{k-2+t} Y'  \\
 \rightarrow & X l_{k-1} \sigma_{k-2} \cdots \sigma_{k-s} r_{k-s-1} \sigma_{k-2} \sigma_{k-1} \sigma_{k} \cdots \sigma_{k-2+t} Y'.
\end{align*}
We then repeat this argument with $x=\sigma_{k-2}$.  Continuing in this manner, we must encounter the switch case at some point; if not, the left cusp would eventually be moved next to the right cusp, and we would have an MCF for a stabilized link which is impossible.  (The front would have a ``zig-zag'', $l_{k-s}r_{k-s-1}$.)  
\end{itemize}

\medskip

\noindent {\bf Case 3:}  $s \geq 1$ and $z= r_{k-s}$.  Then, the presence of the term $\sigma_{k-s} r_{k-s}$ shows that $\Lambda$ is stabilized which is impossible.  (There is a standard Legendrian isotopy that will turn the ``fish tail'' from the $\sigma_{k-s} r_{k-s}$ product into a ``zig-zag''.)

\medskip

\noindent {\bf Case 4:} $z = r_j$ with $k-s < j < k$.  Then,  a Type II move can be applied to produce a Legendrian isotopy
\begin{align*}
X l_k \sigma_{k-1} \sigma_{k-2} \cdots \sigma_{k-s} \sigma_{k+1} \sigma_{k+2} \cdots \sigma_{k+t} (r_{j}Y') \, \rightarrow  & \, (X r_{j-1}) l_{k-2} \sigma_{k-3}  \cdots \sigma_{k-s} \sigma_{k-1} \sigma_{k} \cdots \sigma_{k-2+t} Y',
\end{align*}
where the product $\sigma_{k-3} \cdots \sigma_{k-s}$ just means the identity tangle if $s =2$.  Again, we can apply (B2) and then the inductive hypothesis on $n$ to produce $(\Sigma, \mathcal{C})$.

\medskip

\noindent {\bf Case 5:}  $z = r_{k}$. 

\begin{itemize}
\item[Subcase A:] $s=t=0$.  Then, we have a standard Legendrian unknot in the middle of the diagram.  {\it As long as $\rho \neq 1$}, there can be no handleslides connecting the strands of this unknot, so we apply (B5) to extend $\mathcal{C}$ as we remove the unknot.  This decreases $c$, so that the inductive hypothesis applies.
\item[Subcase B:] Exactly one of $s >0$ or $t>0$.  Then, we can apply a Type I Reidemeister move to reduce the number of cusps.  The $(\Sigma, \mathcal{C})$ is then constructed via (B2) and the inductive hypothesis on $c$.

\item[Subcase C:] $s>0$ and $t>0$.  If the crossing $\sigma_{k-1}$ is a switch for the normal ruling $\sigma$ associated to $\mathcal{C}_\Lambda$, then  we can apply a pinch move, as in Case 2, followed by two Type I moves to decrease $c$ while extending $\mathcal{C}$ via (B4) and (B2).  If $\sigma_{k-1}$ is not a switch, then we apply a Cusp Tangency Move as in (B3) to arrive at 
\[
X l_{k-1}  \sigma_{k-2} \cdots \sigma_{k-s} \sigma_{k} \sigma_{k+1} \sigma_{k+2} \cdots \sigma_{k+t} r_k Y'.
\] 
At this point, Case 4 (reflected vertically) applies.
\end{itemize}

\medskip

\noindent {\bf Case 6:}  $z = \sigma_{j}$ with $j\leq k-s-2$.  As in Case 1, move $z$ to the left of $l_k$ and then apply the inductive hypothesis on $n$.

\medskip
 
\noindent {\bf Case 7:} $z = \sigma_{k-s-1}$.  
Using a Legendrian isotopy, we can group $\sigma_{k-s-1}$ into the product
\[
X (l_k \sigma_{k-1} \sigma_{k-2} \cdots \sigma_{k-s} \sigma_{k+1} \sigma_{k+2} \cdots \sigma_{k+t}) (\sigma_{k-s-1}Y') \rightarrow 
X (l_k \sigma_{k-1} \sigma_{k-2} \cdots \sigma_{k-s} \sigma_{k-s-1} \sigma_{k+1} \sigma_{k+2} \cdots \sigma_{k+t}) Y'
\]
and apply the inductive hypothesis on $n$.

\medskip

\noindent {\bf Case 8:}  $s \geq 1$ and $z= \sigma_{k-s}$.  Using an argument similar to Case 2, we attempt to apply cusp tangency moves to move the left cusp directly next to $z$.  Initially, assume $k-1 >k-s$.  If the crossing $x=\sigma_{k-1}$ directly to the right of $l_k$ is a switch, we can apply a Pinch Move  followed by a Type I Reidemeister Move to go from
\begin{align*}
\cdots l_k \sigma_{k-1} \sigma_{k-2} \cdots \stackrel{(B4)}{\rightarrow} & \cdots l_k \sigma_{k-1} r_k l_k \sigma_{k-2} \cdots \rightarrow \cdots l_k \sigma_{k-2} \cdots \sigma_{k-s} \sigma_{k-s} \sigma_{k+1} \sigma_{k+2} \cdots \sigma_{k+t} Y' \\
\rightarrow & \cdots \sigma_{k-2} \cdots \sigma_{k-s} \sigma_{k-s} l_k\sigma_{k+1} \sigma_{k+2} \cdots \sigma_{k+t} Y'
\end{align*}
and then apply induction on $n$.

If $x=\sigma_{k-1}$ is not a switch, then it is a departure.  Then, we apply (B3) to extend $\mathcal{C}$ during the sequence 
\begin{align*}
 X l_k \sigma_{k-1} \sigma_{k-2} \cdots \sigma_{k-s}  \sigma_{k+1} \sigma_{k+2} \cdots \sigma_{k+t}(\sigma_{k-s} Y') \stackrel{(B3)}{\rightarrow} & X l_{k-1} \sigma_k \sigma_{k-2} \cdots \sigma_{k-s} \sigma_{k-s} \sigma_{k+1} \sigma_{k+2} \cdots \sigma_{k+t} Y'  \\
 \rightarrow & X l_{k-1} \sigma_{k-2} \cdots \sigma_{k-s} \sigma_{k-s} \sigma_k \sigma_{k+1} \sigma_{k+2} \cdots \sigma_{k+t} Y'.
\end{align*}
We repeat this argument until we either find a switch or arrive at a word of the form
\[
X l_{k-s+1}\sigma_{k-s} \sigma_{k-s} \sigma_{k-s+2} \cdots \sigma_{k+t} Y'
\]
If the first $\sigma_{k-s}$ is a switch, then the usual combination of Pinch Move and Type I Reidemeister Move produces
\[
X l_{k-s+1} \sigma_{k-s} \sigma_{k-s+2} \cdots \sigma_{k+t} Y'
\]
which has the form (\ref{eq:Yform}) with $|Y'|= n-1$ so that induction applies.
If instead $\sigma_{k-s}$ is a departure, then we can extend $\mathcal{C}$ over another Cusp Tangency Move and then apply a Reidemeister II move:
\begin{align*}
X l_{k-s+1}\sigma_{k-s} \sigma_{k-s} \sigma_{k-s+2} \cdots \sigma_{k+t} Y' \stackrel{(B3)}{\rightarrow} & X l_{k-s}\sigma_{k-s+1} \sigma_{k-s} \sigma_{k-s+2} \cdots \sigma_{k+t} Y' \\
\rightarrow & X l_{k-s+1} \sigma_{k-s+2} \cdots \sigma_{k+t} Y'. 
\end{align*}
Again, we are able to use the inductive hypothesis on $n$.

\medskip

\noindent {\bf Case 9:} $z = \sigma_j$ with $k-s < j < k$.  Then,  a Type III Reidemeister move can be applied: 
\begin{align*}
X l_k \sigma_{k-1} \sigma_{k-2} \cdots \sigma_{k-s} \sigma_{k+1} \sigma_{k+2} \cdots \sigma_{k+t} (\sigma_{j}Y') \, \rightarrow  & \, (X \sigma_{j-1}) l_k \sigma_{k-1} \sigma_{k-2} \cdots \sigma_{k-s} \sigma_{k+1} \sigma_{k+2} \cdots \sigma_{k+t} Y'.
\end{align*}
We apply (B2) and then the inductive hypothesis on $n$ to produce $(\Sigma, \mathcal{C})$.

\medskip

\noindent {\bf Case 10:}  $z = \sigma_{k}$. 
\begin{itemize}
\item[Subcase A:]  $s=t=0$.  This case cannot occur since the appearance of the product $l_k\sigma_k$ would show that $\Lambda$ is stabilized. 
\item[Subcase B:]  Exactly one of $s>0$ or $t>0$.  Then, a Type II Move allows us to apply induction on $n$. For instance, when $s>0$ and $t=0$,
\begin{align*}
X l_k\sigma_{k-1} \sigma_{k-2} \cdots \sigma_{k-s}(\sigma_k Y') \rightarrow & X l_k\sigma_{k-1} \sigma_k 
\sigma_{k-2} \cdots \sigma_{k-s} Y' \\
\rightarrow & X l_{k-1} \sigma_{k-2} \cdots \sigma_{k-s} Y'.
\end{align*}  
\item[Subcase C:]  Both $s>0$ and $t>0$.  We have
\[
X l_k\sigma_{k-1} \cdots \sigma_{k-s} \sigma_{k+1} \cdots \sigma_{k+t}(\sigma_k Y') \rightarrow  X l_k\sigma_{k-1} \sigma_{k+1}\sigma_k (  
\sigma_{k-2} \cdots \sigma_{k-s})(\sigma_{k+2} \cdots \sigma_{k+t}) Y'. 
\]
If the $\sigma_{k-1}$ is a switch, then we apply (B4) to do a Pinch Move followed by a Type I and Type II Move:
\begin{align*}
X l_k\sigma_{k-1} \sigma_{k+1}\sigma_k \cdots \stackrel{(B4)}{\rightarrow} & X l_k \sigma_{k-1}r_kl_k \sigma_{k+1} \sigma_k \cdots \\
 \rightarrow & X l_k \sigma_{k+1} \sigma_k \cdots   \\
\rightarrow & X l_{k+1} (  
\sigma_{k-2} \cdots \sigma_{k-s})(\sigma_{k+2} \cdots \sigma_{k+t}) Y'   \\
\rightarrow & (X \sigma_{k-2} \cdots \sigma_{k-s})(l_{k+1} \sigma_{k+2} \cdots \sigma_{k+t}) Y'. 
\end{align*}
Finally, if $\sigma_{k-1}$ is not a switch, we use (B3) and apply a Cusp Tangency Move followed by a Type III Move:
\begin{align*}
X l_k\sigma_{k-1} \sigma_{k+1}\sigma_k \cdots \stackrel{(B3)}{\rightarrow} &  X l_{k-1} \sigma_k \sigma_{k+1} \sigma_k \cdots \\
 \rightarrow &  X l_{k-1} \sigma_{k+1} \sigma_{k} \sigma_{k+1} (  
\sigma_{k-2} \cdots \sigma_{k-s})(\sigma_{k+2} \cdots \sigma_{k+t}) Y'  \\
\rightarrow &   (X \sigma_{k+1}) l_{k-1} (  
\sigma_{k-2} \cdots \sigma_{k-s}) (\sigma_{k} \sigma_{k+1} \sigma_{k+2} \cdots \sigma_{k+t}) Y'.
 \end{align*}
Then, induction on $n$ applies.
\end{itemize}

This completes the proof when $\rho \neq 1$.  The only place where the hypothesis $\rho \neq 1$ was used in the above induction was Subcase A of Case 5 where (B5) was applied to remove a Legendrian unknot component.  When $\rho =1$, instead of applying (B5) for this subcase, we can simply apply a Legendrian isotopy to move the unknot component to the left of the rest of the front.  In this manner, the inductive argument produces a cobordism $\Sigma$ from a disjoint union of standard Legendrian unknots (each with 2 cusps and no crossings) to $\Lambda$ together with an MCF $\mathcal{C}$ on $\Sigma$ extending $\mathcal{C}_\Lambda$.  \dr{Finally, we can apply the (B6) to use the ($\leftarrow$) direction of the Pinch Move to join all of the unknots into a single unknot via a cobordism.} 
%The MCF $\mathcal{C}$ extends in a clear way over this cobordism;  \dr{see Remark \ref{rem:extramoves}.}\footnote{\dr{Or, maybe we should make a (B6)...}}

\end{proof}

%\input{Sec8}

%More sections added as needed 

\bibliographystyle{abbrv}
\bibliography{PR}

\end{document}